\documentclass[a4paper,11pt]{article}
\usepackage[utf8]{inputenc}
\usepackage{amsmath}
\usepackage{latexsym}
\usepackage{amssymb,amsfonts,amsmath,mathrsfs}
\usepackage{verbatim}
\usepackage{amsthm}
\usepackage{amsfonts}
\usepackage{amssymb}
\usepackage{amstext}
\usepackage{dsfont}
\usepackage{graphicx}
\usepackage{color}
\usepackage[colorlinks]{hyperref}
\usepackage{epigraph}
\usepackage[left=3cm,top=2cm,bottom=3.5cm,right=2.5cm,nohead,nofoot]{geometry}
\usepackage{tikz}
\usetikzlibrary{arrows,shapes,matrix}

\graphicspath{{Figures/}{.}}%

\title{On the extreme values of the Riemann zeta function on random intervals of the critical line}
\author{
        Joseph \textsc{Najnudel}     \footnote{\texttt{joseph.najnudel@uc.edu}}
       } 

\allowdisplaybreaks[4]

\def\Oc{{\mathcal O}}

\newtheorem{thm}{Theorem}[section]
\newtheorem{proposition}[thm]{Proposition}
\newtheorem{corollary}[thm]{Corollary}

\newtheorem{conjecture}[thm]{Conjecture}

\newtheorem{lemma}[thm]{Lemma}

\newtheorem{rmk}[thm]{Remark}

\numberwithin{equation}{section}

\begin{document}

\maketitle
\begin{abstract}
In the present paper, we show that under the Riemann hypothesis, and for fixed $h, \epsilon > 0$, the  supremum of the real and the imaginary parts of $\log \zeta (1/2 + it)$ for $t \in [UT -h, UT + h]$ are in the interval $[(1-\epsilon) \log \log T, 
(1+ \epsilon) \log \log T]$ with probability tending to $1$ when $T$ goes to infinity, $U$ being a uniform random variable in $[0,1]$. This proves a weak version of a conjecture by Fyodorov, Hiary and Keating, which has recently been intensively studied in the setting of random matrices. We also 
 unconditionally show that the supremum of $\Re \log \zeta(1/2 + it)$ is  
at most $\log \log T + g(T)$ with probability tending to $1$, $g$ being any function tending to infinity at infinity. 

\end{abstract}

\indent
\hrule
\tableofcontents
\indent
\hrule

\section{Introduction}
\label{section:introduction}
In relation with the Riemann hypothesis, the behavior of the function $\zeta$ at or near the critical line has been extensively studied, in particular its maximal order of magnitude. For example, a consequence of the Riemann hypothesis is 
the Lindel\"of hypothesis, which claims that for all $\alpha > 0$, $|\zeta(1/2 + it)| = \mathcal{O}(1 + |t|^{\alpha})$: in fact, this result can be improved to 
$|\zeta(1/2 + it)|  \ll \exp (\mathcal{O} (\log t / \log \log t))$ for $t > 3$  (see Titchmarsh \cite{Tit}, Theorem 14.14 (A)). If one does not assume the Riemann hypothesis, the best known result in this direction, due to Bourgain \cite{Bourgain}, is that the estimate is true for all $\alpha > 13/84$. 
On the other hand, it is also known (see Titchmarsh \cite{Tit}, p. 209) that for all $c <3/4$, $|\zeta(1/2 + it)| \geq \exp( c \sqrt{ \log t/\log \log t})$ for 
arbitrarily large values of $t$: this estimate has recently been improved to 
$ \exp( c \sqrt{ \log t \log \log \log t /\log \log t})$ for all $c < 1/\sqrt{2}$, in a paper by Bondarenko and Seip \cite{BS17}.

It has been conjectured by Farmer, Gonek, and Hughes \cite{FGH}, that the supremum of  $|\zeta(1/2 + it)| $ on $[0,T]$ is 
$\exp ( (1+o(1)) \sqrt{ (1/2)\log T \log \log T})$ for $T \rightarrow \infty$: their heuristics is related to the fact that the large values of $|\zeta|$ are expected
to be roughly independent when they occur at points which are sufficiently far from each other.

A classical result by Selberg \cite{Selberg} shows that for $t$ random, uniform on $[0,T]$, $\log |\zeta (1/2 + it)| /\sqrt{\log \log T}$ converges in law towards a  
centered Gaussian random variable of variance $1/2$, however, it does not give information about the dependence between the values of 
$\log |\zeta|$ at different points. It has been proven by Hughes, Nikeghbali and Yor \cite{HNY} that, in a sense which is made precise in their article, the values of 
 $\log |\zeta|$ are independent if they are taken at points which are very distant from each other, and on the other hand, Bourgade \cite{Bourgade} has proven that the values 
are correlated if they are taken at points which are close enough to each other. For example, for $t$ uniform on $[0,T]$, and $a \in [0,1]$, we have that $(\log |\zeta (1/2 + it)| /\sqrt{\log \log T},  \log |\zeta (1/2 + i(t + (\log T)^{-a}))| /\sqrt{\log \log T})$ tends to a two-dimensional centered Gaussian vector, whose correlation is 
equal to $a$. Moreover, in Section 4 of \cite{Bourgade}, the author explicitly makes the connection between this correlation structure and some branching processes. 

 In \cite{bib:FyKe} and \cite{bib:FHK}, Fyodorov, Hiary and Keating make heuristic computations suggesting that $\log |\zeta|$ behaves like a centered Gaussian field, whose correlation between random points in $[0,T]$ at distance $u \leq 1$ is given by $\min( |\log u|, \log  \log T)$. 
A more sophisticated random field is constructed by Saksman and Webb \cite{SaksWebb}, as a limiting random distribution for the function $\zeta$ itself. 

From the  comparison between $\log |\zeta|$ and the logarithm of the characteristic polynomial of the Circular Unitary Ensemble, which can both be considered as approximations of Gaussian log-correlated fields, and from some moment computations coming from techniques used in statistical mechanics, Fyodorov, Hiary and Keating have formulated a conjecture concerning the supremum of $\log |\zeta|$ on  random intervals of fixed length, which can be stated as follows: 
\begin{conjecture}
\label{conjecturemax1} For $U$ uniform on $[0,1]$ and $h > 0$ fixed, the family of random variables:
\begin{equation}
\left(\sup_{t \in [UT - h, UT + h]} \log \left|  \zeta \left( \frac{1}{2} + i t \right)  \right|  - 
\left(\log \log T - \frac{3}{4} \log \log \log T \right) \right)_{T > 3} \label{qwertyuiop0}
\end{equation}
tends to a limiting distribution when $T$ goes to infinity. 
\end{conjecture}

The study of the maximum of log-correlated Gaussian fields has been done in several different and fairly general settings.
   In the case of  branching Brownian motion and branching random walks,
 for which a tree structure is explicit, the study goes back to the seminal paper by Bramson \cite{Bra78}, and
more precise and  general  results have later been proven by A\"id\'ekon, Hu and Shi (see \cite{HSh09, AShi10, Aid11}). In the  case of  log-correlated Gaussian fields on $[0,1]^d$ ($d \geq 1$) a limit theorem has been 
proven by  Madaule  \cite{bib:Mad13}, and then generalized by Ding, Roy and Zeitouni \cite{RDZ15}. The discrete Gaussian Free Fields have been studied by 
Bramson, Ding and Zeitouni  in \cite{BZe10} and \cite{BDZ13}.

An analogy between the Riemann zeta function and the characteristic polynomial of random matrices has been developed in the last decades, following the 
idea by P\'olya and Hilbert that there may be a way to see the  non-trivial zeros of $t \mapsto \zeta(1/2 + it)$ as the eigenvalues of some Hermitian operator.  
In relation with this analogy, we can mention the results by Montgomery \cite{Montgomery}  on the pair correlation of zeros of $\zeta$, which have been observed by Dyson to be similar to 
what we obtain for eigenvalues of random Hermitian or unitary matrices, the conjecture by Keating and Snaith \cite{KSn00} on the moments of $\zeta$ on the critical line, and the limit theorems by Katz and Sarnak \cite{KatzSarnak} on analogs of the Riemann zeta function, constructed from algebraic curves on function 
fields.

In relation with this analogy, Fyodorov, Hiary and Keating have stated a version of Conjecture \ref{conjecturemax1} in the setting of random matrix theory, which says that
\begin{equation}
\left(\sup_{|z| = 1} \log \left| X_n(z)  \right|  - 
\left(\log n - \frac{3}{4} \log \log n \right) \right)_{n \geq 2} \label{qwertyuiop}
\end{equation} 
tends to a limiting distribution (namely, the law of the average of two independent Gumbel variables), if $X_n$ is the characteristic polynomial of a Haar-distributed, $n \times n$, random unitary matrix.

This conjecture is not fully proven, but Chhaibi, Madaule and Najnudel in \cite{CMN} have recently shown that the family \eqref{qwertyuiop} of random variables 
is tight, improving successive results by Arguin, Belius and Bourgade \cite{bib:ABB} and by Paquette and Zeitouni \cite{bib:PZ}. 

Proving Conjecture \ref{conjecturemax1}, or even only the tightness of \eqref{qwertyuiop0},  seems to be much more difficult than the corresponding results 
in the setting of random matrices. One of the reasons is that there is a priori less randomness in the setting of the Riemann zeta function. Indeed,  in \cite{ABH15}, Arguin, Belius and Harper consider a randomized version of the Riemann zeta function and  show a weaker version of an analog of Conjecture
\ref{conjecturemax1}, giving the two leading order terms $\log \log T - \frac{3}{4} \log \log \log T$. 
The main goal of the present paper is to show that one can progress towards Conjecture \ref{conjecturemax1} in the original setting of the Riemann 
zeta function, without extra randomness. More precisely, we will prove that under the Riemann hypothesis, the leading order term in Conjecture 
\ref{conjecturemax1} is the correct one, which correspond to the precision of the result by Arguin, Belius and Bourgade \cite{bib:ABB}  for random matrices. 

In all this paper, $\log \zeta$ will be defined as the unique version of the logarithm of zeta which is real on $(1, \infty)$, well-defined and continuous everywhere, 
except on the closed half-lines at the left of the zeros and the pole of $\zeta$. The values for which $\log \zeta$ is not well-defined will be considered to be 
implicitly excluded from all infima and suprema where they are involved. Note that despite this exclusion, we still have, for all $a < b$, 
$$\sup_{t \in [a,b]} \Re \log \zeta \left(\frac{1}{2} + it \right) = 
\sup_{t \in [a,b]}  \log \left| \zeta \left(\frac{1}{2} + it \right) \right|,$$
since $\log |\zeta|$ is continuous on the critical line except at the zeros of $\zeta$, where it is equal to $-\infty$ and then is not involved in the supremum in the right-hand side. 
With this convention, our main result is the following: 
\begin{thm} \label{main}
Let us assume the Riemann hypothesis. Then, if $h > 0$, $\epsilon > 0$ are fixed, if $T > 3$ and if $U$ is a uniform random variable in $[0,1]$, then 
$$\mathbb{P} \left( \sup_{t \in [UT -h, UT + h]} \Re \log \zeta \left(\frac{1}{2} + it \right)  \in [(1-\epsilon) \log \log T, (1+\epsilon) \log \log T] \right)
\underset{T \rightarrow \infty}{\longrightarrow} 1,$$
$$\mathbb{P} \left( \sup_{t \in [UT -h, UT + h]}  \Im \log \zeta \left(\frac{1}{2} + it \right)  \in [(1-\epsilon) \log \log T, (1+\epsilon) \log \log T] \right)
\underset{T \rightarrow \infty}{\longrightarrow} 1,$$
$$\mathbb{P} \left( \inf_{t \in [UT -h, UT + h]}  \Im \log \zeta \left(\frac{1}{2} + it \right)  \in [-(1+\epsilon) \log \log T,- (1-\epsilon) \log \log T] \right)
\underset{T \rightarrow \infty}{\longrightarrow} 1.$$
\end{thm} 
In the proof of this theorem, the Riemann hypothesis is in fact used only once, namely in Proposition \ref{sumsprimes}, where we approximate an averaged version of $\log \zeta$ by a sum indexed by prime powers. If we do not assume the Riemann hypothesis, we get some extra error terms which seem difficult to estimate in general. 

However, we have an unconditional and 
stronger result for the upper bound of the real part of $\log \zeta$: 
\begin{thm} \label{unconditional} 
Let $h > 0$, and let $g$ be a function from $[3,\infty)$ to $\mathbb{R}_+$, tending to infinity at infinity. Then, for $T > 3$ and $U$ uniform on $[0,1]$, 
$$ \mathbb{P} \left( \sup_{t \in [UT -h, UT + h]} \Re \log \zeta \left(\frac{1}{2} + it \right) \leq  \log \log T + g(T) \right)
\underset{T \rightarrow \infty}{\longrightarrow} 1.$$
\end{thm} 
The part of Theorem \ref{main} concerning $\Im \log \zeta$ gives some information on the fluctuations of the distribution of the zeros of $\zeta$.
From Titchmarsh \cite{Tit}, Theorem 9.3., we deduce the following: 
\begin{corollary}
Assume the Riemann hypothesis. For $t > 0$, let $N(t)$ be the number of non-trivial zeros of $\zeta$ whose imaginary part is in $[0,t]$, and let 
$$\Delta(t) := N(t) - \frac{t \log t}{2 \pi} + \frac{t ( 1 + \log (2 \pi))}{2 \pi}.$$
 Then, if $h > 0$, $\epsilon > 0$ are fixed, if $T > 3$ and if $U$ is a uniform random variable in $[0,1]$, 
$$\mathbb{P} \left( \sup_{t \in [\max(UT -h,0), UT + h]} \Delta(t)  \in \left[\frac{1-\epsilon}{\pi} \log \log T, \frac{1+\epsilon}{\pi}\log \log T \right] \right)
\underset{T \rightarrow \infty}{\longrightarrow} 1,$$
$$\mathbb{P} \left(  \inf_{t \in [\max(UT -h,0), UT + h]} \Delta(t)   \in \left[-\frac{1+\epsilon}{\pi}\log \log T,- \frac{1-\epsilon}{\pi} \log \log T \right] \right)
\underset{T \rightarrow \infty}{\longrightarrow} 1.$$
\end{corollary} 
As mentioned before, the accuracy of our results is the same as in the paper by Arguin, Belius and Bourgade \cite{bib:ABB} on the 
maximum of the characteristic polynomial of the Circular Unitary Ensemble. In the random matrix setting, 
Paquette and Zeitouni \cite{bib:PZ} and Chhaibi, Madaule and Najnudel \cite{CMN} have successively improved this accuracy. Similar improvements seem to be very difficult to obtain in the setting of the Riemann zeta function.

The proof of our main theorem is divided into several parts.

 The upper bound for the real part is covered by Theorem \ref{unconditional}, which is proven by showing that the supremum of $|\zeta|$ in the segment  $[1/2 + i (UT -h), 1/2 + i( UT + h)]$ is well-controlled by the values of $|\zeta|$ at  about $\log T$ points of the segment. Then, we conclude by using classical estimates of the second moment of $|\zeta|$ on the critical line. 

For the upper bound of the imaginary part, we  prove that the supremum of $\Im \log \zeta$ on the same segment is controlled by the supremum of 
some averages of $\Im \log \zeta$  around about $\log T$ points. Then, we show that these averages are close to finite sums indexed by prime numbers, for which 
we give suitable bounds on the tail of their distribution. 

For the lower bound, we use the fact that averaging $\Re \log \zeta$ or $\Im \log \zeta$ essentially decreases its supremum, up to some error terms 
which can be controlled. We then deduce that it is sufficient to get lower bounds on some sums indexed by primes which are explicitly given. These sums are proven to be sufficiently close, in the sense of their Fourier transform, to Gaussian variables with the same covariance structure. After cutting the sum  into smaller pieces in order to get some approximate branching structure, we obtain a way to apply the second moment method, as for the lower bound on branching random walks. The general principle of this method can be found in the lecture notes by Kistler \cite{bib:Kistler}, and its spirit goes back to 
the work by Bolthausen, Deuschel and Giacomin \cite{BDG}, in their study of the extreme values of the two-dimensional Gaussian free field. 

The sequel of the paper is structured as follows.  In Section \ref{section:unconditional}, we give a proof of Theorem \ref{unconditional}. 
In Section \ref{section:average}, we show that if we average the logarithm of the Riemann zeta function around points of the critical line in a suitable way, then we get 
something close to a finite sum indexed by prime numbers which is explicitly given. In Section \ref{section:upperboundimaginary}, we prove the upper bound part of Theorem \ref{main}. In Section \ref{section:sum}, we use the results in Section \ref{section:average} in order to bound from below the supremum of $\Re \log \zeta$ or
$\Im \log \zeta$ in terms on a sum on prime numbers which is more tractable then the  one given in Section \ref{section:average}. 
  In Section \ref{section:gaussian},  we cut the sum obtained in Section \ref{section:sum} into smaller pieces and show that their joint law is close, in a sense which can be made precise, to a Gaussian family of variables. In Section \ref{section:lowerbound}, we finally use the second moment method in order to prove the lower bound part of Theorem \ref{main}. 

The results proven in the Sections  \ref{section:unconditional},  \ref{section:average}, \ref{section:upperboundimaginary} and  \ref{section:sum} are specific to the setting of the Riemann zeta function: most of the ideas given in these sections are new, even if we also use very classical tools from analytic number theory. Our proof of Theorem \ref{unconditional} is obtained by showing that the
maximum of $|\zeta|$ on $[UT-h, UT + h]$ can be essentially controlled by the maximum on $\mathcal{O}_h (\log T)$ points 
of this interval: this fact can be related to a result proven in \cite{CMN}, which states that the maximal modulus of a polynomial of degree $n$ on the unit circle is uniformly dominated by its maximal modulus on the $2n$-th roots of unity. 

The general method used in  Sections \ref{section:gaussian} and \ref{section:lowerbound} in order to prove the lower bound of the maximum is very classical and its main steps can be found in several papers on the maximum of log-correlated fields, including the paper by Arguin, Belius and Bourgade \cite{bib:ABB} on the leading order term of the maximum of the characteristic polynomial of the CUE on the unit circle. 
\subsection*{Remark} 
After the first submission of the present paper to arXiv.org, our result concerning the real part of zeta has been proven by Arguin, Belius, Bourgade, Radziwi\l\l  \, and Soundararajan \cite{ABBRS}, without assuming the Riemann hypothesis. In particular, they have also proven Theorem \ref{unconditional} (with a shorter proof than ours). 
\subsection*{Questions} 
Some questions can be asked about how the results of this paper can be improved or generalized:
\begin{itemize}
\item   It is natural to try to remove the Riemann hypothesis also in the case of the imaginary part of $\log \zeta$. We have a priori no idea about how the method 
used in \cite{ABBRS} can be adapted to this case, since the value of $\Im \log \zeta$ cannot be directly related to the value of $\zeta$ at the same point. It may also be possible to estimate the extra error terms  in the formula of Proposition \ref{sumsprimes}, which arise when we take into account the zeros of $\zeta$ at the right of the critical line, but our attempts to get good bounds for these extra terms have not succeeded until now.
\item In the statement of the main result, we consider a uniform variable on $[0,1]$. This particular choice of a probability distribution is somehow arbitrary. However, it is an exercise to deduce, from Theorem \ref{main}, the same result for 
any random variable $U$ whose distribution is absolutely continuous with respect to the Lebesgue measure. 
\item Similar problems can be considered for the maximum of more general $L$ functions. In the case of Dirichlet $L$ functions, we expect that a similar result holds, with almost the same proof as for $\zeta$, even if we did not check all the details. 
\item In Theorem \ref{main}, the length of the random interval we consider is constant. One can wonder what happens if we change this interval to an interval whose length varies with $T$. It is reasonable to expect that our method can be used  to get lower and upper bounds for the extreme values of $\log \zeta$: however, there is no reason a priori to think that the leading order of the lower and the upper bounds always match. It can be interesting to see more precisely what happens.
\item One can of course also wonder if it is possible to improve the accuracy of our result. In our proof of the lower bound, we have to study truncated sums on primes as $S_0(k,m)$, and show, by estimating their exponential moments,  that they can be compared with Gaussian variables. Since 
the large values of the sums on primes are of order $\log \log T$, it is not possible to do the truncation much below $\log \log T$, and then, in order to estimate the exponential moments, we need to control moments of sums like $S(k,m)$, up to an exponent of order $\log \log T$. On the other hand, with our method, such moments can be controlled only if all the products of primes which are involved are much smaller than $T$: we can see that, for example, by looking at the proof of Lemma \ref{lemmaSV}. 
This forces to take the averaging scale $H$ of $\log \zeta$ below $\log T/\log \log T$, whereas the natural  speed at which $\log \zeta$ moves on the critical line is of order $\log T$. This loss of a factor $\log \log T$ gives a loss of a term of order $\log \log \log T$ in  the maximum. Hence, we expect that our method cannot give the second order term 
with the exact coefficient $-3/4$, as in the paper by Paquette and Zeitouni \cite{bib:PZ}. However, we do not exclude that more careful estimates and finer cuts of the sums on primes (by letting the parameter $K$ slowly grow to infinity) can give 
lower bounds of the form $\log \log T - A \log \log T$ for some constant $A > 3/4$. 
\end{itemize}

\subsection*{Acknowledgements}
The author would like to thank R. Chhaibi for  helpful discussions we had on the problem solved in the present article, and the referees for their comments and suggestions, which have greatly improved the writing of this paper. One of the referees suggested  some of  the questions stated above. 
\subsection*{Notation} 
In the present paper, we use the Vinogradov notation
$$ f \ll g \Leftrightarrow f = \Oc(g) \ , $$
and when the implicit constant depend on parameters (e.g, on a positive real $h$ and a function $\varphi$), this dependence will be indicated thanks to subscripts (e.g $\ll_{\varphi,h}$ or $\mathcal{O}_{\varphi, h}$). The Fourier transform is normalized as follows: 
$$ \widehat{\varphi}(\lambda)= 
\int_{-\infty}^{\infty} e^{-i \lambda x} \varphi(x) dx.$$
Finally, $\mathcal{P} = \{2,3,5,7,\dots\}$ denotes the set of prime numbers. 

\section{Proof of Theorem \ref{unconditional}}
\label{section:unconditional}
In \cite{CMN}, in order to get an upper bound for the maximal modulus of the characteristic polynomial, we show that the maximal modulus, on the unit circle, of any polynomial of 
degree $n$, is at most $14$ times the maximal modulus on the $2n$-th roots of unity ($14$ is not optimal). In this section, we show a quite similar result, 
except that the unit circle is replaced by the real line, and the polynomials of a given degree are replaced by linear combinations of complex exponentials whose frequencies are
supported by a given compact set. The following lemma shows that such functions are uniquely determined and well-controlled by their values on a discrete set of points:
\begin{lemma}
There exists a universal function $\varphi$ from $\mathbb{R}$ to $\mathbb{R}$, continuous and decaying faster than any power at infinity, satisfying the following property: if $\lambda > 0$ and if $f$ be a finite linear combination of 
functions of the form $x \mapsto e^{i \mu x}$ with $|\mu| \leq \lambda$, then for all $x \in \mathbb{R}$, 
$$f(x) = \sum_{k \in \mathbb{Z}} 
\varphi((2\lambda x/\pi) - k ) f(k\pi/2\lambda).$$
\end{lemma}
\begin{proof}
We take for $\varphi$ a function whose Fourier transform is real and even (then $\varphi$ is real-valued),  equal to $1$ in a neighborhood of $[-\pi/2,\pi/2]$ and to $0$ in a neighborhood of $\mathbb{R} \backslash (-\pi, \pi)$, smooth (hence $\varphi$ is rapidly decaying).

Once $\varphi$ is fixed, it is enough, by linearity, to check the following equality:
$$e^{i \mu x} =  \sum_{k \in \mathbb{Z}} 
\varphi((2\lambda x/\pi) - k )e^{i \mu k\pi/2\lambda}$$
for all $\mu \in [-\lambda, \lambda]$, or equivalently, 
$$ e^{i \mu x} =  \sum_{k \in \mathbb{Z}}  \psi(k),$$
where
$$\psi(y) = \varphi((2\lambda x/\pi) - y )e^{i \mu y\pi/2\lambda},$$
and then 
$$\widehat{\psi} (\theta) = e^{ix(  \mu - 2 \lambda \theta /\pi)} \widehat{\varphi} (- \theta + \mu \pi/2 \lambda). $$
Since $\psi$ is smooth and rapidly decaying at infinity, we can apply Poisson summation formula: 
$$  \sum_{k \in \mathbb{Z}}  \psi(k) = 
 \sum_{k \in \mathbb{Z}}  \widehat{\psi} (2 \pi k).$$
Now, 
$$\widehat{\psi} (0) =  e^{ix \mu}  \widehat{\varphi} ( \mu \pi/2 \lambda) =   e^{ix \mu}$$
since $| \mu \pi/2 \lambda | \leq \pi/2$, and for $k \neq 0$,
$$ \widehat{\psi} (2 \pi k) = e^{ix(  \mu - 4 \lambda k)} \widehat{\varphi} (- 2 \pi k + \mu \pi/2 \lambda) = 0$$
since 
$$ | - 2 \pi k + \mu \pi/2 \lambda| \geq 2 \pi - \pi/2 > \pi.$$
 \end{proof}
\begin{rmk}
This proposition can easily be extended to all bounded, continuous functions whose Fourier transform, in the sense of the distributions, is supported in $[-\lambda, \lambda]$. 
\end{rmk}
 A consequence of the lemma is the following: 
 \begin{proposition} \label{Fourier}
 For all $A > 1$, for any function $f$ satisfying the assumptions of the previous lemma, and  for all $x_0 \in \mathbb{R}$, $h, R > 0$,  
$$\sup_{x \in [x_0 - h, x_0 + h]} 
|f(x)|
\ll_A \left( \sum_{k \in \mathbb{Z}, |k| \leq  \lambda h }
|f( x_0 + (k \pi/2 \lambda))|
 + \sum_{k \in \mathbb{Z}, |k| \leq  R } \frac{|f( x_0 + (k \pi/2 \lambda))|}{1 + |k|^A}
 + \frac{\sup_{\mathbb{R}}|f|}{1 + 
 R^{A-1}} \right).$$
 \end{proposition}
  \begin{proof}
  By translating $f$, we can assume $x_0 = 0$, moreover, it is sufficient to show the result corresponding to $R \rightarrow \infty$, i.e.
  $$\sup_{x \in [- h, h]} 
|f(x)|
\ll_A \left( \sum_{k \in \mathbb{Z}, |k| \leq  \lambda h }
|f( k \pi/2 \lambda)|
 + \sum_{k \in \mathbb{Z}  } \frac{|f( k \pi/2 \lambda))|}{1 + |k|^A} \right).$$
 Now, by the lemma, 
 $$\sup_{x \in [- h, h]} 
|f(x)| \leq \sum_{k \in \mathbb{Z}}
|f(k \pi/2 \lambda)| 
\, \sup_{y \in [ - (2 \lambda h /\pi)
 -k, (2 \lambda h /\pi)
 -k]} |\varphi(y)|.$$
 Now, by the assumption on $\varphi$ made in the lemma, there exists $K_A > 0$ such that 
 $\varphi(y) \leq K_A/(1+ |y|^A)$
 for all $y \in \mathbb{R}$. 
 Hence, for $|k| \leq \lambda h$, 
 $$\sup_{y \in [ - (2 \lambda h /\pi)
 -k, (2 \lambda h /\pi)
 -k]} |\varphi(y)| \leq K_A
 \leq K_A \left( 1 + \frac{1}{1+ |k|^A} \right),$$
and, for $|k| \geq \lambda h$, 
$$|k| - \frac{2 \lambda h}{\pi} 
\geq |k| - \frac{2 |k|}{\pi} 
\geq |k|/3,$$
and then 
$$\sup_{y \in [ - (2 \lambda h /\pi)
 -k, (2 \lambda h /\pi)
 -k]} |\varphi(y)| \leq 
 \sup_{|y| \geq |k|/3} 
  |\varphi(y)| \leq 
  \frac{K_A}{ 1 + (|k|/3)^A}
  \leq \frac{3^A K_A}{1 + |k|^A}.$$
 This gives the desired result. 
  \end{proof}
The next step of our proof of Theorem \ref{unconditional} is to show that locally on the critical line, the Riemann zeta function is not far from being a linear combination of complex exponentials. 
\begin{proposition}
Let $Z$ be the function from $\mathbb{R}$ to $\mathbb{R}$ given by 
$$Z(t) = \zeta(1/2 + it) e^{i \theta(t)},$$
for
$$\theta(t) = \Im \log \Gamma (1/4 + it/2) - \frac{t}{2} \log  \pi,$$
where we take the continuous version of $\Im \log \Gamma (1/4 + it/2)$  vanishing at zero.
Then for $t_0 \geq 2$, and uniformly on $t \in [t_0 - t_0^{1/4}, t_0 + t_0^{1/4}]$, we have 
$$Z(t) = 2 \sum_{k = 1}^{\lfloor \sqrt{t_0/2\pi} \rfloor} \frac{\cos 
\left( \frac{t}{2}   \log \left( \frac{t_0}{2 \pi k^2} \right)
- \frac{t_0}{2} - \frac{\pi}{8} \right)}{\sqrt{k}} + \mathcal{O}(t_0^{-1/4}).$$
\end{proposition}
\begin{proof}
The Riemann-Siegel formula gives (see \cite{Tit}, p. 89):
$$Z(t) = 2 \sum_{k=1}^{\lfloor\sqrt{t/2\pi} \rfloor} \frac{\cos (\theta (t) - t \log k)}{\sqrt{k}} + \mathcal{O} (t^{-1/4})$$
where complex Stirling formula gives  the expansion: 
$$\theta(t) =  \frac{t}{2} \log (t /2 \pi) - \frac{t}{2} - \frac{\pi}{8} + \mathcal{O}(1/t).$$
Hence, 
\begin{align*}
\theta(t) & = 
\frac{t}{2} \left( \log (t_0/2\pi) 
+ \frac{t - t_0}{t_0} 
+ \mathcal{O} \left(\frac{(t-t_0)^2}{t_0^2}\right) \right)
- \frac{t_0}{2} - \frac{t-t_0}{2}
- \frac{\pi}{8} + \mathcal{O}(1/t_0)
\\ & = \frac{t}{2} \log(t_0/2\pi)
- \frac{t_0}{2} - \frac{\pi}{8} 
+ \mathcal{O} \left( \frac{1 + 
(t-t_0)^2}{t_0} \right)
\\ & = \frac{t}{2} \log(t_0/2\pi)
- \frac{t_0}{2} - \frac{\pi}{8} 
+ \mathcal{O}(t_0^{-1/2}).
\end{align*}
We deduce 
\begin{align*}
Z(t) & = 2 \sum_{k = 1}^{\lfloor \sqrt{t/2\pi} \rfloor} \frac{\cos 
\left( \frac{t}{2}   \log \left( \frac{t_0}{2 \pi k^2} \right)
- \frac{t_0}{2} - \frac{\pi}{8} \right)}{\sqrt{k}} + 
\mathcal{O} \left( t_0^{-1/2} \sum_{k=1}^{\lfloor \sqrt{t/2\pi} \rfloor} k^{-1/2} \right) + 
\mathcal{O}(t_0^{-1/4})
\\ & =  2 \sum_{k = 1}^{\lfloor \sqrt{t/2\pi} \rfloor} \frac{\cos 
\left( \frac{t}{2}   \log \left( \frac{t_0}{2 \pi k^2} \right)
- \frac{t_0}{2} - \frac{\pi}{8} \right)}{\sqrt{k}}  + 
\mathcal{O}(t_0^{-1/4})
\end{align*}
This expression differs from the expression of the proposition by a bounded number of terms, since
$$\left|\sqrt{\frac{t}{2 \pi}} - 
\sqrt{\frac{t_0}{2 \pi}}\right| 
= \frac{|t - t_0|}{\sqrt{2 \pi t} +  \sqrt{2 \pi t_0}} = \mathcal{O} (t_0^{-1/4}) = \mathcal{O}(1),$$
and these terms are $\mathcal{O} (t_0^{-1/4})$. 
 \end{proof}
From the two last propositions, we deduce the following: 
\begin{proposition} \label{xxxxxxxx}
Let $h >0$, $T \geq t_0 \geq 50 ( 1+ h^4)$. Then, 
$$\sup_{t \in [t_0 - h, t_0 + h]} 
\left|\zeta\left(\frac{1}{2} + i t\right) \right|^2
\ll \left( 1 + h \log T + 
\sum_{k \in \mathbb{Z}, 
|k| \leq h \log(T/2 \pi)}
\left|\zeta\left(\frac{1}{2} + i \left(t_0 + \frac{k \pi}{2 \log (T/2 \pi)} \right)  \right) \right|^2 
\right. $$ $$\left. + \sum_{k \in \mathbb{Z}, 
|k| \leq T^{1/4}} \frac{\left|\zeta\left(\frac{1}{2} + i \left(t_0 + \frac{k \pi}{2 \log (T/2 \pi)} \right)  \right) \right|^2}{1 + |k|^3}  \right).
$$
\end{proposition}
\begin{proof}
We will prove the majorization 
with $t_0^{1/4}$ instead of $T^{1/4}$ in the last sum, which is stronger. With this change, and
because of the inequality satisfied by $t_0$, all the values of $t$ such that $\zeta(\frac{1}{2} + it)$ is involved in the modified inequality are in the interval $[t_0 - t_0^{1/4}, t_0 + t_0^{1/4}]$. 
The previous proposition shows that 
for these values of $t$, 
$$Z(t) = H(t) + \mathcal{O}(t_0^{-1/4})$$
where $H(t)$ is dominated by $t_0^{1/4}$, and is also a linear combination of functions of the form $e^{i \mu t}$, with $\mu \in [- \frac{1}{2}\log (t_0/2 \pi),\frac{1}{2} \log (t_0/2 \pi)]$.
We deduce 
$$\left|\zeta \left( \frac{1}{2} + i t \right) \right|^2 
= (H(t))^2 + \mathcal{O}(1),$$
where 
$(H(t))^2$ is a linear combination of $e^{i \mu t}$ where $\mu$ is in $[- \log (t_0/2 \pi), \log (t_0/2 \pi)]$, and a fortiori in $[- \log (T/2 \pi), \log (T/2\pi)]$. 
Since we have a bounded error at each term when we replace $|\zeta(\frac{1}{2} + it)|^2$ by $(H(t))^2$, it is sufficient to show an equality of the following form: 
$$\sup_{t \in [t_0 - h, t_0 + h] } 
(H(t))^2 
\ll \left( 1+ \sum_{k \in \mathbb{Z}, 
|k| \leq h \log(T/2 \pi)} \left(H\left(t_0 + \frac{k \pi}{2 \log (T/2\pi)}\right)\right)^2
\right. $$ $$ +  \left. \sum_{k \in \mathbb{Z}, |k | \leq t_0^{1/4}} 
\frac{\left(H\left(t_0 + \frac{k \pi}{2 \log (T/2\pi)}\right)\right)^2}{1+ |k|^3} \right).$$
Now, this inequality is a consequence of Proposition \ref{Fourier}, applied to $f = H^2$, $\lambda = \log (T/2 \pi)$, $A = 3$, $R = t_0^{1/4}$, since  
$$\frac{\sup_{\mathbb{R}} H^2}{1 + (t_0^{1/4})^2} = \mathcal{O} (1).$$

\end{proof}
We deduce the following bound on the maximum of $|\zeta|$ on a random interval of fixed size, which, by applying Markov's inequality, completes the proof of 
Theorem \ref{unconditional}: 
\begin{proposition}
Let $U$ be a uniform variable on $[0,1]$, and $h > 0$. Then, for all $T \geq 10$, 
$$\mathbb{E} \left[ \sup_{t \in [UT -h,UT + h]} 
\left| \zeta \left( \frac{1}{2} + it \right) \right|^2 \right]
\ll_h (\log T)^2.$$
\end{proposition} 
\begin{proof}
If we replace $\ll$ by  $\ll_h$ in the statement of Proposition  \ref{xxxxxxxx}, it remains true as soon as 
  $T \geq 10$ and $T \geq t_0 \geq 0$, and not only for $T \geq t_0 \geq 50 (1 + h^4)$.
 Indeed, the supremum in the left-hand side in  Proposition  \ref{xxxxxxxx}  is bounded by a quantity depending only on $h$ for $0 \leq t_0 \leq 50(1+ h^4)$. 
 Hence, we can write: 
\begin{align*} 
& \mathbb{E}\left[ \sup_{t \in [UT - h, UT + h]} 
\left|\zeta\left(\frac{1}{2} + i UT\right) \right|^2 \right]
\\ & \ll_h \left( 1 + h \log T + 
\sum_{k \in \mathbb{Z}, 
|k| \leq h \log(T/2 \pi)}
\mathbb{E} \left[\left|\zeta\left(\frac{1}{2} + i \left(UT + \frac{k \pi}{2 \log (T/2 \pi)} \right)  \right) \right|^2 \right]
\right. \\ & \left. + \sum_{k \in \mathbb{Z}, 
|k| \leq T^{1/4}} \frac{\mathbb{E} \left[\left|\zeta\left(\frac{1}{2} + i \left(UT + \frac{k \pi}{2 \log (T/2 \pi)} \right)  \right) \right|^2\right]}{1 + |k|^3}  \right).
\end{align*}
Each expectation in the right-hand side 
is the average of $|\zeta|^2$ on an interval of length $T$ of the critical line, included in  the interval
$$I (T,h) := \left[\frac{1}{2} - \frac{i \pi}{2}  \left( h +\frac{T^{1/4} }{\log (T/2 \pi)}
\right), \frac{1}{2} + iT + \frac{i \pi}{2}  \left( h +\frac{T^{1/4} }{\log (T/2 \pi)}
\right)\right].$$
We deduce 
\begin{align*}
& \mathbb{E}\left[ \sup_{t \in [UT - h, UT + h]} 
\left|\zeta\left(\frac{1}{2} + i UT\right) \right|^2 \right]
\\ & 
\ll_h \left( 1 + h \log T 
+ \frac{1 + h \log (T/2 \pi) 
+ \sum_{k \in \mathbb{Z}} \frac{1}{1 + |k|^3} } {T} \int_{I(T,h)} |\zeta(s)|^2 |ds|\right). 
\end{align*}
By a classical result of Hardy and Littlewood on the second moment of $\zeta$ (see \cite{Tit}, Theorem 7.3), 
the last integral is dominated (with an implicit constant depending on $h$) by $T \log T$, which gives the desired result. 
\end{proof}

\section{Averaging  $\log \zeta$}
\label{section:average}
It is known, from the Euler product and the series of the logarithm, that for $\Re(s) > 1$, 
$$\log \zeta(s) = \sum_{n \geq 1} \ell(n) n^{-s},$$
where $\ell(n) = 1/k$ if $n$ is a $k$-th power of a prime ($k \geq 1$ integer), and $\ell(n) = 0$ otherwise. 
If we apply this formula for $s + it$ instead of $s$, and if we average by integrating with respect to $\varphi(t) dt$, then we get a sum in the right-hand side
which involves the Fourier transform of $\varphi$. This is particularly interesting if $\widehat{\varphi}$ is compactly supported, since the sum has finitely many 
non-zero terms in this case. More precisely, we have the following: 
\begin{proposition} \label{logRe1}
Let $\varphi$ be an integrable function
from $\mathbb{R}$ to $\mathbb{R}$, such that  $\widehat{\varphi}$ is compactly supported. 
Then, for $\Re(s) > 1$, the following 
quantity: 
$$L_{\varphi} (s) = \int_{-\infty}^{\infty} 
\log \zeta(s + it) \varphi(t) dt$$
is well-defined,  and one has 
$$L_{\varphi}(s) =\sum_{n \geq 1} \ell(n) n^{-s}\widehat{\varphi}(\log n),$$
Moreover, the last formula defines an analytic 
continuation of $L_{\varphi}$ to the whole complex plane. 
\end{proposition}
\begin{proof}
For $\Re(s) > 1$ and $t \in \mathbb{R}$, 
$$\log \zeta(s + it) 
= -  \sum_{p \in \mathcal{P}} \log (1 - p^{-s-it}) = \sum_{p \in \mathcal{P}} \sum_{k = 1}^{\infty} \frac{ p^{-k(s+ it)} }{k}
= \sum_{n \geq 1} \ell(n) n^{-s-it},
$$
all the series being absolutely convergent and dominated by 
$\sum_{n \geq 1} n^{-\Re(s)}$, 
uniformly in $t$ if $s$ is fixed. 
Integrating in $t$, we get, using this domination and the fact that $\varphi$ is integrable: 
$$L_{\varphi}(s) = \sum_{n \geq 1} \ell(n) 
n^{-s} \int_{-\infty}^{\infty} n^{-it} \varphi(t) dt =  \sum_{n \geq 1} \ell(n) 
n^{-s} \widehat{\varphi}(\log n).$$
The last series has finitely many nonzero terms since $\widehat{\varphi}$ is compactly supported, and then it defines an entire function extending $L_{\varphi}$. 
\end{proof}
The result we have just proven gives some information on $\log \zeta$ at points whose real part is strictly larger than $1$. Of course, we are more interested
in what happens on the critical line. To extend our previous result to the critical strip, we can think about the principle of analytic continuation, but we need to be careful, since $\log \zeta$ is not holomorphic everywhere because of 
the zeros and the pole of $\zeta$. However, if we assume the Riemann hypothesis, the only problem at the right of the critical line comes from the pole of $\zeta$ at $1$. The main result of the section is the following proposition, which shows that this pole gives a well-controlled error term when the  function $\varphi$ satisfies some suitable extra assumptions:
\begin{proposition} \label{sumsprimes}
Let us  assume the Riemann hypothesis. Let $\varphi$ be a function from $\mathbb{R}$ to $\mathbb{R}$, dominated by any negative power at infinity, and whose Fourier transform is compactly supported. Then, for $\sigma \in [1/2,1)$, 
$\tau \in \mathbb{R}$, $H > 0$,
\begin{equation}
\int_{-\infty}^{\infty} 
\log \zeta(\sigma + i (\tau + tH^{-1}))
\varphi(t) dt 
= \sum_{n \geq 1} \ell(n) n^{-\sigma - i\tau} \widehat{\varphi}\left( \frac{\log n}{H} \right)  + \mathcal{O}_{\varphi}
\left(1 + \frac{e^{\mathcal{O}_{\varphi} (H)}}{1+ |\tau|} \right). \label{123456789}
\end{equation}
\end{proposition}
\begin{rmk}
Contrarily to the case $\Re(s) > 1$, the error term does not vanish in general, since one can check, from the discontinuity of the logarithm, that the left-hand side is not holomorphic in $\sigma + i \tau$ if we allow $\sigma$ to go below $1$. 
Of course, the integral does not need that $\log \zeta$ is well-defined at the left of the zeros or the pole. 
\end{rmk}
\begin{proof}
For all $z \in \mathbb{R}$, we have 
$$\varphi(z) = \frac{1}{2 \pi} \int_{\mathbb{R}} e^{i z \lambda} \widehat{\varphi} (\lambda) d \lambda,$$
and since $\widehat{\varphi}$ is compactly supported, this formula can be extended to all $z \in \mathbb{C}$, which gives an analytic continuation of $\varphi$. 
Let us now define an entire function $V$ as $V(z) := H \varphi(Hz)$. In Lemma 5 of Tsang \cite{Tsang86}, the 
assumption (2.10) is satisfied for any given value of $H$. Indeed, for  $\sigma - 2 \leq y \leq 0$, 
$$ V(x + iy) = \frac{H}{2 \pi}   \int_{\mathbb{R}} e^{(-y + ix) H \lambda} \widehat{\varphi} (\lambda) d \lambda
 =  \frac{H}{2 \pi}   \int_{\mathbb{R}} \frac{e^{(-y + ix) H \lambda}}{ (-y + ix)^2 H^2} \widehat{\varphi}'' (\lambda) d \lambda.$$
by integration by part: note that $\widehat{\varphi}$ is smooth since $\varphi$ is rapidly decaying at infinity. 
Since $\widehat{\varphi}'' $ is compactly supported and $y$ is uniformly bounded, we have $V(x+ iy) = \mathcal{O}_{\varphi, H} (x^{-2})$, uniformly in $y \in [\sigma-2, 0]$, and a fortiori   $V(x+ iy) = \mathcal{O}_{\varphi, H} (|x|^{-1} (\log |x|)^{-2})$: the assumption (2.10) of \cite{Tsang86}, Lemma 5 is satisfied. This lemma can then 
be applied: however, it does not directly give  uniformity of the error term with respect to $H$, so we need a little extra work to deduce Proposition \ref{sumsprimes}. Let us recall here the strategy of the proof of the lemma in \cite{Tsang86}. We have
$$\int_{-\infty}^{\infty} 
\log \zeta(\sigma + i (\tau + tH^{-1}))
\varphi(t) dt  = \int_{-\infty}^{\infty} 
\log \zeta(\sigma + i (\tau + u))
V(u) du.$$ 
We shift the path of integration by $i (\sigma -2)$. The rate of decay of $V$ at infinity implies that there is no effect 
of this shift related to the tail of the integrals. Since there is no zero of $\zeta$ at the right of $\sigma$ (we assume Riemann hypothesis), the integral only changes because of the discontinuity of $\log \zeta$ at points of the interval $(\sigma, 1)$, and this change gives a term which can be explicitly written. After shifting the path of integration, we get an 
integral involving the value of $\log \zeta$ at the line $\Re(s) = 2$. Using the series of $\log \zeta$, this integral can be written as a series of integrals indexed by the integers, which then gives the first term of the right-hand side of \eqref{123456789}, after shifting the paths of integration by  $-i (\sigma -2)$ (this second shift replaces the path of integration at its initial position). 
By  looking in more detail at the proof of \cite{Tsang86}, Lemma 5, we find that the error term in our proposition can be exactly written, for $\sigma \in [1/2,1)$, as 
$$
- 2 \pi \int_0^{1 - \sigma} V(-\tau - i \alpha) d \alpha,$$
assuming the Riemann hypothesis (if we do not make this assumption, there are extra terms involving the zeros of $\zeta$ at the right of $\sigma$). 
We know that
$$V(-\tau - i \alpha)  =  \frac{H}{2 \pi}   \int_{\mathbb{R}} e^{(\alpha - i \tau) H \lambda} \widehat{\varphi} (\lambda) d \lambda =  \frac{H}{2 \pi}   \int_{\mathbb{R}} \frac{e^{(\alpha - i \tau) H \lambda}}{ (\alpha - i \tau)^2 H^2} \widehat{\varphi}'' (\lambda) d \lambda.$$
Since $\widehat{\varphi}$ is compactly supported, the two equalities respectively give 
$$\int_0^{1/2}  |V(-\tau - i \alpha)|  d \alpha = 
\mathcal{O}_{\varphi} (e^{\mathcal{O}_{\varphi}(H)}).$$
 and 
$$\int_0^{1/2}  |V(-\tau - i \alpha)|  d \alpha = 
\mathcal{O}_{\varphi} ( H^{-1} \tau^{-2} e^{\mathcal{O}_{\varphi}(H)}).$$
Combining the two estimates, we get 
$$\int_0^{1/2}  |V(-\tau - i \alpha)|  d \alpha =\mathcal{O}_{\varphi} ( ( 1+ H \tau^2)^{-1} e^{\mathcal{O}_{\varphi}(H)})).$$
 If $H \leq 1$, we deduce that the error term is $\mathcal{O}_{\varphi}(1)$, whereas 
for $H > 1$, we get a bound  $\mathcal{O}_{\varphi} ( ( 1+ \tau^2)^{-1} e^{\mathcal{O}_{\varphi}(H)}))$. This implies 
Proposition \ref{sumsprimes} in both cases. 
\end{proof}
\begin{rmk}
As we see in the proof above, the assumption on the decay of $\widehat{\varphi}$ can be easily relaxed (it is enough to know that $\widehat{\varphi}''$ is well-defined and continuous), and the denominator $1 + |\tau|$ can be improved to $(1 + \tau^2)$ in the error term. However, these possible improvements are not particularly useful for our purpose. 
\end{rmk}
\begin{rmk}
If we do not assume the Riemann hypothesis, we get some extra error terms. It seems difficult to get good estimates of them. For example, if a zero $\rho$ is far from the critical line, for example $\rho = 3/4 + i \gamma$ (this situation is not proven to be impossible), and if $\tau = \gamma$ is the imaginary part of this zero, we get, for $\sigma = 1/2$, a term 
equal to the integral of $\alpha \mapsto V(-i \alpha)$ between $0$ and $1/4$, i.e. the integral of $y \mapsto \varphi(-i y)$ between $0$ and $H/4$. Since $\varphi$ can increase exponentially on the imaginary axis, the bound we get for the error term is  exponentially increasing with $H$, which is much too large since we will need to take $H$ close to a power of $\log T$, whereas the extreme values of $\log \zeta$ we consider have order $\log \log T$. 
However, if we prevent $\tau$ to be too close to zeros of $\zeta$ which are far from the critical line, it may be possible to get some useful information about the error terms. 
\end{rmk}

The previous proposition  is only interesting if there exist functions $\varphi$ satisfying the corresponding assumptions. Indeed: 
\begin{proposition} \label{existence}
There exists a function
$\varphi$, real, nonnegative, even, dominated by any negative power at infinity, and such that its Fourier transform is compactly supported, takes values in $[0,1]$, is even and equal to $1$ at zero (which implies that the integral of $\varphi$ is $1$). 
\end{proposition}
\begin{proof}
Let $\alpha$ be a nonnegative, smooth, compactly supported, even function whose $L^2$ norm is equal to $1$. We define $\psi$ as the convolution of $\alpha$ with itself:  
$$\psi(x) 
= \int_{\mathbb{R}} 
\alpha(y) \alpha(x-y) dy.$$
It is clear that $\psi$
is smooth, compactly supported, takes values in $[0,1]$ (by Cauchy-Schwarz inequality and the fact that $\alpha$ is nonnegative), is even and equal to $1$ at zero. 
We now define $\varphi$ as the inverse Fourier transform of $\psi$. 
This function is real and even since it is the case for $\psi$, dominated by any power at infinity since $\psi$ is smooth and compactly supported, and nonnegative since it is $2\pi$ times the square of the inverse Fourier transform of $\alpha$, which is real since $\alpha$ is real and even. 
\end{proof}

\section{The upper bound for the imaginary part} 
 \label{section:upperboundimaginary}
Similarly as what we have seen for $\Re \log \zeta$, the supremum of $\Im \log \zeta$ on an interval can be controlled by its values at finitely many points. 
This comes from the fact that the argument of $\zeta$ on the critical line has positive jumps of size $ \pi$ when we reach imaginary parts of zeros of $\zeta$, and decreases continuously, in a very well-controlled way, between the zeros of $\zeta$. We deduce that  $\Im \log \zeta$ cannot decrease too fast on the critical line: 
\begin{proposition} \label{argzeta}
For $2 \leq t_1 \leq t_2$ which are not imaginary parts of zeros of $\zeta$, we have: 
$$\Im \log \zeta(1/2 + i t_2) \geq \Im \log \zeta(1/2 + i t_1) - (t_2 - t_1) \log t_2 + \mathcal{O}(1).$$
\end{proposition} 
\begin{proof}
From Theorem 9.3. of \cite{Tit}, we have for $t \geq 2$: 
$$ \Im \log \zeta(1/2 + i t) = \pi N(t) - \frac{1}{2} t \log t + \frac{t(1 + \log (2 \pi))}{2} + \mathcal{O}(1),$$
where $N(t)$ denotes the number of zeros of $\zeta$ with imaginary part in the interval $(0,t]$. Since 
$N(t_1) \leq N(t_2)$, we deduce 
\begin{align*}
& \Im \log \zeta(1/2 + i t_2) - \Im \log \zeta(1/2 + i t_1) 
\\ & \geq - \frac{1}{2} [(t_2 \log t_2 - t_2) - (t_1 \log t_1 - t_1)] + (t_2 - t_1) \frac{\log (2 \pi)}{2} + 
 \mathcal{O}(1) 
\\ & \geq - \frac{1}{2} \int_{t_1}^{t_2} \log u \, du + \mathcal{O}(1),
\end{align*}
 which proves the claim. 
\end{proof}
From this result, we deduce that the argument of $\zeta$ can be controlled by some of its averages, which then implies, from Proposition \ref{sumsprimes}, that it is also controlled by suitable finite sums indexed by primes. 
The control by the averages is  given in Proposition \ref{B9}, proven in several steps, given by  Lemmas \ref{un}, \ref{deux} and \ref{trois}. The averages are then compared with sums on primes in Proposition \ref{B1B2B3}. These sums are estimated in Lemmas \ref{P1}, \ref{P2} and \ref{P3}, which finally gives the main result of this section in Proposition \ref{upperboundimaginary}.
\begin{lemma} \label{un}
Let $\varphi$ be  a function from $\mathbb{R}$ to $\mathbb{R}_+$ with integral $1$, and decaying faster than any power at infinity.  Let $h > 0$, $\epsilon \in (0,1)$. Then, for $\tau > 3$ large enough depending only on $h$, $\epsilon$ and $\varphi$, and 
for $(\log \tau)^{1/10} \leq H \leq \log \tau$, 
\begin{align*}
&\underset{t \in [\tau - 2 h, \tau + 2 h]}{\sup} \left| \int_{-\infty}^{\infty} \Im \log \zeta(1/2 + i (t + u H^{-1})) \varphi(u) du \right| 
\\ & \geq (1-\epsilon) \underset{t \in [\tau -  h, \tau +  h]}{\sup} \left| \Im \log \zeta (1/2 + it) \right| - \epsilon \underset{t \in [\tau -2  h, \tau +  2h]}{\sup} \left| \Im \log \zeta (1/2 + it) \right|
- \mathcal{O}_{\varphi,\epsilon, h}(H^{-1} \log \tau). 
\end{align*}
\end{lemma}
\begin{proof}
Let
$$ M_1  := \underset{t \in [\tau -  h, \tau +  h]}{\sup} \left| \Im \log \zeta (1/2 + it) \right|, $$
$$M_2  := \underset{t \in [\tau -  2h, \tau +  2h]}{\sup} \left| \Im \log \zeta (1/2 + it) \right|.$$
The quantity $M_1$ is either the  supremum of the positive part of $\Im \log \zeta$ or the  supremum of 
its negative part. We assume that it is the supremum of the positive part: the other case can be covered similarly, up to small details which are left to the reader.  
There exists $t_0 \in [\tau - h, \tau + h]$ such that 
$$ \Im \log \zeta (1/2 + it_0) \geq M_1 -1.$$
We can assume $\tau > 2(1+h)$, which implies $t_0 > 2$, and then, for $u  > 0$, by using Proposition \ref{argzeta},
$$\Im \log \zeta (1/2 + i(t_0 + u H^{-1})) \geq M_1 - uH^{-1} \log(t_0 + u H^{-1}) + \mathcal{O}(1), $$
for $u \in [-hH,0]$, 
$$ \Im \log \zeta (1/2 + i(t_0 + u H^{-1})) \geq - M_2,$$
and 
for $u  < -hH$, 
$$ | \Im \log \zeta (1/2 + i(t_0 + u H^{-1}))|  \ll \log( 2 + |t_0| + |u| H^{-1}),$$
the last estimate coming from Theorem 9.4. of \cite{Tit}.   
Let us now integrate these estimates with respect to $\varphi(u-u_0) d u$, with $u_0 > 0$ to be chosen later. 
The first estimate gives 
\begin{align*}
& \int_0^{\infty} \Im \log \zeta (1/2 + i(t_0 + u H^{-1})) \varphi(u-u_0)  du 
\\ & \geq M_1 \int_{-u_0}^{\infty} \varphi(v) dv - H^{-1} \int_{-u_0}^{\infty}  (v + u_0) [\log (1+t_0) + \log(1+|v|) + \log(1+u_0)] \varphi(v) d v + \mathcal{O}(1),
\end{align*}
since $H \geq (\log \tau)^{1/10} \geq  (\log 3)^{1/10} > 1$, then for $v \geq -u_0$,
$$ t_0 + (v+u_0) H^{-1} \leq t_0+v + u_0 \leq ( 1+ t_0)(1+|v|)(1 + u_0),$$
and the integral of $\varphi$ is $1$.
Since $\varphi$ is integrable with respect to $(1+|v|)(1 + \log(1+|v|) )\, dv$, we deduce 
\begin{align*}
& \int_0^{\infty} \Im \log \zeta (1/2 + i(t_0 + u H^{-1})) \varphi(u-u_0)  du 
\\ & \geq M_1 \int_{-u_0}^{\infty} \varphi(v) dv - \mathcal{O}_{\varphi,u_0} (1 + H^{-1} \log t_0)
\\ &  \geq M_1 \int_{-u_0}^{\infty} \varphi(v) dv - \mathcal{O}_{\varphi, u_0} (H^{-1} \log \tau),
\end{align*}
the last line coming from the fact that on the one hand,  we can assume $\tau > 2h$ and then $\log (t_0) = \log (\tau) + \mathcal{O}(1)$ for all $t \in [\tau - h, \tau + h]$, and on the other hand, $H^{-1} \log \tau \geq 1$ by assumption on $H$. 
The second estimate of $\Im \log \zeta$ gives$$ \int_{-hH}^{0} \Im \log \zeta (1/2 + i(t_0 + u H^{-1})) \varphi(u-u_0)  du 
 \geq - M_2 \int_{-\infty}^{-u_0} \varphi(v) dv.$$
Finally, the last estimate gives
\begin{align*}
& \int_{-\infty}^{-hH} \Im \log \zeta (1/2 + i(t_0 + u H^{-1})) \varphi(u-u_0)  du 
\\ & \geq -\mathcal{O} \left( \int_{-\infty}^{-hH - u_0}[ \log( 2 + t_0) + \log (2 + u_0) + \log (2+|v|) ]  \varphi(v) dv \right) 
\\ & \geq -\mathcal{O}_{\varphi,u_0,h,A} \left( H^{-A} \log t_0 \right) =  -\mathcal{O}_{\varphi,u_0,h,A} \left( H^{-A} \log \tau \right),
\end{align*}
for any $A  > 0$, since $\varphi$ is rapidly decaying at $-\infty$  by assumption. 
Since $H \geq (\log \tau)^{1/10}$, we obtain, by taking $ A = 10$, a lower bound $-\mathcal{O}_{\varphi,u_0,h}(1)$. 
Adding the three integrals on the intervals $(-\infty,-hH]$, $[-hH,0]$ and $[0, \infty)$, and translating the interval of integration, we deduce 
\begin{align*}
& \int_{-\infty}^{\infty} \Im \log \zeta (1/2 + i(t_0 + u_0 H^{-1} +  u H^{-1})) \varphi(u)  du 
\\ & \geq  M_1 \int_{-u_0}^{\infty} \varphi(v) dv  - M_2 \int_{-\infty}^{-u_0} \varphi(v) dv - \mathcal{O}_{\varphi,u_0,h}(H^{-1} \log \tau).
\end{align*}
We can now choose  $u_0$ depending only on $\varphi$ and $\epsilon$, sufficiently large in order to have
$$\int_{-\infty}^{-u_0} \varphi(v) dv \leq \epsilon.$$
Then, by taking $t_1 = t_0 + u_0 H^{-1} = t_0 + \mathcal{O}_{\varphi,\epsilon} (H^{-1})$, we get 
\begin{align*}
& \int_{-\infty}^{\infty} \Im \log \zeta (1/2 + i(t_1 + u H^{-1})) \varphi(u)  du 
\\ & \geq  M_1 (1-\epsilon)  - M_2 \epsilon- \mathcal{O}_{\varphi,\epsilon,h}(H^{-1} \log \tau).
\end{align*}
If $\tau$ is large enough depending on $h$, $\epsilon$ and $\varphi$, then $H \geq (\log \tau)^{1/10}$ can be assumed to be sufficiently large in order to 
have 
$$t_1 - t_0 =  \mathcal{O}_{\varphi,\epsilon} (H^{-1}) \leq h,$$
and then $t_1 \in [\tau - 2h, \tau + 2h]$. 

This completes the proof of the lemma.

\end{proof}
We deduce the following probabilistic result: 
\begin{lemma} \label{deux}
Let $\varphi$ be as in the previous lemma, let $\epsilon \in (0,1)$, $A \geq 1$, $h \geq 0$. For $T > 100$, let $H := (\log T)(\log \log \log T)^{1/2} (\log \log T)^{-1}$. Then, for $U$ random, uniform on $[0,1]$, we have 
\begin{align*}
& \mathbb{P} \left[  \underset{t \in [UT -  h, UT +  h]}{\sup} \left| \Im \log \zeta (1/2 + it) \right|  \geq A (1+\epsilon) \log \log T \right]
\\ & \leq  \mathbb{P} \left[  \underset{t \in [UT - 2 h, UT+ 2 h]}{\sup} \left| \int_{-\infty}^{\infty} \Im \log \zeta(1/2 + i (t + u H^{-1})) \varphi(u) du \right| \geq A \log \log T   \right]
\\ & + 2 \,  \mathbb{P} \left[  \underset{t \in [UT -  h, UT +  h]}{\sup} \left| \Im \log \zeta (1/2 + it) \right|  \geq 10 A (1+\epsilon) \log \log T \right]
+ \mathcal{O}_{\varphi,h,\epsilon} ( T^{-0.99}).
\end{align*}
\end{lemma}
\begin{rmk}
The implied constant does not depend on $A$. 
\end{rmk}
\begin{proof}
Except on an event of probability $\mathcal{O}_{\varphi, h, \epsilon} (T^{-1})$, $\tau = UT$ is large enough in order to apply the previous lemma. Moreover, by changing the implicit constant  
in $\mathcal{O}_{\varphi,h,\epsilon} ( T^{-0.99})$, we can assume that $T$ is sufficiently large, in order to have $(\log \tau)^{1/10} \leq (\log T)^{1/10} \leq H$. We also have, for $T$ large enough: 
\begin{align*}
 \mathbb{P} [ H > \log \tau] & = \mathbb{P} \left[ \log (UT) < (\log T)(\log \log \log T)^{1/2} (\log \log T)^{-1} \right]
\\ & \leq \mathbb{P} \left[ \log (UT) < \frac{1}{100} \log T \right] = \mathbb{P} [ UT < T^{1/100} ] = T^{-0.99}.
\end{align*}
Applying the previous lemma (with a different value of $\epsilon$), we deduce that, outside an event of probability $\mathcal{O}_{\varphi,h,\epsilon} ( T^{-0.99})$, 
\begin{align*}
&\underset{t \in [\tau - 2 h, \tau + 2 h]}{\sup} \left| \int_{-\infty}^{\infty} \Im \log \zeta(1/2 + i (t + u H^{-1})) \varphi(u) du \right| 
\\ & \geq (1-\epsilon) \underset{t \in [\tau -  h, \tau +  h]}{\sup} \left| \Im \log \zeta (1/2 + it) \right| - \epsilon \underset{t \in [\tau -2  h, \tau +  2h]}{\sup} \left| \Im \log \zeta (1/2 + it) \right| - \mathcal{O}_{\varphi,\epsilon, h}(H^{-1} \log T),
\end{align*}
for $\tau = UT$.
Let us assume that $\epsilon < 1/12$, that the second supremum is as least $A ( \log \log T) /(1 - 12 \epsilon) $, and that the last supremum is at most $10 A (\log \log T) /(1 - 12 \epsilon) $. In this case, outside an event of probability $\mathcal{O}_{\varphi,h,\epsilon} ( T^{-0.99})$, 
we get 
\begin{align*}
& \underset{t \in [\tau - 2 h, \tau + 2 h]}{\sup} \left| \int_{-\infty}^{\infty} \Im \log \zeta(1/2 + i (t + u H^{-1})) \varphi(u) du \right| 
\\ & \geq \frac{A(\log \log T)( 1 - 11\epsilon)}{1 - 12\epsilon} -  \mathcal{O}_{\varphi,\epsilon, h}((\log \log \log T)^{-1/2} \log \log T)
\\ & \geq A \log \log T,
\end{align*} 
the last equality being true as soon as $T$ is large enough depending on $\varphi,\epsilon, h$ (not on $A$ since  $A \geq 1$). This can always be assumed by changing the implicit constant in  $\mathcal{O}_{\varphi,h,\epsilon} ( T^{-0.99})$.
We then get: 
\begin{align*}
& \mathbb{P} \left[  \underset{t \in [\tau -  h, \tau +  h]}{\sup} \left| \Im \log \zeta (1/2 + it) \right|  \geq \frac{A \log \log T}{1-12 \epsilon} \right]
\\ & \leq  \mathbb{P} \left[  \underset{t \in [\tau - 2 h, \tau+ 2 h]}{\sup} \left| \int_{-\infty}^{\infty} \Im \log \zeta(1/2 + i (t + u H^{-1})) \varphi(u) du \right| \geq A \log \log T   \right]
\\ & + \,  \mathbb{P} \left[  \underset{t \in [\tau - 2 h, \tau + 2  h]}{\sup} \left| \Im \log \zeta (1/2 + it) \right|  \geq \frac{10 A \log \log T}{1-12 \epsilon}\right]
+ \mathcal{O}_{\varphi,h,\epsilon} ( T^{-0.99}).
\\ & \leq  \mathbb{P} \left[  \underset{t \in [\tau - 2 h, \tau+ 2 h]}{\sup} \left| \int_{-\infty}^{\infty} \Im \log \zeta(1/2 + i (t + u H^{-1})) \varphi(u) du \right| \geq A \log \log T   \right]
\\ & + \,  \mathbb{P} \left[  \underset{t \in [\tau' - h, \tau' + h]}{\sup} \left| \Im \log \zeta (1/2 + it) \right|  \geq \frac{10 A \log \log T}{1-12 \epsilon}\right]
\\ & + \,  \mathbb{P} \left[  \underset{t \in [\tau'' - h, \tau'' + h]}{\sup} \left| \Im \log \zeta (1/2 + it) \right|  \geq \frac{10 A \log \log T}{1-12 \epsilon}\right]
+ \mathcal{O}_{\varphi,h,\epsilon} ( T^{-0.99}),
\end{align*}
where $\tau' = \tau + h$ and $\tau'' = \tau - h$. Since $\tau$ is uniform on $[0,T]$, it is possible to define random variables $\tau_1$ and $\tau_2$ with the same laws as $\tau'$ and $\tau''$, and both equal to $\tau$ with probability at least $1 - (h/T)$ (for example, take $\tau_1 = \tau + T \mathds{1}_{\tau < h}$ and $\tau_2 = \tau -   T \mathds{1}_{\tau> T-h})$. 
We deduce that we can replace $\tau'$ and $\tau''$ by $\tau_1$ and $\tau_2$, and then both by $\tau$, in the previous estimate. We then get the claim in the lemma, with $\epsilon$ replaced by $(12 \epsilon)/(1-12 \epsilon)$. Since this quantity 
can take any value in $(0,1)$, we are done. 
\end{proof}
We then deduce  the following: 
\begin{lemma} \label{trois}
With the notation of the previous lemma: 
\begin{align*}
& \mathbb{P} \left[  \underset{t \in [UT -  h, UT +  h]}{\sup} \left| \Im \log \zeta (1/2 + it) \right|  \geq A (1+\epsilon) \log \log T \right]
\\ & \leq  \mathbb{P} \left[  \underset{t \in [UT - 2 h, UT+ 2 h]}{\sup} \left| \int_{-\infty}^{\infty} \Im \log \zeta(1/2 + i (t + u H^{-1})) \varphi(u) du \right| \geq A \log \log T   \right]
\\ & + (\log T) \mathbb{P} \left[  \underset{t \in [UT - 2 h, UT+ 2 h]}{\sup} \left| \int_{-\infty}^{\infty} \Im \log \zeta(1/2 + i (t + u H^{-1})) \varphi(u) du \right| \geq 10 \log \log T   \right]
\\ &+ \mathcal{O}_{\varphi,h,\epsilon} ( T^{-0.98}).
\end{align*}
\end{lemma} 
\begin{proof}
We can iterate the result of the previous lemma, which gives, for $k_0 \geq 1$ integer: 
\begin{align*}
& \mathbb{P} \left[  \underset{t \in [UT -  h, UT +  h]}{\sup} \left| \Im \log \zeta (1/2 + it) \right|  \geq A (1+\epsilon) \log \log T \right]
\\ & \leq  \sum_{k = 0}^{k_0 - 1} 
2^k \mathbb{P} \left[  \underset{t \in [UT - 2 h, UT+ 2 h]}{\sup} \left| \int_{-\infty}^{\infty} \Im \log \zeta(1/2 + i (t + u H^{-1})) \varphi(u) du \right| \geq 10^k A \log \log T   \right]
\\ &+  2^{k_0} \,  \mathbb{P} \left[  \underset{t \in [UT -  h, UT +  h]}{\sup} \left| \Im \log \zeta (1/2 + it) \right|  \geq 10^{k_0} A (1+\epsilon) \log \log T \right]
+ \mathcal{O}_{\varphi,h,\epsilon} \left( 2^{k_0} T^{-0.99} \right). 
\end{align*}
We now take $k_0 = 1 + \lfloor \log \log T \rfloor $, which implies that the last probability is zero for $T$ large enough depending on $h$, since, by  Theorem 9.4. of Titchmarsh \cite{Tit}, 
$$ \underset{t \in [UT -  h,U T +  h]}{\sup} \left| \Im \log \zeta (1/2 + it) \right| \leq    \underset{t \in [-  h, T +  h]}{\sup} \left| \Im \log \zeta (1/2 + it) \right|  \ll_h \log T,$$
whereas 
$$ 10^{k_0} A (1+\epsilon) \log \log T \geq 10^{k_0} \geq (\log T)^{\log 10}.$$
For $T$ large enough, $2^{k_0} T^{-0.99} \leq T^{-0.98}$. 
In the sum in $k$, each term corresponding to $k  > 0$ is at most 
$$ 2^{k}\mathbb{P} \left[  \underset{t \in [UT - 2 h, UT+ 2 h]}{\sup} \left| \int_{-\infty}^{\infty} \Im \log \zeta(1/2 + i (t + u H^{-1})) \varphi(u) du \right| \geq 10 \log \log T   \right],$$
since $A \geq 1$. Hence, the sum of the terms for $k = 1$ to $k_0 - 1$ is at most 
$$ (2^{k_0}-2)\mathbb{P} \left[  \underset{t \in [UT - 2 h, UT+ 2 h]}{\sup} \left| \int_{-\infty}^{\infty} \Im \log \zeta(1/2 + i (t + u H^{-1})) \varphi(u) du \right| \geq 10 \log \log T   \right],$$
where $$2^{k_0}-2 \leq 2^{1 + \log \log T} = 2 (\log T)^{\log 2} \leq \log T$$
for $T$ large enough. 
Adding the term corresponding to $k = 0$ gives the desired result. 
\end{proof}
The properties of the imaginary part of $\zeta$ on the critical line imply that the previous lemma can be rewritten as follows: 
\begin{proposition} \label{B9}
Let $\varphi$ be as in the previous lemma, let  $A > B \geq 1$, $h > 0$. For $T > 100$, let $H := (\log T)(\log \log \log T)^{1/2} (\log \log T)^{-1}$.  Then, for $U$ random, uniform on $[0,1]$,
\begin{align*}
& \mathbb{P} \left[  \underset{t \in [UT -  h, UT +  h]}{\sup} \left| \Im \log \zeta (1/2 + it) \right|  \geq A\log \log T \right]
\\ & \leq ( \log T) \left(\sup_{d \in [-2h,2h]} \mathbb{P} \left[  \left| \int_{-\infty}^{\infty} \Im \log \zeta(1/2 + i (UT + d + u H^{-1})) \varphi(u) du \right| \geq B \log \log T   \right] \right)
\\ & + (\log T)^2 \left(\sup_{d \in [-2h,2h]} \mathbb{P} \left[  \left| \int_{-\infty}^{\infty} \Im \log \zeta(1/2 + i (UT + d + u H^{-1})) \varphi(u) du \right| \geq 9 \log \log T   \right] \right)
\\ &+ \mathcal{O}_{\varphi,h,A,B} ( T^{-0.98}).
\end{align*}
\end{proposition}
\begin{proof}
Let $A' := (A+B)/2$. Applying the previous lemma to $A'$ instead of $A$,  and $\epsilon \leq  (A/A')-1$, depending only on $A$ and $B$, we get 
\begin{align*}
& \mathbb{P} \left[  \underset{t \in [UT -  h, UT +  h]}{\sup} \left| \Im \log \zeta (1/2 + it) \right|  \geq A \log \log T \right]
\\ & \leq  \mathbb{P} \left[  \underset{t \in [UT - 2 h, UT+ 2 h]}{\sup} \left| \int_{-\infty}^{\infty} \Im \log \zeta(1/2 + i (t + u H^{-1})) \varphi(u) du \right| \geq A' \log \log T   \right]
\\ & + (\log T) \mathbb{P} \left[  \underset{t \in [UT - 2 h, UT+ 2 h]}{\sup} \left| \int_{-\infty}^{\infty} \Im \log \zeta(1/2 + i (t + u H^{-1})) \varphi(u) du \right| \geq 10 \log \log T   \right]
\\ &+ \mathcal{O}_{\varphi,h,A,B} ( T^{-0.98}).
\end{align*}
Discarding an event of probability $\mathcal{O}_h(T^{-1})$, which can be absorbed in the error term, we can assume that  $[UT -3 h, UT+ 3 h]$ is included in $[2, T]$. Hence, for $UT - 2h \leq t_1 < t_2 \leq UT + 2h$ and $u \in [-hH,hH]$, we 
get from Proposition \ref{argzeta}, 
$$ \Im \log \zeta(1/2 + i (t_2 + u H^{-1})) \geq \Im \log \zeta(1/2 + i (t_1 + u H^{-1})) - (t_2 - t_1) \log T + \mathcal{O}(1). $$
If $T$ is large enough (which implies $H \geq 1$), we also get, for  $u \notin [-hH,hH]$,
$$ \Im \log \zeta(1/2 + i (t_2 + u H^{-1})) \geq \Im \log \zeta(1/2 + i (t_1 + u H^{-1})) - \mathcal{O} (\log (T+ |u|))$$
by the classical bound of $\Im \log \zeta$ on the critical line (Theorem 9.4. of \cite{Tit}). 
Integrating with respect to $\varphi(u) du$, using the rapid decay of $\varphi$ and the fact that $H \geq (\log T)^{1/10}$, we deduce, for $T$ large enough, 
\begin{align*}
\int_{-\infty}^{\infty} \Im \log \zeta(1/2 + i (t_2 + u H^{-1})) \varphi(u) du 
 & \geq \int_{-\infty}^{\infty} \Im \log \zeta(1/2 + i (t_1 + u H^{-1})) \varphi(u) du   \\ & - (t_2 - t_1) \log T + \mathcal{O}_h(1).
\end{align*}
Hence, for $T$ large enough, the extrema of the last integral for  $t_1 \in [UT - 2h, UT + 2h]$ are controlled, up to an error $\mathcal{O}_h(1)$, by the highest and the lowest values corresponding to $t_1 = UT - 2 h + 4 h k /(\lfloor \log T \rfloor - 1)$, where 
$k \in \{0,1,\dots, \lfloor \log T \rfloor - 1 \}$.
A union bound then gives: 
\begin{align*}
& \mathbb{P} \left[  \underset{t \in [UT -  h, UT +  h]}{\sup} \left| \Im \log \zeta (1/2 + it) \right|  \geq A \log \log T \right]
\\ & \leq  \sum_{k = 0}^{ \lfloor \log T \rfloor - 1 }  \mathbb{P} \left[  \left| \int_{-\infty}^{\infty} \Im \log \zeta(1/2 + i (UT   - 2 h + 4 h k /(\lfloor \log T \rfloor - 1)+ u H^{-1})) \varphi(u) du \right| \right.
\\ & \qquad \qquad \qquad \left.  \geq A' \log \log T +  \mathcal{O}_h(1)   \right]
\\ & \qquad  + (\log T)  \sum_{k = 0}^{ \lfloor \log T \rfloor - 1 }  \mathbb{P} \left[  \left| \int_{-\infty}^{\infty} \Im \log \zeta(1/2 + i (UT   - 2 h + 4 h k /(\lfloor \log T \rfloor - 1) + u H^{-1})) \varphi(u) du \right|
\right.
\\ & \qquad \qquad \qquad \left. \geq 10 \log \log T +  \mathcal{O}_h(1)   \right]
\\ & \qquad + \mathcal{O}_{\varphi,h,A,B} ( T^{-0.98}).
\end{align*}
By changing the implicit constant in $ \mathcal{O}_{\varphi,h,A,B} ( T^{-0.98})$, we can assume $T$ large enough, depending on $h$, in order to have
$$ A' \log \log T +  \mathcal{O}_h(1)   \geq B \log \log T,$$
$$ 10 \log \log T +  \mathcal{O}_h(1)   \geq 9\log \log T,$$
which then implies the result of the proposition. 
\end{proof}
It remains to estimate the probability involved in the last proposition. First, it is possible to split the probability into three pieces, involving sums on primes we discussed previously: 
\begin{proposition} \label{B1B2B3}
Let $B > 1$, $B_1, B_2, B_3 > 0$ such that $B_1 + B_2 + B_3  < B$. Let us take the notation of the previous proposition, and let us assume that $\varphi$ is even, with compactly supported Fourier transform. Then, under the Riemann hypothesis, 
and for any $R > 0$, $d \in [-2h,2h]$, 
\begin{align*}
& \mathbb{P} \left[  \left| \int_{-\infty}^{\infty} \Im \log \zeta(1/2 + i (UT + d + u H^{-1})) \varphi(u) du \right|
  \geq B \log \log T   \right] 
\\ & \leq P_1 + P_2 + P_3 + \mathcal{O}_{\varphi,h,B,B_1,B_2,B_3} (T^{-0.99}),
\end{align*}
where $$P_1 = \mathbb{P} \left[ \left|  \sum_{p \in \mathcal{P}, p \leq R } p^{-1/2 - i(UT + d)} \widehat{\varphi} \left( \frac{\log p}{H} \right) \right| \geq B_1 \log \log T \right],$$
$$P_2= \mathbb{P} \left[ \left| \sum_{p \in \mathcal{P},p >R} p^{-1/2 - i(UT + d)} \widehat{\varphi} \left( \frac{\log p}{H} \right) \right| \geq B_2 \log \log T \right],$$
$$P_3= \mathbb{P} \left[ \frac{1}{2} \left| \sum_{p \in \mathcal{P}} p^{-1 -2 i(UT + d)} \widehat{\varphi} \left( \frac{2 \log p}{H} \right) \right| \geq B_3 \log \log T \right].$$
\end{proposition}
\begin{proof}
By using Proposition \ref{sumsprimes} for $\sigma = 1/2$, we get
\begin{align*}
& \int_{-\infty}^{\infty} \Im \log \zeta(1/2 + i (UT + d + u H^{-1})) \varphi(u) du
\\ & = \Im \sum_{n \geq 1} \ell(n) n^{-1/2 - i(UT+d)} \widehat{\varphi} \left( \frac{\log n}{H} \right)  + \mathcal{O}_{\varphi} \left( 1 + \frac{e^{\mathcal{O}_{\varphi} (H)}}{1 + |UT + d|}  \right).
\end{align*}
Discarding and event of probability $\mathcal{O}_{h}(T^{-0.99})$, we can assume that $UT + d \geq T^{1/100}$, which implies, since $H = (\log T) (\log \log \log T)^{1/2}(\log \log T)^{-1}$, that the error term is $\mathcal{O}_{\varphi}(1)$ 
 for $T$ large enough 
depending on $\varphi$, and then smaller than $B' \log \log T$ for $T$ large enough depending on $\varphi, B, B_1, B_2, B_3$, where 
$$B' := \frac{B - (B_1 + B_2 + B_3)}{2} > 0.$$
 The last condition on $T$ can be assumed by changing the implicit constant in $\mathcal{O}_{\varphi,h,B,B_1,B_2,B_3} (T^{-0.99})$, and then it is enough to show: 
\begin{align*}
& \mathbb{P} \left[  \left| \Im \sum_{n \geq 1} \ell(n) n^{-1/2 - i(UT+d)} \widehat{\varphi} \left( \frac{\log n}{H} \right)  \right|
  \geq (B - B') \log \log T   \right] 
\\ & \leq P_1 + P_2 + P_3 + \mathcal{O}_{\varphi,h,B,B_1,B_2,B_3} (T^{-0.99}).
\end{align*}
Now, we have 
\begin{align*}
& \mathbb{P} \left[  \left| \Im \sum_{n \geq 1} \ell(n) n^{-1/2 - i(UT+d)} \widehat{\varphi} \left( \frac{\log n}{H} \right)  \right|
  \geq (B - B') \log \log T   \right] 
\\ & \leq  \mathbb{P} \left[  \left| \sum_{n \geq 1} \ell(n) n^{-1/2 - i(UT+d)} \widehat{\varphi} \left( \frac{\log n}{H} \right)  \right|
  \geq (B - B') \log \log T   \right] 
\\ & \leq \mathbb{P} \left[  \left| \sum_{n \geq 1, \ell(n) = 1, n \leq R} \ell(n) n^{-1/2 - i(UT+d)} \widehat{\varphi} \left( \frac{\log n}{H} \right)  \right|
  \geq B_1  \log \log T   \right] 
\\ & + \mathbb{P} \left[  \left| \sum_{n \geq 1, \ell(n) = 1, n > R} \ell(n) n^{-1/2 - i(UT+d)} \widehat{\varphi} \left( \frac{\log n}{H} \right)  \right|
  \geq B_2  \log \log T   \right] 
\\ & + \mathbb{P} \left[  \left| \sum_{n \geq 1, \ell(n) = 1/2} \ell(n) n^{-1/2 - i(UT+d)} \widehat{\varphi} \left( \frac{\log n}{H} \right)  \right|
  \geq B_3  \log \log T   \right] 
\\ & + \mathbb{P} \left[  \left| \sum_{n \geq 1, 0 <  \ell(n) < 1/2} \ell(n) n^{-1/2 - i(UT+d)} \widehat{\varphi} \left( \frac{\log n}{H} \right)  \right|
  \geq B_4 \log \log T   \right],
\end{align*}
for $$B_4  := B - B' - B_1 - B_2 - B_3 =  \frac{B - (B_1 + B_2 + B_3)}{2} > 0.$$
In this sum of four probabilities, the first one is $P_1$, the second is $P_2$, the third is $P_3$ and the fourth is equal to $0$ for $T$ large enough depending only on $B_4$, since 
$$ \left| \sum_{n \geq 1, 0 <  \ell(n) < 1/2} \ell(n) n^{-1/2 - i(UT+d)} \widehat{\varphi} \left( \frac{\log n}{H} \right)  \right|
\leq \sum_{r \geq 3} \sum_{p \in \mathcal{P}} p^{-r/2}$$
is uniformly bounded (recall that $|\widehat{\varphi}| \leq 1$ since $\varphi$ is nonnegative with integral $1$).
\end{proof}
We will now estimate the probabilities $P_1, P_2, P_3$. The main tool we will use is Lemma 
3 of  Soundararajan \cite{Sound}, which is  presented as a standard mean value estimate by the author (a similar result can be found in Lemma 3.3 of Tsang's thesis \cite{Tsang84}), and which is stated as follows: 
\begin{lemma} \label{Lemma3Sound}
For $T$ large enough and $2 \leq x \leq T$, for $k$ a natural number such that 
$x^k \leq T/\log T$, and for any complex numbers $a(p)$ indexed by the primes, 
we have 
$$\int_{T}^{2T} \left| \sum_{p \leq x, p \in \mathcal{P}} \frac{a(p)}{p^{1/2 + it}} \right|^{2k} dt
\ll T k! \left( \sum_{p \leq x, p \in \mathcal{P}} \frac{|a(p)|^2}{p} \right)^k.$$
\end{lemma} 
The quantity $P_1$ can then be estimated as follows: 
\begin{lemma} \label{P1}
With the notation above, for $R := e^{\log T/(2 \log \log T)}$ and $B_1 > 1$,
$$P_1 \ll_{B_1} (\log T)^{-1 - \log B_1}.$$
\end{lemma}
\begin{proof}
We apply Lemma \ref{Lemma3Sound}, with   $a(p) = p^{i(T-d)} \widehat{\varphi}(H^{-1}\log p)  $ and $x = R$. 

By Markov's inequality, we deduce, for $T$ large enough, and $k \geq 0$ integer such that $R^k \leq T/\log T$,
$$P_1 \ll (B_1 \log \log T)^{-2k} k! \left(\sum_{ p \in \mathcal{P}, p \leq R } p^{-1} \right)^k. $$
For $T$ large enough, it is possible to take $k  = \lfloor \log \log T \rfloor$, and then 
\begin{align*}
P_1 &  \ll  (B_1 \log \log T)^{-2\log \log T + 2} \left( e^{-1} \log \log T \right)^{\log \log T} \sqrt{ \log \log T}  \left( \log \log T + \mathcal{O}(1) \right)^{\log \log T}.
\\ & \ll_{B_1} (\log T)^{ -  \log B_1 - 1}.
\end{align*} 
By changing the implicit constant, we can remove the assumption that $T$ is large. 
\end{proof}
\begin{lemma} \label{P2}
With the notation above, for $R := e^{\log T/(2 \log \log T)}$ and $B_2 > 0$,
$$P_2 \ll_{\varphi, B_2} (\log T)^{-10}.$$
\end{lemma}
\begin{proof}
We apply again Lemma \ref{Lemma3Sound}, now with $a(p) = p^{i(T-d)} \widehat{\varphi}(H^{-1}\log p)  \mathds{1}_{p > R}$, and $x$ equal to the smallest integer such that $a(p) = 0$ for all $p > x$. Since $\widehat{\varphi}$ is compactly  supported, one has $x = e^{\mathcal{O}_{\varphi}(H)}$.  For $T$ large enough depending on $\varphi$, we can then apply the lemma for $k = \lfloor 100  \log \log T (\log \log \log T)^{-1} \rfloor$, since under this assumption, 
$$( e^{\mathcal{O}_{\varphi}(H)} )^{k} \leq e^{\mathcal{O}_{\varphi}((\log \log \log T)^{-1/2} \log T ) } \leq \frac{T}{\log T}.$$
For $T$ large enough depending on $\varphi$, we have $k \leq \log \log T$, and 
$$\sum_{p \in \mathcal{P}} \frac{|a(p)|^2}{p} \leq \sum_{R < p \leq e^{\mathcal{O}_{\varphi}(H)}} p^{-1} \leq  \sum_{ e^{\log T/(2 \log \log T)} < p \leq T} p^{-1} \ll \log \log \log T.$$
Hence, 
\begin{align*}
P_2 & \ll  (B_2 \log \log T)^{-2k} k^k (\mathcal{O} (\log \log \log T))^k 
\leq (\log T)^{-10}, 
\end{align*}
for $T$ large enough depending on $\varphi$ and $B_2$. This proves the lemma. 
\end{proof}
\begin{lemma} \label{P3}
With the notation above, for  $B_3 > 0$,
$$P_3 \ll_{\varphi, B_3} (\log T)^{-10}.$$
\end{lemma}
\begin{proof}
We use Lemma \ref{Lemma3Sound} for $2T$ instead of $T$, and 
$a(p) = \frac{1}{2} 
p^{-1/2 + 2i(T-d)} \widehat{\varphi}(2 H^{-1}  \log p)$. Again, $a(p) = 0$ for 
$p \geq  e^{\mathcal{O}_{\varphi} (H)}$, and then we can again take $x =  e^{\mathcal{O}_{\varphi} (H)}$, and 
$k = \lfloor 100 \log \log T (\log \log \log T)^{-1} \rfloor$ for $T$ large enough depending on $\varphi$. 
Moreover, 
$$\sum_{p \in \mathcal{P}} \frac{|a(p)|^2}{p} \leq \sum_{p \in \mathcal{P}} p^{-2}
\ll 1 \ll \log \log \log T,$$
which implies that an exactly similar computation as in the previous proof gives  the result of the lemma. 
\end{proof}
We deduce the leading order of an upper bound  for $\Im \log \zeta$: 
\begin{proposition} \label{upperboundimaginary}
Let us assume the Riemann hypothesis. Then, 
for all $A > 1$, $h > 0$, $U$ random, uniform on $[0,1]$, 
$$\mathbb{P} \left[  \underset{t \in [UT -  h, UT +  h]}{\sup} \left| \Im \log \zeta (1/2 + it) \right|  \geq A\log \log T \right] \underset{T \rightarrow \infty}{\longrightarrow} 0.$$
\end{proposition}
\begin{proof}
We choose arbitrarily a function $\varphi$ satisfying all the assumptions given in this section, which is possible by Proposition \ref{existence}. 
By Proposition \ref{B1B2B3}, Lemmas \ref{P1}, \ref{P2}, \ref{P3}, we have for all 
$B > 1$, $B_1 > 1$, $B_2, B_3 > 0$ such that $B_1 + B_2 + B_3 < B$, 
\begin{align*}
& \underset{d \in [-2h,2h]}{\sup} \mathbb{P} \left[  \left| \int_{-\infty}^{\infty} \Im \log \zeta(1/2 + i (UT + d + u H^{-1})) \varphi(u) du \right| \geq B \log \log T 
\right] 
\\ & \ll_{\varphi,h,B,B_1,B_2,B_3}  (\log T)^{-1 - \log B_1} + (\log T)^{-10} 
+ T^{-0.99}
\end{align*}
By taking $B_1 = \sqrt{B}$, $B_2 = B_3 = 
(B- \sqrt{B})/3$, we deduce, for $1 < B \leq 100$, 
\begin{align*}
& \underset{d \in [-2h,2h]}{\sup} \mathbb{P} \left[  \left| \int_{-\infty}^{\infty} \Im \log \zeta(1/2 + i (UT + d + u H^{-1})) \varphi(u) du \right| \geq B \log \log T 
\right] 
\\ & \ll_{\varphi,h,B}  (\log T)^{-1 - (1/2) \log B},
\end{align*}
and in particular, 
\begin{align*}
& \underset{d \in [-2h,2h]}{\sup} \mathbb{P} \left[  \left| \int_{-\infty}^{\infty} \Im \log \zeta(1/2 + i (UT + d + u H^{-1})) \varphi(u) du \right| \geq 9 \log \log T 
\right] 
\\ & \ll_{\varphi,h,B}  (\log T)^{-1 - (1/2) \log 9} \ll (\log T)^{-2.09}.
\end{align*}
We then conclude the proof of the present proposition by taking (say) $B = \min(\sqrt{A}, 100)$ and by applying Proposition \ref{B9}.

\end{proof}

\section{A lower bound of the supremum in term of sums on primes} 
 \label{section:sum}
In this section, the quantity we will bound from below is the supremum of the positive part of 
$\Re (\kappa  \log \zeta (1/2 + i\tau))$, for $\tau \in [UT - h, UT + h]$, $\kappa \in \{1,i,-i\}$, $U$  uniform  on $[0,1]$. The parameter $\kappa$ is used to 
deal with the real part and the imaginary part of $\log \zeta$ at the same time. 
The facts we use are the following: on the one hand, averaging $\log \zeta$ makes this quantity smoother and then tends to decrease its supremum, on the other hand, by Proposition \ref{sumsprimes}, it gives something related to finite sums over primes. The goal of this section is to prove the following proposition: 
\begin{proposition} \label{comparison}
Let us assume the Riemann hypothesis. Then, for $T > 10$, $h  >0$, $U$ uniform on $[0,1]$, $H$ integer such that  $(\log (3+T))^{1/10} \leq H \leq \frac{\log T}{(\log \log T)^2}$, we have, with probability tending to $1$ when $T$ goes to infinity: 
\begin{align*}
& \underset{\tau \in [UT-h,UT+h]}{\sup} \left( \Re \left( \kappa \log 
\zeta \left( \frac{1}{2} + i \tau \right) \right) \right)_+
\\ & \geq \sup_{k \in \{0, 1, \dots, H-1\}}
\Re \left( \kappa \sum_{p \in \mathcal{P} \cap [1,e^H]}
p^{-\frac{1}{2} - i (TU - h/2 + kh/H)} \right)
 + \mathcal{O}_h (\sqrt{\log \log T}).
\end{align*}
\end{proposition}
In this proposition, the fact that $H$ is small with respect to $\log T$ is crucial, in order to estimate enough moments of the sums on primes. In the next sections, we will fix a precise value of $H$, which will be the integer part of  $(\log T)^{1-\delta}$, where $\delta > 0$ is a suitably chosen small positive parameter. The assumptions of Proposition   \ref{comparison} will then be satisfied for $T$ large enough. 

The first step of the proof of Proposition \ref{comparison} consists in bounding the left-hand side from below by a series which is similar to the right-hand side, but with smooth cutoff 
instead of sharp cutoff at $e^H$. Such smooth cutoff is naturally obtained by using Proposition \ref{sumsprimes}. 

\begin{proposition}\label{comparisonrandom}
Let us assume the Riemann hypothesis. Let 
$\varphi$ be a  real, nonnegative, even function, dominated by any negative power at infinity, and such that its Fourier transform $\psi := \widehat{\varphi}$ is compactly supported, takes values in $[0,1]$, is even and equal to $1$ at zero. 
For $H > 1$, $\tau \in \mathbb{R}$,  
let us define: 
$$\Lambda_{\psi} 
(\tau, H) := 
\sum_{n \geq 1} \ell(n) n^{-\frac{1}{2}- i \tau}
\psi\left( \frac{\log n}{H} \right).$$
Then, for $\kappa \in \{1,-i,i\}$, $h > 0$, $t \in \mathbb{R}$, $A >0$, 
we have 
\begin{align*}
& \sup_{\tau \in [t - h, 
t + h]} \left(\Re \left(\kappa \log \zeta
\left(\frac{1}{2} + i \tau \right) \right)
\right)_+
\\ & \geq \sup_{\tau \in [t - (h/2), 
t + (h/2)]} \Re(\kappa \Lambda_{\psi} 
(\tau, H))
+ \mathcal{O}_{\varphi,A,h} \left( 
1 + H^{-A} \log (2 + |t|) + 
\frac{e^{\mathcal{O}_{\varphi} (H)}}{1 + |t|} \right).
\end{align*}
Moreover, there exists $\alpha >0 $, depending only on $\varphi$, such that if we assume the extra condition
$(\log (3 + |t|))^{1/10} \leq H \leq 
\alpha \log (3 + |t|)$,  we have 
$$ \sup_{\tau \in [t - h, 
t + h]} \left(\Re \left( \kappa  \log \zeta
\left(\frac{1}{2} + i \tau \right) 
\right)
\right)_+
 \geq \sup_{\tau \in [t - (h/2), 
t + (h/2)]} \Re (\kappa \Lambda_{\psi} 
(\tau, H))
+ \mathcal{O}_{\varphi,h} (1).$$
Finally, let us assume that  $h > 0$, $T > 0$,  $H
\geq (\log (3+T))^{1/10}$ is an integer depending only on $T$, negligible with $\log T$ when $T$ goes to infinity, and $U$ is  a uniform variable on $[0,1]$. Then, with probability 
tending to $1$ when $T$ goes to infinity, we have 
\begin{align*}
& \sup_{\tau \in [UT - h, UT + h]}
\left(\Re \left( \kappa  \log \zeta
\left(\frac{1}{2} + i \tau \right)\right)
\right)_+
\\ & \geq \underset{k \in \{0,1, \dots, H-1\}}{\sup} \Re (\kappa \Lambda_{\psi} (UT - h/2
+ kh/H,H))
+ \mathcal{O}_{\varphi,h} (1).
\end{align*}
\end{proposition}
\begin{proof}
By applying Proposition \ref{sumsprimes}, for $\sigma = 1/2$, we get 
\begin{align*}
\Re \left(\kappa \sum_{n \geq 1} \ell(n) n^{-1/2 - i\tau}\widehat{ \varphi}\left( \frac{\log n}{H} \right)  \right)
& = \int_{-\infty}^{\infty} \Re \left( \kappa
\log   \zeta(1/2 + i (\tau + tH^{-1})) \right)
\varphi(t) dt \\ &   + \mathcal{O}_{\varphi}
\left(1 + \frac{e^{\mathcal{O}_{\varphi} (H)}}{1+ |\tau|} \right).
\end{align*}
Since $\varphi$ is nonnegative, 
\begin{align*}
& \int_{-\infty}^{\infty} 
\Re \left(\kappa \log \zeta(1/2 + i (\tau + tH^{-1})) \right)
\varphi(t) dt 
\\ & \leq \left( \int_{-hH/2}^{hH/2} \varphi(t) dt \right)
\left( \underset{ u \in [\tau - h/2, \tau + h/2] }{\sup} \Re \left(\kappa \log (\zeta (1/2 + i u) \right) \right)
 \\&  + \int_{\mathbb{R} \backslash 
 [-h H/2, hH/2]}
\Re \left( \kappa  \log  \zeta(1/2 + i (\tau + tH^{-1})) \right)
\varphi(t) dt.
\end{align*}
Now, 
$(\log |\zeta (1/2 + i \tau)|)_+ = \mathcal{O}(\log (2+|\tau|))$ (for example, from Theorem 4.11 of \cite{Tit}). On the other hand, from Theorem 9.2 and Theorem 9.6 (B) of \cite{Tit}, 
$|\Im \log \zeta (1/2 + i \tau)| =
\mathcal{O}(\log (2+|\tau|))$.
Hence, the positive part of the last integral is dominated  by 
\begin{align*}
&\int_{\mathbb{R} \backslash 
 [-h H/2, hH/2]} \log (2+ |\tau| + |t| H^{-1}) \varphi(t)
 \leq 
 \int_{\mathbb{R} \backslash 
 [-h H/2, hH/2]} \log ((2+ |\tau|)
 (2 + |t|)) \varphi(t)
  \\ & \leq \log (2+ |\tau|) 
  \int_{\mathbb{R} \backslash 
 [-h H/2, hH/2]} \varphi 
 +  \int_{\mathbb{R} \backslash 
 [-h H/2, hH/2]} \log (2+|t|) \varphi(t) dt
 \\ & \ll_{A} (Hh)^{-A} \log (2 + |\tau|), 
\end{align*}
for all $A >0$. Here, in the first inequality, we used that $H > 1$, and in the last estimate, that $\varphi$ is rapidly decaying at infinity. 

Hence, 
\begin{align*}
& \int_{-\infty}^{\infty} 
\Re \left(\kappa \log \zeta(1/2+ i (\tau + tH^{-1})) \right)
\varphi(t) dt 
\\ & \leq \left( 1 - \mathcal{O}_A ((Hh)^{-A})
\right) 
 \underset{ u \in [\tau - h/2, \tau + h/2] }{\sup} \left( \Re \left( \kappa \log \zeta(1/2 + iu)  \right) \right)_+
   + \mathcal{O}_A \left( (Hh)^{-A}
 \log (2 + |\tau|) \right).
 \\ & \leq \underset{ u \in [\tau - h/2, \tau + h/2] }{\sup}  \left( \Re \left( \kappa \log \zeta(1/2 + iu)  \right) \right)_+
 +  \mathcal{O}_A \left( (Hh)^{-A}
 \log (2 + |\tau| + h) \right),
\end{align*}
and 
\begin{align*}
&\Re \left(\kappa \sum_{n \geq 1} \ell(n) n^{-1/2 - i\tau} \widehat{\varphi}\left( \frac{\log n}{H} \right)  \right)
\\ & \leq  \underset{ u \in [\tau - h/2, \tau + h/2] }{\sup} \left(\Re \left(\kappa \log \zeta (1/2 + i u) \right) \right)_+
 +  \mathcal{O}_{\varphi,A,h} \left( 1 + H^{-A} 
 \log (2 + |\tau|) + \frac{e^{\mathcal{O}_{\varphi}(H)}}{1 + |\tau|} \right).
\end{align*}
Taking the supremum over $\tau \in [t - h/2, t + h/2]$, which gives 
$$1 + |\tau| \gg_h 1 + |t|, 
\log (2 + |\tau|) \ll_h \log (2+ |t|),$$
we deduce 
\begin{align*}
&\underset{\tau \in [t-h/2, t+h/2]}{\sup}
\Re \left( \kappa \sum_{n \geq 1} \ell(n) n^{-1/2 - i\tau} \psi\left( \frac{\log n}{H} \right) \right) \\ &  \leq 
 \underset{ u \in [t  - h,t + h] }{\sup} \left(\Re \left( \kappa \log \zeta (1/2 + i u)\right) \right)_+
 + \mathcal{O}_{\varphi,A,h} \left( 1 + H^{-A} 
 \log (2 + |t|) + \frac{e^{\mathcal{O}_{\varphi}(H)}}{1 + |t|} \right)
 \end{align*}
This gives the first claim of the proposition.  For the second claim, we have, by assumption, $H \geq (\log 3)^{1/10} > 1$, and then we can apply the previous proposition, with $A = 10$. 
In the error term, 
$$H^{-A} \log (2 + |t|) \leq 
(\log (3 + |t|))^{-1} \log (2 + |t|) 
\leq 1,$$
and if we take $\alpha$ (depending only on $\varphi$) such that 
$\mathcal{O}_{\varphi}(H) \leq H/\alpha$, 
we get 
$$\frac{e^{\mathcal{O}_{\varphi}(H)}}{1 + |t|} \leq \frac{e^{\log (3 + |t|)}}{1 + |t|}
\leq 3,$$
which implies that the error term is 
$\mathcal{O}_{\varphi, h}(1)$. 

For the last claim of the proposition, we observe that  $H \geq (\log (3 + UT))^{1/10}$, and that for $T$ large enough, $H \leq (\alpha/2) \log (3+T)$ ($\alpha$ being the value associated to the second claim of the proposition), which implies 
\begin{align*}
\mathbb{P} [H \geq \alpha \log (3 + UT)]
 &\leq \mathbb{P} [(\alpha/2) \log (3+T)
\geq  \alpha \log (3+UT)] 
\leq \mathbb{P} [3+T \geq (3+UT)^2]
\\ & \leq \mathbb{P} [3 + T \geq 9 + U^2 T^2 ] \leq \mathbb{P} [U \leq 1/\sqrt{T}]
\underset{T \rightarrow \infty}{\longrightarrow} 0.
\end{align*}
Hence, we can apply the second estimate of the proposition to $t = TU$ with probability tending to $1$ when $T$ goes to infinity. 
By restricting the supremum in the right-hand side of the minorization, we are done. 
\end{proof}
The expression of $\Re (\kappa \Lambda_{\psi})$ is a sum indexed by the powers of primes. The following result shows that with high probability, one can 
get rid of the powers  with exponent strictly larger than $1$: 
\begin{proposition} \label{ridofpowers}
With the notation of the previous proposition, and under the extra condition 
$H \leq \frac{\log T}{(\log \log T)^2}$,  we have with probability tending to $1$ when $T$ goes to infinity: 
\begin{align*}
& \sup_{k \in \{0, 1, \dots, H-1\}} 
\left| \Lambda_{\psi} (TU - h/2 + kh/H, H)
- \sum_{p \in \mathcal{P}} 
p^{-\frac{1}{2} - i (TU - h/2 + kh/H)} 
\psi \left( \frac{\log p}{H} \right) \right| \\ & = \mathcal{O} (\sqrt{\log \log T}).
\end{align*}
\end{proposition}
\begin{proof}
It is sufficient to check that with probability tending to $1$, 
$$\sup_{k \in \{0,1, \dots, H-1\}}
\left| \sum_{r \geq 3} \sum_{p \in \mathcal{P}} 
\frac{1}{r} p^{- \frac{r}{2} - ir (UT - h/2 + kh/H)} 
\psi \left( \frac{r \log p}{H} \right)
\right|
 = \mathcal{O} (\sqrt{\log \log  T}), $$
 $$\sup_{k \in \{0,1, \dots, H-1\}}
\left| \sum_{p \in \mathcal{P}} 
\frac{1}{2} p^{- 1 - 2i (UT - h/2 + kh/H)} 
\psi \left( \frac{2 \log p}{H} \right)
\right|
 = \mathcal{O} (\sqrt{\log \log T}). $$
 The first supremum is uniformly bounded by
 the universal constant
 $$\sum_{r \geq 3} \sum_{n \geq 2} 
 n^{-r/2} < \infty,$$
 it is then sufficient to bound the second supremum. For $r \geq 0$ integer, the $2r$-th moment of the quantity inside the supremum is equal to 
 $$2^{-2r}  
 \sum_{p_1, \dots, p_r, q_1, \dots, q_r \in \mathcal{P}}
 \prod_{j=1}^r (p_jq_j)^{-1}\psi \left( \frac{2 \log p_j}{H} \right) \psi \left( \frac{2 \log q_j}{H} \right) 
  \int_{0}^{1}
 \left( \frac{p_1 \dots p_r}{q_1 \dots q_r} \right)^{-2i(uT - h/2 + kh/H)} du. $$
The sum of the terms such that 
$$\prod_{j =1}^r 
 p_j = \prod_{j=1}^r 
 q_j$$ 
 is equal to 
 $$\mathbb{E} \left[ \left|
 \sum_{p \in \mathcal{P}} \frac{X_p}{2p} \psi \left( \frac{2 \log p}{H} \right) \right|^{2r} \right],$$
 where $(X_p)_{p \in \mathcal{P}}$ are i.i.d., uniform on the unit circle. 
 For any other term, the integral above between $0$ and $1$ is at most, in absolute value, 
 $$\left( T \left| \log \left(\frac{p_1 
 \dots p_r}{q_1 \dots q_r} \right) \right| 
 \right)^{-1} = T^{-1} 
 \left| \int_{p_1 \dots p_r}^{q_1 \dots q_r} \frac{dx}{x} \right|^{-1} 
 \leq \frac{\max(p_1 \dots p_r, q_1 \dots q_r)}{T},$$
 since the two bounds of the integral are two distinct integers, which implies that the length of the interval of integration is at least $1$. 
 We deduce that the sum of the terms for which 
 $p_1 \dots p_r \neq q_1 \dots q_r$ is at most, in absolute value: 
 $$2^{1-2r} T^{-1}
 \sum_{p_1, \dots, p_r, q_1 \dots q_r \in \mathcal{P} \cap [1, e^{AH/2}], p_1 \dots p_r < q_1 \dots q_r} (p_1 \dots p_r)^{-1},$$
 if $\psi$ is supported in $[-A,A]$ (recall that $0 \leq \psi \leq 1$).
 The last sum is bounded by 
 $$\left(\sum_{p \in \mathcal{P} \cap [1, e^{AH/2}]}
 p^{-1} \right)^r
 \left(\sum_{p \in \mathcal{P} \cap [1, e^{AH/2}]}
1 \right)^r = [\mathcal{O}(\log  (2 + AH))]^r
\left(\mathcal{O} \left( \frac{e^{AH/2}}{1+AH} \right)\right)^r \leq B^r e^{rAH/2},
$$
 where $B > 1$ is a universal constant. 
 We then get, for $r \geq 1$, 
 \begin{align*}
 \mathbb{E} \left[\left| \sum_{p \in \mathcal{P}} 
\frac{1}{2} p^{- 1 - 2i (UT - h/2 + kh/H)} 
\psi \left( \frac{2 \log p}{H} \right)
\right|^{2r} \right] 
\leq \frac{B^r e^{rAH/2}}{T}  + \mathbb{E} \left[ \left|
 \sum_{p \in \mathcal{P}} \frac{X_p}{2p} \psi \left( \frac{2 \log p}{H} \right) \right|^{2r} \right].
 \end{align*}
 By crudely bounding  each term 
 of the sum on $p \in \mathcal{P}$  by its absolute value, we also get 
 $$\mathbb{E} \left[\left| \sum_{p \in \mathcal{P}} 
\frac{1}{2} p^{- 1 - 2i (UT - h/2 + kh/H)} 
\psi \left( \frac{2 \log p}{H} \right)
\right|^{2r} \right] 
\leq \left( \sum_{p \in \mathcal{P} \cap [1, e^{AH/2}]} 
p^{-1} \right)^{2r} 
\leq (B' \log(2+ AH))^{2r} ,$$
where $B' >1$ is universal. 
By summing the hyperbolic cosine series, we obtain, separating the cases 
$B^r e^{rAH/2} \leq T$ and $B^r e^{rAH/2} > T$, for $\lambda > 0$,  
\begin{align*}
&\mathbb{E} \left[ 
\cosh \left( \lambda \left| \sum_{p \in \mathcal{P}} 
\frac{1}{2} p^{- 1 - 2i (UT - h/2 + kh/H)} 
\psi \left( \frac{2 \log p}{H} \right)
\right| \right) \right]
\\ & \leq \sum_{0 \leq r \leq \frac{2\log T}
{AH + 2 \log B }} \frac{\lambda^{2r}}{(2r)!}
 + \sum_{r >  \frac{2\log T}
{AH + 2 \log B }} \frac{(\lambda B' \log (2+AH))^{2r}}{(2r)!} 
+ \mathbb{E} \left[ \cosh \left(\lambda
\left|\sum_{p \in \mathcal{P}} 
\frac{X_p}{2p} \psi \left(\frac{2 \log p}{H} \right) \right| \right) \right].
\end{align*}
The first sum is bounded by $\cosh \lambda \leq e^{\lambda}$. Now, by looking at the ratio between two consecutive terms and by bounding the first term with the Stirling formula, we deduce, for all $u > 0$, 
$$\sum_{r \geq 2u} \frac{u^{2r}}{(2r)!} 
\ll  1.$$
Hence, the second sum above is dominated by $1$, provided that 
$$\frac{2 \log T}{AH + 2 \log B} 
\geq 2 B' \lambda  \log (2+AH),$$
i.e.
$$B' \lambda (AH + 2 \log B) \log (2 +AH) 
\leq \log T.$$
Now, since we assume that $H \leq \frac{\log T}{(\log \log T)^2}$, we check that this condition is satisfied for $T$ large enough (depending on $\varphi$), if 
$\lambda \ll \sqrt{\log \log T}$. Finally,
from the inequality 
$$\cosh |z| \leq e^{|z|} \leq e^{|\Re(z)| + 
|\Im(z)|} 
\leq e^{2 \sup(|\Re(z)|, |\Im(z)|)}
\leq e^{2 \Re(z)} + e^{- 2 \Re(z)} +
e^{2 \Im(z)} + e^{-2 \Im(z)},$$
 the rotation invariance and the symmetry of the law of $X_p$, we get
\begin{align*}
& \mathbb{E} \left[ \cosh \left(\lambda
\left|\sum_{p \in \mathcal{P}} 
\frac{X_p}{2p} \psi \left(\frac{2 \log p}{H} \right) \right| \right) \right]
 \leq 4 \, \mathbb{E} \left[ \exp \left( 2 \lambda
\Re \sum_{p \in \mathcal{P}} 
\frac{X_p}{2p} \psi \left(\frac{2 \log p}{H} \right)  \right) \right]
\\ & = 4 \, \prod_{p \in \mathcal{P} 
\cap [1,e^{AH/2}]} 
\mathbb{E} \left[ e^{ \frac{\lambda 
\psi(2 H^{-1}\log p)}{p} \Re X_p} \right]
= 4 \, \prod_{p \in \mathcal{P} 
\cap [1,e^{AH/2}]} 
\mathbb{E} \left[ \cosh \left( \frac{\lambda 
\psi(2 H^{-1}\log p)}{p} \Re X_p \right) \right]
\\ & \leq 4 \prod_{p \in \mathcal{P} 
\cap [1,e^{AH/2}]} \cosh (\lambda/p)
\leq 4 \prod_{p \in \mathcal{P} 
\cap [1,e^{AH/2}]} e^{\lambda^2 /2p^2}
\leq 4 e^{\lambda^2 \sum_{n \geq 2} n^{-2}/2} \leq 4 e^{\lambda^2}. 
\end{align*}
Hence, for $0 \leq \lambda \ll \sqrt{\log \log T}$ and $T$ large enough depending on $\varphi$, 
$$\mathbb{E} \left[ 
\cosh \left( \lambda \left| \sum_{p \in \mathcal{P}} 
\frac{1}{2} p^{- 1 - 2i (UT - h/2 + kh/H)} 
\psi \left( \frac{2 \log p}{H} \right)
\right| \right) \right] \ll  e^{\lambda^2}.
$$
and then 
$$\mathbb{E} \left[ 
\exp \left( \lambda \left| \sum_{p \in \mathcal{P}} 
\frac{1}{2} p^{- 1 - 2i (UT - h/2 + kh/H)} 
\psi \left( \frac{2 \log p}{H} \right)
\right| \right) \right] \ll  e^{\lambda^2}.
$$
Note that estimates of exponential moments of linear combinations of $X_p$ have are also intensively used in the paper 
by Arguin, Belius and Harper \cite{ABH15}, where randomized versions of $\zeta$ are considered. 

The probability that the sum inside the exponential is larger than $2 \sqrt{\log \log T}$ is then dominated by 
$$ e^{-2\lambda \sqrt{\log \log T} + \lambda^2} = e^{- \log \log T} = \frac{1}{\log T}, $$
by taking $\lambda = \sqrt{\log \log T}$. 
A union bound on $k$ gives 
$$\mathbb{P} \left[\sup_{k \in \{0,1, \dots, H-1\}}
\left| \sum_{p \in \mathcal{P}} 
\frac{1}{2} p^{- 1 - 2i (UT - h/2 + kh/H)} 
\psi \left( \frac{2 \log p}{H} \right)
\right| \geq 2 \sqrt{\log \log T} \right] = \mathcal{O} \left( \frac{H}{\log T} \right),$$
which tends to zero when $T$ goes to infinity. 
\end{proof}
In the  next result, we show that we can replace the smooth cutoff by a sharp cutoff, and then get rid of the function $\varphi$: 
\begin{proposition} \label{ridofphi}
Under the same condition as in the previous proposition (in particular $H \leq \log T (\log \log T)^{-2}$), we have with probability tending to $1$ when $T$ goes to infinity: 
\begin{align*}
& \sup_{k \in \{0, 1, \dots, H-1\}} 
\left| \sum_{p \in \mathcal{P}} 
p^{-\frac{1}{2} - i (TU - h/2 + kh/H)} 
\left( \mathds{1}_{p \leq e^H} - \psi \left( \frac{\log p}{H} \right) \right) \right|  = \mathcal{O}_{\varphi} (\sqrt{\log \log T}).
\end{align*}
\end{proposition}
\begin{proof}
As above, we first compute the moment of order $2r$ of the quantity inside the supremum, and we get:  $$ 
 \sum_{p_1, \dots, p_r, q_1, \dots, q_r \in \mathcal{P}}
 \prod_{j=1}^r (p_jq_j)^{-1/2}\chi \left( \frac{ \log p_j}{H} \right) \chi \left( \frac{ \log q_j}{H} \right) 
  \int_{0}^{1}
 \left( \frac{p_1 \dots p_r}{q_1 \dots q_r} \right)^{-i(uT - h/2 + kh/H)} du,$$
 where $\chi(x) := \mathds{1}_{|x| \leq 1} 
 - \psi(x)$. 
 The sum of the terms with $p_1 \dots p_r
 = q_1 \dots q_r$ is equal to 
 $$\mathbb{E} \left[ \left| \sum_{p \in \mathcal{P}} \frac{X_p}{\sqrt{p}} 
 \chi \left( \frac{\log p}{H} \right) \right|^{2r} \right].$$
 By majorizing the oscillating integral between $0$ and $1$ as in the previous proof, we get a majorization, in absolute value, of the sum of the terms with $p_1 \dots p_r \neq q_1 \dots q_r$, by 
 $$4 T^{-1} \sum_{p_1, \dots, p_r, 
 q_1 \dots q_r \in \mathcal{P} \cap [1, e^{A'H}], p_1 \dots p_r < q_1 \dots q_r}
 (p_1 \dots p_r)^{-1/2} (q_1 \dots q_r)^{1/2}$$
 where $A' = \max(A,1)$ (recall that $-1 \leq \chi \leq 1$ and that $\chi$ is supported in $[-A',A']$).
 Now, for $a \in \{-1/2,1/2\}$, the crude bound 
 $$\sum_{p \in \mathcal{P} \cap [1, e^{A'H}]} p^a \leq \sum_{n=1}^{\lfloor e^{A'H} \rfloor} n^a \leq \int_0^{ e^{A'H}+1} x^a dx  \leq \frac{(2e^{A'H})^{a+1}}{a+1} \leq 3 e^{(a+1)A'H},$$
 used $r$ times for $a = 1/2$ and $r$ times for $a = -1/2$, gives 
 \begin{align*}
 & \mathbb{E} \left[\left| \sum_{p \in \mathcal{P}} 
p^{-\frac{1}{2} - i (TU - h/2 + kh/H)} 
\left( \mathds{1}_{p \leq e^H} - \psi \left( \frac{\log p}{H} \right) \right) \right|^{2r} \right]
\\ & \leq \mathbb{E} \left[ \left| \sum_{p \in \mathcal{P}} \frac{X_p}{\sqrt{p}} 
 \chi \left( \frac{\log p}{H} \right) \right|^{2r} \right]
  + \frac{4 \cdot 3^{2r} e^{2rA'H}}{T}.
 \end{align*}
 Now, if we write 
 $$\Delta = \min \left(\left| \sum_{p \in \mathcal{P}} 
p^{-\frac{1}{2} - i (TU - h/2 + kh/H)} 
\left( \mathds{1}_{p \leq e^H} - \psi \left( \frac{\log p}{H} \right) \right) \right|, \log \log T \right),$$ 
 we have obviously $\mathbb{E} [\Delta^{2r}] \leq (\log \log T)^{2r}$, and then, by expanding the hyperbolic cosine, we get, for all $\lambda > 0$,  
 \begin{align*}\mathbb{E} 
 \left[ \cosh(\lambda \Delta) \right] 
&  \leq \mathbb{E} \left[ \cosh\left|\lambda \sum_{p \in \mathcal{P}} \frac{X_p}{\sqrt{p}} 
 \chi \left( \frac{\log p}{H} \right) \right| \right]
 + \sum_{0 \leq r \leq \log T/(2 A'H)}
 \frac{4 \cdot (3 \lambda)^{2r}}{(2r)!}
\\ &  + \sum_{r > \log T/(2A'H)} \frac{(\lambda \log \log T)^{2r}}{(2r)!}
\\ &  =  \mathbb{E} \left[ \cosh\left|\lambda \sum_{p \in \mathcal{P}} \frac{X_p}{\sqrt{p}} 
 \chi \left( \frac{\log p}{H} \right) \right| \right] + \mathcal{O} \left(e^{3\lambda} \right),
\end{align*}
provided that 
$$\frac{\log T}{2A' H} \geq 2 \lambda \log \log T.$$
For $\lambda \ll \sqrt{\log \log T}$, 
this inequality is true for $T$ large enough since we assume $H \leq \frac{\log T}{(\log \log T)^2}$. 
As in the previous proof, we get 
 \begin{align*}
& \mathbb{E} \left[ \cosh \left(\lambda
\left|\sum_{p \in \mathcal{P}} 
\frac{X_p}{\sqrt{p}} \chi \left(\frac{ \log p}{H} \right) \right| \right) \right]
 \leq 4 \, \mathbb{E} \left[ \exp \left( 2 \lambda
\Re \sum_{p \in \mathcal{P}} 
\frac{X_p}{\sqrt{p}} \chi \left(\frac{ \log p}{H} \right)  \right) \right]
\\ & = 4 \, \prod_{p \in \mathcal{P} 
\cap [1,e^{A'H}]} 
\mathbb{E} \left[ e^{ \frac{2 \lambda 
\chi( H^{-1}\log p)}{\sqrt{p}} \Re X_p} \right]
= 4 \, \prod_{p \in \mathcal{P} 
\cap [1,e^{A'H}]} 
\mathbb{E} \left[ \cosh \left(
\frac{2\lambda 
\chi( H^{-1}\log p)}{\sqrt{p}} \Re X_p \right) \right]
\\ & \leq 4 \prod_{p \in \mathcal{P} 
\cap [1,e^{A'H}]} \cosh (2 \lambda \chi( H^{-1}\log p)/\sqrt{p})
\leq 4 \prod_{p \in \mathcal{P} 
\cap [1,e^{A'H}]} e^{2 \lambda^2
\chi^2( H^{-1}\log p) /p}.
\end{align*}
Since $\chi$ is smooth on $[0,1]$ and equal to $0$ at $0$, we have 
$$|\chi(x)| \leq \int_{0}^x |\chi'(y)| dy
\leq x \sup_{[0,1)} |\psi'|
\ll_{\varphi} x$$
for $x \in [0,1]$. Of course, this estimate remains true for $x > 1$ since $|\chi| \leq 1$. Hence 
$$\sum_{p \in \mathcal{P} 
\cap [1,e^{A'H}]} \frac{\chi^2 (H^{-1} \log p)}{p} 
\ll_{\varphi} H^{-2} \sum_{p \in \mathcal{P} 
\cap [1,e^{A'H}]}  \frac{\log^2 p}{p}.$$
 If $\pi$ denotes the prime counting function, we get for $t > 1$, by using the prime number theorem at the third line (a weak form is sufficient here):
 \begin{align*}\sum_{p \in \mathcal{P} 
\cap [1,t]}  \frac{\log^2 p}{p}
 & = \int_{[1,t]} \frac{\log^2 x}{x} d \pi(x)
 \\ & = \left[ \frac{\log^2 x}{x}  \pi(x)
 \right]_1^t
 - \int_1^t \frac{2 \log x - \log^2 x}{x^2} 
 \pi(x) dx 
 \\ & \ll \log t + \int_{1}^t \frac{1 + \log x}{x} dx \ll \log^2 t.
 \end{align*}
 This estimates gives 
 $$\sum_{p \in \mathcal{P} 
\cap [1,e^{A'H}]} \frac{\chi^2 (H^{-1} \log p)}{p} \ll_{\varphi} 1,$$
$$ \mathbb{E} \left[ \cosh \left(\lambda
\left|\sum_{p \in \mathcal{P}} 
\frac{X_p}{\sqrt{p}} \chi \left(\frac{ \log p}{H} \right) \right| \right) \right]
\ll e^{\mathcal{O}_{\varphi}(\lambda^2)}
$$
and then 
$$\mathbb{E}[e^{\lambda \Delta}]
\ll \mathbb{E}[\cosh(\lambda \Delta)] 
\ll e^{\mathcal{O}_{\varphi} (\lambda^2)}.
$$
By taking $\lambda= \sqrt{\log \log T}$, we get, for $C > 0$, 
$$\mathbb{P}[\Delta \geq C \sqrt{\log \log T}] \ll e^{-(C - \mathcal{O}_{\varphi}
(1)) \log \log T}
 \ll \frac{1}{\log T},$$
 if $C$ is large enough depending only on $\varphi$. Since $C \sqrt{\log \log T}
 \leq \log \log T$ for $T$ large enough depending only on $\varphi$, we deduce, under these conditions, 
 $$\mathbb{P} 
\left[\left| \sum_{p \in \mathcal{P}} 
p^{-\frac{1}{2} - i (TU - h/2 + kh/H)} 
\left( \mathds{1}_{p \leq e^H} - \psi \left( \frac{\log p}{H} \right) \right) \right| \geq \mathcal{O}_{\varphi}
(\sqrt{\log \log T}) \right]
\ll \frac{1}{\log T}$$
 Taking the union bound on $k \in \{0,1, \dots, H-1\}$ gives the desired result, since 
 $H = o(\log T)$ for $T \rightarrow \infty$. 
\end{proof}
We can now easily finish the proof of Proposition \ref{comparison}. 

{\it Proof of Proposition \ref{comparison}:} 
 We first arbitrarily fix a function $\varphi$ satisfying the assumptions of 
Proposition \ref{comparisonrandom}, which is possible by Proposition \ref{existence}. Then, we combine Propositions \ref{comparisonrandom}, \ref{ridofpowers} and \ref{ridofphi}.

\section{Comparison with Gaussian variables}
 \label{section:gaussian}

From now, we fix  the following quantities: $h > 0$, $T > 10$, $\delta \in (0,1/2)$, $K \geq 2$ integer, 
$H := \lfloor(\log T)^{1- \delta} \rfloor$.
For $T$ large enough depending on $\delta$, 
Proposition \ref{comparison} applies, since 
$(\log (3+T))^{1/10} \leq H  \leq \frac{\log T}{(\log \log T)^2}$. 
We then define, for $m \in \{0, 1, \dots, K-1\}$, $k \in \{0,1, \dots, H-1\}$: 
$$S(k, m) 
:=\Re \left( \kappa \sum_{p \in \mathcal{P}
\cap (e^{e^{m \log H / K}}, 
e^{e^{(m+1) \log H/K}} ]} p^{-\frac{1}{2} 
- i (TU - h/2 + kh/H)} \right).$$
We now see that the main term in the right-hand side of the estimate in Proposition \ref{comparison} is (up to $\mathcal{O}(1)$ because of the term
indexed by $p=2 \leq e^{e^0}$) equal to the supremum, for $k \in \{0,1, \dots, H-1\}$, 
of the sum of $S(k,m)$ for $m \in \{0,1,\dots,K-1\}$.
 Hence, if we show that with high probability, there exists $k \in \{0,1, \dots, H-1\}$ such that all the
values of $S(k,m)$ ($m \in \{0,1,\dots,K-1\}$) are large, then we will deduce that with high probability, the supremums involved in Proposition \ref{comparison} are also large. This will give a lower bound for the supremum of $\Re (\kappa \log \zeta)$ on the segment $[1/2 + i(UT-h),1/2 + i(UT+h)]$. 

Note that this method of cutting the sums into a fixed, large number $K$ of pieces, is classical in the study of the leading order term of the maximum of log-correlated fields: it is called {\it coarse graining} by  Kistler in \cite{bib:Kistler}, and it is also used by Arguin, Belius and Bourgade in \cite{bib:ABB}. 

Let us recall that the random variable $S(k,m)$ implicitly depends  on 
$T$, $\delta$ (which together give $H$), $K$, $h$ and $\kappa$. We will also use a truncated version of $S(k,m)$, defined as follows: 
$$S_0(k,m) := \min( (\log T)^{\delta/3}, 
\max (-(\log T)^{\delta/3}, S(k,m))).$$
The reason of such a truncation is the following: we will need to consider exponential moments of $S(k,m)$, which can be 
expanded by using the usual exponential series, giving a series of moments of $S(k,m)$. Only moments of sufficiently small order can be controlled in an efficient way, so we need to truncate the exponential series. The error we make with this truncation is acceptable only if $S(k,m)$ is not too large, which is ensured if it is replaced by 
$S_0(k,m)$.

We will show that in a sense which is made precise, the variables $S_0(k,m)$ for $0 \leq k \leq H-1$,   $0 \leq m \leq K-1$ are not far from being the 
components of a Gaussian vector with a similar covariance structure. 

This  comparison with Gaussian variables, which is also classical in the study of log-correlated fields,  will be done in two steps. 

In the first step, we observe that the random phases $(p^{-iUT})_{p \in \mathcal{P}}$ 
tend in law to i.i.d. uniform random variables $(X_p)_{p \in \mathcal{P}}$ on the unit circle, in the sense of the finite-dimensional marginals. It is then natural 
to compare the variables $S(k,m)$ and $S_0(k,m)$ with the variables $V(k,m)$ defined by 
$$V(k,m) := \Re \left( \kappa  \sum_{p \in \mathcal{P}
\cap (e^{e^{m \log H / K}}, 
e^{e^{(m+1) \log H/K}} ]} X_p p^{-\frac{1}{2} 
- i (- h/2 + kh/H)} \right).$$ 
Indeed, we will show that the joint Fourier-Laplace transforms of $S_0(k,m)$ and $V(k,m)$ are close to each other, when they are taken at points 
whose modulus is not too large. For this purpose, we need to get a comparison between the variables $(p^{-iUT})_{p \in \mathcal{P}}$  and 
 $(X_p)_{p \in \mathcal{P}}$ which goes beyond the finite-dimensional marginals, and in particular, in moment computations,  it is essential to deal 
with sums only involving primes which are much smaller than $T$ (indeed, $e^H \leq e^{(\log T)^{1-\delta}} = T^{o(1)}$). 

The second step consists in a comparison between the variables  $V(k,m)$ and the components of a Gaussian vector.
A very similar  comparison is done by Arguin, Belius and Harper 
\cite{ABH15} in their study of the supremum of randomized versions of $\log |\zeta|$; we can also mention an earlier work by Kowalski and Nikeghbali \cite{KowNik} on a close topic. 
 In our comparison with Gaussian vectors, we use the independence 
of the variables $(X_p)_{p \in \mathcal{P}}$ in a crucial way. Indeed, if we remove the variables $V(k,0)$ which involve the small primes, and which needs a separate study, the variables $V(k,m)$, $1 \leq m \leq K-1$ involve sums of many independent variables with small variances.

Here is the main result obtained in the first step: 
\begin{proposition} \label{TV}
For all $\lambda_0, \dots, \lambda_{K-1}$, 
$\mu_0, \dots, \mu_{K-1}$, complex with modulus at most $(\log T)^{\delta/100}$, and for all $k, \ell \in \{0,1, \dots,H-1\}$, we have 
\begin{align*}\mathbb{E} 
\left[ \exp \left( \sum_{m=0}^{K-1} 
(\lambda_m S_0(k,m) + \mu_m S_0(\ell,m)) \right)
\right] & = \mathbb{E} 
\left[ \exp \left( \sum_{m=0}^{K-1} 
(\lambda_m V(k,m) + \mu_m V(\ell,m)) \right)
\right]\\ &  + \mathcal{O} 
\left( \exp (-(\log T)^{\delta/4}) \right),
\end{align*}
if $T$ is large enough depending only on 
$\delta$ and $K$. 

\end{proposition}
\begin{proof}
Let us denote: 
$$\mathcal{S} :=  \sum_{m=0}^{K-1} 
(\lambda_m S(k,m) + \mu_m S(\ell,m)),$$
$$\mathcal{S}_0 := \sum_{m=0}^{K-1} \left( 
\lambda_m S_0(k,m) + \mu_m S_0(\ell,m) \right) $$
and 
$$\mathcal{V} := \sum_{m=0}^{K-1} \left( 
\lambda_m V(k,m) + \mu_m V(\ell,m) \right).$$
We have to compare the exponential moments of $\mathcal{S}_0$ and $\mathcal{V}$. This will be done 
by expanding the exponential series and by comparing   $\mathbb{E}[ \mathcal{S}_0^r]$ and $\mathbb{E}[ \mathcal{V}^r]$ for positive integers $r$. In fact, these moments will both be compared with the corresponding moments of the non-truncated sum $\mathcal{S}$. 
If $M$ is the maximum of the modulii of 
the $\lambda_m$'s and the $\mu_m$'s, we get the following lemma: 
\begin{lemma} \label{lemmaSV}
For all integers $r \geq 1$, we have
$$|\mathbb{E} [\mathcal{S}^r] 
- \mathbb{E} [\mathcal{V}^r]| 
\ll \frac{(2M)^r  e^{3Hr/2}}{T}.$$
\end{lemma}
\begin{proof}
For $p \in \mathcal{P} \cap [3, e^H]$, let us denote $\alpha_p := \kappa \lambda_m$, 
$\beta_p  := \kappa \mu_m$, $\gamma_p 
:= \overline{\kappa} \lambda_m$, $\delta_p 
:= \overline{\kappa} \mu_m$ if
$p \in (e^{e^{m \log H/K}}, 
e^{e^{(m+1) \log H/K}} ]$. 
We have:
$$\mathcal{S} = \frac{1}{2} \sum_{p \in \mathcal{P}
\cap [3, e^H]} p^{-1/2} \left(  p^{ - i(TU - h/2)} (\alpha_p p^{-i kh/H} + \beta_p 
p^{-i \ell h/H} ) + 
p^{i(TU - h/2)} (\gamma_p p^{i kh/H} + \delta_p 
p^{i \ell h/H} )
\right) .$$
Expanding the $r$-th power, we get
\begin{align*}
\mathbb{E} [\mathcal{S}^r] 
& = 2^{-r} 
\sum_{p_1, \dots, p_r \in \mathcal{P} 
\cap [3, e^H]} P^{-1/2} 
\sum_{A \coprod B \coprod C \coprod D 
= \{1,\dots, r\}} \alpha_{P_A} 
P_A^{-i(-h/2 + kh/H)} \dots 
\\ & \dots \times \beta_{P_B} 
P_B^{-i(-h/2 + \ell h/H)}
\gamma_{P_C} P_C^{i(-h/2 + kh/H)}
\delta_{P_D} P_C^{i(-h/2 + \ell h/H)}
\int_0^1 \left(\frac{ P_C P_D }{P_A P_B} 
\right)^{i T u}   du,
\end{align*}
where, for $X \in \{A,B,C,D\}$, 
$$P := \prod_{j=1}^r p_j, 
\; P_X := \prod_{j \in X} p_j,$$
$$\alpha_{P_X} := \prod_{j \in X} 
\alpha_{p_j}, \; \beta_{P_X} := \prod_{j \in X} 
\beta_{p_j},$$
$$\gamma_{P_X} := \prod_{j \in X} 
\gamma_{p_j}, \; \delta_{P_X} := \prod_{j \in X} 
\delta_{p_j}.$$
If we add the terms for which $P_A P_B = P_C P_D$, we get exactly $\mathbb{E}[\mathcal{V}^r]$. 
We then need to bound the other terms. We have
$$\left|\int_{0}^{1} 
\left( \frac{P_C P_D}{P_A P_B} \right)^{iTu} du  \right|
\leq 2 \left( T \left| \log \left(\frac{P_C P_D}{P_A P_B} \right) \right| \right)^{-1} \leq  2 T^{-1} 
\max( P_A P_B, P_C P_D ) \leq \frac{2 P}{T}.$$
Since 
$$|\alpha_{P_A} \beta_{P_B} \alpha_{P_C} 
\beta_{P_D} | \leq M^{|A| + |B| + |C| + |D|} \leq M^r,$$ we deduce
\begin{align*}|\mathbb{E} [\mathcal{S}^r] 
- \mathbb{E} [\mathcal{V}^r]| 
&\leq 2^{-r}
\sum_{p_1, \dots, p_r \in \mathcal{P} \cap 
[3, e^H]} P^{-1/2} (4^r) (M^r) (2P/T)
\\ & \leq  2(2M)^r T^{-1} \sum_{p_1, \dots, p_r \in \mathcal{P} \cap 
[3, e^H]} P^{1/2}
\\ & \leq 2(2M)^r T^{-1} \left( \sum_{p \in \mathcal{P} \cap [3, e^H]} p^{1/2}
\right)^r
\leq \frac{2 (2M)^r  e^{3Hr/2}}{T},
\end{align*}
which proves the lemma. 
\end{proof}
Let us now  compare the moments of $\mathcal{S}$ with the moments of 
$\mathcal{S}_0$.
\begin{lemma}
For all integers $r \geq 1$, 
$$|\mathbb{E} [\mathcal{S}^r] 
- \mathbb{E} [\mathcal{S}_0^r]|  \ll (2M)^r K^{r+1} \exp ( - (\log T)^{\delta/3} ) \left( (r!)^{1/2} 
 (\log T)^{r \delta/12} + \frac{2^r e^{3Hr/2}}{\sqrt{T}} \right). $$
\end{lemma} 
\begin{proof}
We have 
\begin{align*} & |\mathbb{E}[\mathcal{S}^r] 
- \mathbb{E}[\mathcal{S}_0^r] |
 \leq \mathbb{E} [|\mathcal{S}^r - \mathcal{S}_0^r| \mathds{1}_{\mathcal{S} \neq \mathcal{S}_0}] \leq \mathbb{E} [(|\mathcal{S}|^r + |\mathcal{S}_0|^r) \mathds{1}_{\mathcal{S} \neq \mathcal{S}_0}]
\\ & \leq 
\sum_{m = 0}^{K-1} \mathbb{E} \left[(|\mathcal{S}|^r + |\mathcal{S}_0|^r)
(\mathds{1}_{|S(k,m)| \geq (\log T)^{\delta/3}}  + \mathds{1}_{|S(\ell,m)| \geq (\log T)^{\delta/3}} ) \right]
\\ & \leq \sum_{m = 0}^{K-1}\mathbb{E} \left[
\left( \left( M \sum_{m'= 0}^{K-1} (|S(k,m')| + |S
(\ell,m')|) \right)^r + 
\left( M \sum_{m'= 0}^{K-1} (|S_0(k,m')| + |S_0
(\ell,m')|) \right)^r \right) 
\dots \right. \\ & \left. \dots \times 
(\mathds{1}_{|S(k,m)| \geq (\log T)^{\delta/3}}  + \mathds{1}_{|S(\ell,m)| \geq (\log T)^{\delta/3}} ) \right]
\\ & \leq  M^r (2K)^{r-1} 
\sum_{0 \leq m, m' \leq K-1} 
\mathbb{E} \left[ 
(|S(k,m')|^r  + |S(\ell,m')|^r  + 
|S_0(k,m')|^r  + |S_0(\ell,m')|^r) 
\dots \right. \\ & \left. \dots \times 
(\mathds{1}_{|S(k,m)| \geq (\log T)^{\delta/3}}  + \mathds{1}_{|S(\ell,m)| \geq (\log T)^{\delta/3}} ) \right]
\\ & \leq   2 M^r (2K)^{r-1} 
\sum_{0 \leq m, m' \leq K-1} 
\mathbb{E} \left[ 
(|S(k,m')|^r  + |S(\ell,m')|^r) 
(\mathds{1}_{|S(k,m)| \geq (\log T)^{\delta/3}}  + \mathds{1}_{|S(\ell,m)| \geq (\log T)^{\delta/3}} ) \right]
\\ & = 2 M^r (2K)^{r-1} 
\sum_{j, j' \in \{1,2\}} 
\sum_{0 \leq m, m' \leq K-1} 
\mathbb{E} [\mathds{1}_{|S(v_j,m)| \geq 
(\log T)^{\delta/3}}  |S(v_{j'},m')|^r ] 
\end{align*}
where $v_1 := k$, $v_2 := \ell$. 

Hence, for $w \geq 1$ integer, 
\begin{align*}
& |\mathbb{E}[\mathcal{S}^r] 
- \mathbb{E}[\mathcal{S}_0^r] |
 \leq 2 (\log T)^{-w \delta/3} 
 M^r (2K)^{r-1}
 \sum_{j, j' \in \{1,2\}} 
\sum_{0 \leq m, m' \leq K-1} 
\mathbb{E} [ |S(v_j,m)|^w  |S(v_{j'},m')|^r 
] 
\\ &  
\leq 2 (\log T)^{-w\delta/3} 
  M^r (2K)^{r-1}
 \sum_{j, j' \in \{1,2\}} 
\sum_{0 \leq m, m' \leq K-1} 
\mathbb{E} [ (S(v_j,m))^{2w} ]^{1/2}  
\mathbb{E} [ (S(v_{j'},m'))^{2r} ]^{1/2}.
\end{align*}
Now, $\mathbb{E} [(S(k,m))^{2r}]$ corresponds to $\mathbb{E}[\mathcal{S}^{2r}]$ in the case where $\lambda_m = 1$, 
and all the other $\lambda_j$'s and $\mu_j$'s are zero (and then $M = 1$).  The previous lemma then implies: 
$$\mathbb{E} [(S(k,m))^{2r}]
= \mathbb{E} [(V(k,m))^{2r}] + \mathcal{O} 
\left( \frac{2^{2r}e^{3Hr}}{T} \right).$$
where 
\begin{align*}
\mathbb{E} \left[ (V(k,m))^{2r} \right] 
 & \leq \mathbb{E} \left[ 
\left| \sum_{p \in \mathcal{P} \cap 
 (e^{e^{m \log H/K}}, e^{e^{(m+1) \log H/K}}] }  \frac{X_p}{\sqrt{p}} \right|^{2r} 
 \right] 
 \\ & = \sum_{p_1, \dots, p_r, q_1, \dots, q_r \in \mathcal{P} \cap (e^{e^{m \log H/K}}, e^{e^{(m+1) \log H/K}}]} 
 \frac{\mathds{1}_{p_1 \dots p_r = q_1 \dots q_r}}{p_1 \dots p_r} 
 \\ & \leq r! \sum_{p_1, \dots, p_r \in 
 (e^{e^{m \log H/K}}, e^{e^{(m+1) \log H/K}}]} 
 \frac{1}{p_1 \dots p_r} \\ & 
 \leq r!\left( \sum_{p \in [3, e^H], p \in \mathcal{P}} \frac{1}{p}  \right)^r 
 =  r! \left( \log H  + \mathcal{O}(1) 
  \right)^r \leq r! (\log T)^{r\delta/6},
\end{align*}
if $T$ is large enough depending on $\delta$. Note that in the last line, we have used the classical estimate 
$$\sum_{p \leq x, p \in \mathcal{P}} \frac{1}{p} =  \log \log x + \mathcal{O}(1),$$
available for all $x \geq 2$ (it is a consequence of Mertens' second theorem, and it can also be deduced from  the prime number theorem). 
Now, we deduce
\begin{align*}
 &|\mathbb{E} [\mathcal{S}^r] 
- \mathbb{E} [\mathcal{S}_0^r] |
\\ & 
 \ll (\log T)^{-w\delta/3} 
M^r (2K)^{r-1}  K^2 \left( (w!)^{1/2} (\log T)^{w\delta/12} + \frac{2^{w} e^{3Hw/2}}{\sqrt{T}} \right)\left( (r!)^{1/2} (\log T)^{r\delta/12} + \frac{2^{r} e^{3Hr/2}}{\sqrt{T}} \right)
\\ & \ll (2M)^r K^{r+1}
\left( (w!)^{1/2} (\log T)^{-w\delta/4} + \frac{2^{w} e^{3Hw/2}}{\sqrt{T}} \right)
\left( (r!)^{1/2} (\log T)^{r\delta/12} + \frac{2^{r} e^{3Hr/2}}{\sqrt{T}} \right).
\end{align*}

We now take $w = \lfloor (\log T)^{\delta/3} \rfloor$, which implies 
\begin{align*}(w!)^{1/2} (\log T)^{- w \delta/4} 
& \leq [w^{1/2} (\log T)^{- \delta/4} ]^w
 \leq [(\log T)^{-\delta/12} ]^w \\ &  \leq \exp \left( - (\delta/12) (\log 
\log T )\lfloor(\log T)^{\delta/3} \rfloor
\right) 
\\ & \leq \exp ( - (\log T)^{\delta/3} ), 
\end{align*}
whereas 
$$\frac{2^w e^{3Hw/2}}{\sqrt{T}}  \leq 
T^{-1/2}  \exp \left( (\log T)^{\delta/3} 
(\log 2) + 3 (\log T)^{1 - (2\delta/3)}/2  
\right) \leq T^{-0.49},$$
if $T$ is large enough depending on $\delta$. 
Hence, 
$$|\mathbb{E} [\mathcal{S}^r] - \mathbb{E} 
[\mathcal{S}_0^r] |
 \ll (2M)^r K^{r+1} \exp ( - (\log T)^{\delta/3} ) \left( (r!)^{1/2} 
 (\log T)^{r \delta/12} + \frac{2^r e^{3Hr/2}}{\sqrt{T}} \right), $$
which shows the lemma.
\end{proof}
{\it End of the proof of Proposition \ref{TV}:} 
Combining the two lemmas, we get 
\begin{align*}
& |\mathbb{E} [\mathcal{S}_0^r] 
- \mathbb{E} [\mathcal{V}^r] |
\\ & \ll  (2M)^r K^{r+1} \exp ( - (\log T)^{\delta/3} ) \left( (r!)^{1/2} 
 (\log T)^{r \delta/12} + \frac{2^r e^{3Hr/2}}{\sqrt{T}} \right)
 + \frac{(2M)^r e^{3Hr/2}}{T}. 
\end{align*}
Now, under the assumption of the proposition, 
$M \leq (\log T)^{\delta/100}$. 
If $r \leq (\log T)^{\delta/4}$, 
\begin{align*}
& (2M)^r K^{r+1} \exp( - (\log T)^{\delta/3} ) (r!)^{-1/2} (\log T)^{r \delta/12}
\\ & \leq \left(2K^2 (\log T)^{\delta\left( \frac{1}{100} + \frac{1}{12} \right)} \right)^r
 \exp( - (\log T)^{\delta/3} )
 \\ & \leq \exp \left( (\log T )^{\delta/4}
 [ \log ( 2 K^2) + (7 \delta/75) \log 
   \log T] - (\log T )^{\delta/3} \right)
   \\ & \leq \exp \left( - \frac{1}{2}(\log T )^{\delta/3} \right)
 \end{align*} 
 if $T$ is large enough depending on $\delta$ and $K$. 
 If $r \geq (\log T)^{\delta/4}$, 
 \begin{align*}
& (2M)^r K^{r+1} \exp( - (\log T)^{\delta/3} ) (r!)^{-1/2} (\log T)^{r \delta/12}
\\ & \leq \left(2K^2 (\log T)^{\delta\left( \frac{1}{100} + \frac{1}{12} \right)} \right)^r
 \exp( - (\log T)^{\delta/3} )
 (e/r)^{r/2}
 \\ & \leq \left( 2  K^2 (\log T)^{7 \delta/75}
 e^{1/2} (\log T)^{- \delta/8}  \right)^r
 \exp \left( -(\log T )^{\delta/3} \right)
 \\ & \leq \left( 2  K^2
 e^{1/2} (\log T)^{- 19\delta/600}  \right)^r
 \exp \left( -(\log T )^{\delta/3} \right)
 \\ & \leq \exp \left( -(\log T )^{\delta/3} \right)
\end{align*}
for $T$ large enough depending on $\delta$ and $K$. 
Under the same assumptions, 
for $r \leq (\log T)^{\delta/2}$, 
\begin{align*}
& (2M)^r K^{r+1} \exp \left( - (\log T)^{\delta/3} \right)  \frac{2^r e^{3Hr/2}}{ r! \sqrt{T}} 
\\ & \leq  (2M)^r K^{2r} \frac{2^r e^{3Hr/2}}{\sqrt{T}} \\ &  \leq T^{-1/2} 
\exp \left( (\log T)^{\delta/2} 
\log ( 4K^2 (\log T)^{\delta/100} ) 
+ (3/2) (\log T)^{1 - (\delta/2)} ) \right)
\leq T^{-0.49}
\end{align*}
and 
$$ \frac{(2M)^r e^{3Hr/2}}{T r!} 
\leq T^{-1} \exp \left( (\log T)^{\delta/2}
\log (2M) + (3/2) (\log T)^{1 - (\delta/2)}
\right) \leq T^{-0.99}.$$
All these estimates imply that, for 
all $r \leq (\log T)^{\delta/2}$, $T$ large enough depending on $\delta$ and $K$, 
$$ \frac{1}{r!}\left| 
\mathbb{E} [\mathcal{S}_0^r] - 
\mathbb{E} [\mathcal{V}^r] \right|
\ll \exp \left( - \frac{1}{2} (\log T)^{\delta/3} \right).$$
On the other hand, by the truncation,
$$| \mathcal{S}_0| \leq 2 K M (\log T)^{\delta/3} 
\leq 2K (\log T)^{\delta \left( \frac{1}{3} 
+ \frac{1}{100}  \right)} 
\leq (\log T)^{0.35 \delta},$$
for $T$ large enough depending on $\delta$ and $K$, and then for $r \geq (\log T)^{\delta/2}$, 
$$\frac{1}{r!} \mathbb{E} [|\mathcal{S}_0|^r] 
\leq  (e/r)^r (\log T)^{0.35 r \delta}
\leq \left(e (\log T)^{-\delta/2} 
 (\log T)^{0.35 \delta} \right)^r
 \leq e^{-r},$$
 whereas, by using the previous estimate 
 of $\mathbb{E} [(V(k,m))^{2r}]$, 
 \begin{align*}
 \frac{1}{r!} \mathbb{E} [|\mathcal{V}|^r] 
 & = \frac{1}{r!} ||\mathcal{V}||_r^r 
 \leq \frac{1}{r!} 
 \left( M \sum_{m=0}^{K-1} 
 (||V(k,m)||_r + ||V(\ell,m)||_r )\right)^r
 \\ & \leq \frac{1}{r!}
 \left(  M\sum_{m=0}^{K-1} 
 (||V(k,m)||_{2r} + ||V(\ell,m)||_{2r} )\right)^r
 \\& \leq \frac{1}{r!}
 \left( 2 K M  (r!)^{1/2r} (\log T)^{\delta/12} 
 \right)^r 
 \\ & \leq (r!)^{-1/2} 
 (2KM (\log T)^{\delta/12} )^r 
 \\ & \leq (e/r)^{r/2}
  (2KM (\log T)^{\delta/12} )^r 
  \\ & \leq \left(2 e^{1/2} 
  ( \log T)^{-\delta/4}
  K (\log T)^{\delta/100}
  (\log T)^{\delta/12} \right)^r
  \\ & \leq e^{-r}.   
 \end{align*}
Expanding the exponential, we deduce 
\begin{align*}
& \left|\mathbb{E} [\exp (\mathcal{S}_0)] 
- \mathbb{E} [\exp (\mathcal{V})] 
\right| 
\leq \sum_{r \geq 1} 
\frac{1}{r!} \left| \mathbb{E} [\mathcal{S}_0^r] - \mathbb{E} [\mathcal{V}^r]
\right|
\\ & \ll \sum_{1 \leq r \leq (\log T)^{\delta/2} } \exp \left( - \frac{1}{2} 
(\log T)^{\delta/3} \right) 
+ \sum_{r > (\log T)^{\delta/2}} 
e^{-r}
\\ & \ll (\log T)^{\delta/2} 
\exp \left( - \frac{1}{2} 
(\log T)^{\delta/3} \right) 
+ e^{-  (\log T)^{\delta/2}} 
\\ & \ll \exp \left( - 
(\log T)^{\delta/4} \right),
\end{align*}
when $T$ is large enough depending on $\delta$ and $K$. 
\end{proof} 
The next step consists in comparing
the family $(V(k,m))_{k \in \{0, \dots, 
H-1\}, m \in \{0,1, \dots, K-1\}}$
 with a Gaussian family with a close  covariance structure. 
We have 
$$V(k,m) = \Re \sum_{n \in \mathbb{N} 
\cap (e^{e^{m \log H/K}}, e^{e^{(m+1) \log 
H/K}}] } Y_n n^{-ikh/H},$$
where 
$$Y_n :=
\kappa X_n n^{-1/2 + ih/2}  \mathds{1}_{n \in \mathcal{P}} .$$
 The covariance matrix of $Y_n$, seen as the  two-dimensional vector
$(\Re(Y_n), \Im(Y_n))$, is 
 $0$ if $n$ is not prime  and 
 $I_2/(2n)$ if $n$ is prime, $I_2$ being the two-dimensional identity matrix. If we "average the variance" with the prime number theorem, we get $I_2/(2 n \log n)$. Moreover, if we replace the sum by an integral, this leads us to compare $V(k,m)$ with the Gaussian variable: 
 $$G(k,m) := \Re \int_{e^{e^{m \log H/K}}}^{e^{e^{(m+1) \log H/K}}} \frac{t^{-ikh/H}}
 { \sqrt{2 t \log t}} d W_t,$$
 where $W_t := W^{(1)}_t + i  W^{(2)}_t$, 
 $W^{(1)}$ and $W^{(2)}$ being two independent standard Brownian motions. 
 We have the following result: 
 \begin{proposition} \label{VG}
For all $\lambda_1, \dots, \lambda_{K-1}$, 
$\mu_1, \dots, \mu_{K-1}$, complex with modulus at most $(\log T)^{1/20 K}$, and for all $k, \ell \in \{0,1, \dots,H-1\}$, we have 
\begin{align*}\mathbb{E} 
\left[ \exp \left( \sum_{m=1}^{K-1} 
(\lambda_m V(k,m) + \mu_m V(\ell,m)) \right)
\right] & = \mathbb{E} 
\left[ \exp \left( \sum_{m=1}^{K-1} 
(\lambda_m G(k,m) + \mu_m G(\ell,m)) \right)
\right]\\ &  + \mathcal{O}_h 
\left( \exp (-(\log T)^{1/10K}) \right),
\end{align*}
if $T$ is large enough depending only on $\delta$ and $K$. 
 \end{proposition}
 \begin{rmk}
 The proposition is not true if we add 
 terms involving $V(k,0)$ and $V(\ell,0)$, since  the small primes give variables which are not close to Gaussian ones. 
This will not be a major issue in the proof of our main result, but it may create some difficulties if we want to get finer aymptotics on the maximum of $\log \zeta$. 
The sums involving  small primes will be considered separately, at the end of this paper. 
 \end{rmk}
This proposition is very similar to results given by Arguin, Belius and Harper in \cite{ABH15}: see for example  Proposition 
2.4 of that paper for a comparison between variables similar to $V(k,m)$ and Gaussian variables with the same covariance structure, and Lemma 2.1 of \cite{ABH15} for an approximation of this covariance structure. However, the quantities involved in \cite{ABH15} and here are not exactly the same, and the covariance of the variables $G(k,m)$ has not an exact branching structure, contrarily to the approximation given in Lemma 2.1 of \cite{ABH15}. Moreover, in the present situation, we need to allow the parameters $\lambda_m$ and $\mu_m$ to grow with $T$, whereas the constant $C$ is fixed in 
proposition 2.4 of \cite{ABH15}. For these reasons and for sake of completeness, we provide a full proof of Proposition 
\ref{VG} here, even if the main arguments involved are essentially the same as in \cite{ABH15}. 
 \begin{proof}
 We have, using the independence of the 
 $Y_n$'s: 
 \begin{align*}
 &\mathbb{E} \left[ \exp \left( \lambda_m
 V(k,m) + \mu_m V(\ell,m) \right) \right]
 \\ & = \prod_{n \in \mathbb{N} 
 \cap (e^{e^{m \log H/K}}, e^{e^{(m+1) \log 
 H/K}}]} \mathbb{E} [e^{
\lambda_m  \Re (Y_n
 n^{-ikh/H})  + \mu_m  \Re (Y_n
 n^{-i \ell h/H})  }].
 \end{align*}
Now, we have, from the fact that 
$Y_n$ is rotationally invariant: 
$$\mathbb{E} [\lambda_m  \Re (Y_n
 n^{-ikh/H})  + \mu_m  \Re (Y_n
 n^{-i \ell h/H})  ] = 0,$$
 \begin{align*}
 &\mathbb{E} \left[\left(\lambda_m  \Re (Y_n
 n^{-ikh/H})  + \mu_m  \Re (Y_n
 n^{-i \ell h/H}) \right)^2 \right] 
\\ & = \frac{1}{4} \mathbb{E} \left[\left(Y_n
 (\lambda_m n^{-ikh/H}  +\mu_m n^{-i \ell h/H}) + \overline{Y_n}
  (\lambda_m n^{ikh/H}  +\mu_m n^{i \ell h/H}) \right)^2 \right]
  \\ & =  \frac{1}{2}  
  (\lambda_m n^{-ikh/H}  +\mu_m n^{-i \ell h/H}) (\lambda_m n^{ikh/H}  +\mu_m n^{i \ell h/H}) \mathbb{E}[|Y_n|^2] 
  \\ & = \frac{\mathds{1}_{n \in \mathcal{P}}}{2n} 
  (\lambda_m n^{-ikh/H}  +\mu_m n^{-i \ell h/H}) (\lambda_m n^{ikh/H}  +\mu_m n^{i \ell h/H})
  \\ & =  \frac{\mathds{1}_{n \in \mathcal{P}}}{2n} \left( \lambda_m^2 
  + \mu_m^2 + 2 \lambda_m \mu_m 
  \cos ((k-\ell) ( \log n) h /H) \right).
 \end{align*}
Moreover, for $m \geq 1$, we have
$$n \geq e^{e^{K^{-1}\log H}},$$
which implies 
\begin{align*}  \left|\lambda_m  \Re (Y_n
 n^{-ikh/H})  + \mu_m  \Re (Y_n
 n^{-i \ell h/H}) \right|
& \leq \frac{|\lambda_m| + |\mu_m|}{\sqrt{n}} \leq 2  (\log T)^{1/20K}
n^{-0.49} 
 e^{- \frac{1}{100} e^{K^{-1}\log H }}
 \\ & \ll  n^{-0.49} (\log T)^{1/20K}
  e^{- \frac{1}{100} (\log T)^{(1-\delta)/K}}
  \\ & \leq n^{-0.49}, 
 \end{align*}
 for $T$ large enough depending on $\delta$ and $K$.
 We deduce, by expanding the exponential: 
 \begin{align*}
 & \mathbb{E} [e^{
\lambda_m  \Re (Y_n
 n^{-ikh/H})  + \mu_m  \Re (Y_n
 n^{-i \ell h/H})  }]
\\&  = 1 + \frac{\mathds{1}_{n \in \mathcal{P}}}{4n} \left( \lambda_m^2 
  + \mu_m^2 + 2 \lambda_m \mu_m 
  \cos ((k-\ell) ( \log n) h /H) \right)
  + \mathcal{O} \left( \sum_{r \geq 3}
  \frac{n^{-0.49 r}}{r!} \right)
  \\ & = 1 + \frac{\mathds{1}_{n \in \mathcal{P}}}{4n} \left( \lambda_m^2 
  + \mu_m^2 + 2 \lambda_m \mu_m 
  \cos ((k-\ell) ( \log n) h /H) \right)
  + \mathcal{O} \left( n^{-1.47} \right)
 \end{align*}
 For $T$ large enough depending on $\delta$ and $K$, the second term of the last formula is 
 smaller than $n^{-0.98}$, which allows to take the logarithm and to use the estimate 
 $\log(1+x) = x + \mathcal{O}(x^2)$, available for $|x| \leq 1/2$: 
 \begin{align*}
 & \log \mathbb{E} [e^{
\lambda_m  \Re (Y_n
 n^{-ikh/H})  + \mu_m  \Re (Y_n
 n^{-i \ell h/H})  }]
\\ &  = \frac{\mathds{1}_{n \in \mathcal{P}}}{4n} \left( \lambda_m^2 
  + \mu_m^2 + 2 \lambda_m \mu_m 
  \cos ((k-\ell) ( \log n) h /H) \right)
  + \mathcal{O} \left( n^{-1.47} \right),
  \end{align*}
  and then 
   \begin{align*}
 &\mathbb{E} \left[ \exp \left( \lambda_m
 V(k,m) + \mu_m V(\ell,m) \right) \right]
 \\ & = \exp 
 \left( E +  \sum_{p \in \mathcal{P} \cap (e^{e^{m \log H/K}}, e^{e^{(m+1) \log 
 H/K}}]} 
 \frac{1}{4p} \left( \lambda_m^2 
  + \mu_m^2 + 2 \lambda_m \mu_m 
  \cos ((k-\ell) ( \log p) h /H) \right)
  \right),
 \end{align*}
 where 
   \begin{align*}
   |E|  &\ll \sum_{n \geq e^{e^{K^{-1} \log 
 H}}} n^{-1.47} 
 \ll \exp \left( - 0.47 H^{1/K} \right)
 \ll \exp \left( - 0.47(\log T)^{
 (1-\delta)/K} \right)
 \\ & \ll \exp \left( - (\log T)^{1/2K} \right).
 \end{align*}
 (recall that $\delta \in (0,1/2)$). 
 Hence, 
 \begin{align*}
 &\mathbb{E} \left[ \exp \left( \lambda_m
 V(k,m) + \mu_m V(\ell,m) \right) \right]
 \\ & = \exp 
 \left( \left( \Sigma (e^{e^{(m+1) \log H/K}})
 - \Sigma (e^{e^{m \log H/K}}) \right) 
 \left( \frac{\lambda_m^2 + \mu_m^2}{4}
 \right) \dots 
 \right. \\ & \left. \dots  + \left(\Sigma (e^{e^{(m+1) \log H/K}}, (k-\ell)h/H)
 - \Sigma (e^{e^{m \log H/K}}, (k-\ell)h/H)
 \right) \left( \frac{\lambda_m \mu_m}{2}
 \right) \dots \right. \\ & \left. \dots + \mathcal{O}
 \left( \exp \left( - (\log T)^{1/2K} \right) \right) \right),
  \end{align*}
 where 
 $$\Sigma(x) := \sum_{p \in \mathcal{P}
 \cap [2,x]} \frac{1}{p}, \; \;  \Sigma(x, \theta) := \sum_{p \in \mathcal{P}
 \cap [2,x]} \frac{\cos(\theta \log p)}{p}.$$
 For $2 \leq  x < y$, let us write 
 $$I(x,y, \theta) := 
 \int_{x}^y \frac{\cos( \theta \log t)}{t \log t} dt,$$
 $$I(x,y) := I(x,y,0) = \log \log y - \log \log x.$$
 Integrating by parts, we get: 
 \begin{align*}
 & \Sigma(y,\theta) - \Sigma(x,\theta) 
 = \int_{(x,y]} \frac{\partial}{\partial t} 
 I(x,t,\theta) d \rho(t) 
 \\ & = I(x, y,  \theta) 
 +  \int_{(x,y]} \frac{\partial}{\partial t} 
 I(x,t,\theta) d (\rho(t) - t) 
 \\ & = I(x, y,  \theta) 
 + \left[ \frac{\partial}{\partial t} 
 I(x,t,\theta) (\rho (t) - t) \right]_{x+}^{y+} - \int_{x}^{y} 
 \frac{\partial^2}{\partial t^2} 
 I(x,t,\theta) (\rho (t) - t) dt,
 \end{align*}
 where, for $t \in (x,y]$,  $$\rho(t) := \sum_{p \in \mathcal{P}
 \cap [2,t]} \log p = 
 t +  \mathcal{O} \left( t \exp \left( 
 - a \sqrt{\log t} \right) \right)
 = t +  \mathcal{O} \left( t \exp \left( 
 - a \sqrt{\log x} \right) \right)
 ,$$
 by a classical refinement of the prime number theorem due to de la Vall\'ee Poussin ($a > 0 $ is a universal constant). 
 Note that since we assume Riemann hypothesis, we could have used a better estimate of $\rho(t)$, but this is not needed for our computations. 
 Now, we have 
 $$\left|\frac{\partial}{\partial t} 
 I(x,t,\theta)\right|
 \leq \frac{1}{t \log t} \ll \frac{1}{t}$$ 
 and 
$$
 \left|\frac{\partial^2}{\partial t^2} 
 I(x,t,\theta)\right|
  = \left| \frac{- (\theta/t) \sin(\theta \log t)}{t \log t} - \frac{(1 + \log t)\cos(\theta \log t)}{t^2 \log^2 t} \right|
 \ll \frac{(1+ \theta)}{t^2 \log t}.$$
 We then get, by estimating the bracket and the last integral, 
 $$\Sigma(y, \theta) - \Sigma(x,\theta)
  =  I(x,y,\theta) + \mathcal{O} 
  \left( (1+ \theta) (1+ \log \log y - \log \log x) \exp \left( - a \sqrt{\log x} \right)  \right).$$
  For $x = e^{e^{m \log H/K}}$, 
  $y = e^{e^{(m+1) \log H/K}}$, 
  $\theta = 0$ or $\theta = 
  (k-\ell)h/H \in [-h,h]$, the last error term is 
\begin{align*}
& \ll_h \left(1 + \frac{ \log H}{K}
  \right) \exp \left( - a 
 H^{1/2K} \right)
 \\ & \ll ( 1+ \log \log T) 
 \exp ( - a (\log T)^{(1-\delta)/2K} )
 \\ & \ll \exp (- (\log T)^{1/4K} ),
 \end{align*}
 if $T$ is large enough depending on $\delta$ and $K$. 
 Multiplying these error terms by 
 $(\lambda_m^2 + \mu_m^2)/4$ and 
 $\lambda_m \mu_m/2$ respectively, which are, in 
 absolute value,  
 at most $(\log T)^{1/10K}$, and 
 adding them, we get something which is 
 $$ \ll_h  (\log T)^{1/10K} \exp (- (\log T)^{1/4K} ) \ll  \exp (- (\log T)^{1/5K} ) $$
for $T$ large enough depending on 
$\delta$ and $K$. 
Hence, we deduce
 \begin{align*}
 &\mathbb{E} \left[ \exp \left( \lambda_m
 V(k,m) + \mu_m V(\ell,m) \right) \right]
 \\ & = \exp 
 \left(  I (e^{e^{m \log H/K}}, 
 e^{e^{(m+1) \log H/K}}) 
 \left( \frac{\lambda_m^2 + \mu_m^2}{4}
 \right) \dots 
 \right. \\ & \left. \dots  +  I (e^{e^{m \log H/K}}, 
 e^{e^{(m+1) \log H/K}}, (k-\ell)h/H)
 \left( \frac{\lambda_m \mu_m}{2}
 \right) \dots \right. \\ & \left. \dots + \mathcal{O}_h
 \left( \exp \left( - (\log T)^{1/5K} \right) \right) \right).
  \end{align*}
  Using the independence of the variables 
  $V(k,m)$ for different values of $m$ gives the following:
$$\mathbb{E} 
\left[ \exp \left( \sum_{m=1}^{K-1} 
(\lambda_m V(k,m) + \mu_m V(\ell,m)) \right)
\right] = \exp \left( \mathcal{A} 
+ \mathcal{O}_h \left( K \exp 
( - (\log T)^{1/5K} ) \right) \right),$$
  where 
  \begin{align*}
  \mathcal{A} 
  & = \sum_{m=1}^{K-1} 
   I (e^{e^{m \log H/K}}, 
 e^{e^{(m+1) \log H/K}}) 
 \left( \frac{\lambda_m^2 + \mu_m^2}{4}
 \right)
 \\ & + \sum_{m=1}^{K-1}   I (e^{e^{m \log H/K}}, 
 e^{e^{(m+1) \log H/K}}, (k-\ell)h/H)
 \left( \frac{\lambda_m \mu_m}{2}
 \right).
  \end{align*}
  Since $|I(x,y,\theta)| \leq \log \log y 
  - \log \log x$ and $\lambda_m, \mu_m 
  \leq (\log T )^{1/20K}$, 
  we can bound $\mathcal{A}$ by a telescopic sum which gives
  $$|\mathcal{A}| \leq 
    (\log T )^{1/10K} \log H
    \leq  (\log T )^{1/10K} \log \log T,$$
    and then 
   \begin{align*}
  &  \mathbb{E} 
\left[ \exp \left( \sum_{m=1}^{K-1} 
(\lambda_m V(k,m) + \mu_m V(\ell,m)) \right)
\right] = \exp (\mathcal{A}) 
\left( 1 + \mathcal{O}_h \left( K 
\exp (-( \log T)^{1/5K}) \right) \right)
\\ &  = \exp(\mathcal{A})
 + \mathcal{O}_h \left( K \exp \left((\log T)^{1/10K} \log \log T - 
 (\log T)^{1/5K}  \right)\right) 
 \\ & =  \exp(\mathcal{A})
 + \mathcal{O}_h \left( 
\exp (-( \log T)^{1/10K}) \right),
   \end{align*}
    for $T$ large 
  enough depending on $K$ and $\delta$.
  This completes the proof of the proposition, provided that we check that 
   $$\mathbb{E} 
\left[ \exp \left( \sum_{m=1}^{K-1} 
(\lambda_m G(k,m) + \mu_m G(\ell,m)) \right)
\right] = \exp \left( \mathcal{A} \right).$$
By independence of $G(k,m)$ for different values of $m$, it is sufficient to check 
\begin{align*}
&\mathbb{E} 
\left[ \exp \left(  
\lambda_m G(k,m) + \mu_m G(\ell,m)\right) 
\right]
\\ & = \exp \left(  I (e^{e^{m \log H/K}}, 
 e^{e^{(m+1) \log H/K}}) 
 \left( \frac{\lambda_m^2 + \mu_m^2}{4}
 \right)
\dots \right. \\ &  \left. \dots +    I (e^{e^{m \log H/K}}, 
 e^{e^{(m+1) \log H/K}}, (k-\ell)h/H)
 \left( \frac{\lambda_m \mu_m}{2}
 \right) \right).
 \end{align*}
 Since $(G(k,m), G(\ell,m))$ is a centered Gaussian vector, it is sufficient to check that its covariance structure is given by 
 $$\mathbb{E} [(G(k,m))^2]
 =  \mathbb{E} [(G(\ell,m))^2]
 = \frac{1}{2} I (e^{e^{m \log H/K}}, 
 e^{e^{(m+1) \log H/K}}),$$ 
 $$\mathbb{E} [G(k,m) G(\ell,m)] 
 = \frac{1}{2} I (e^{e^{m \log H/K}}, 
 e^{e^{(m+1) \log H/K}},(k-\ell) h/H ).$$
 By including the case $k = \ell$, it is enough to check the last equality, which is proven as follows, using the fact that (formally)
 $\mathbb{E} [(d W_t)^2] = 0$ and 
 $\mathbb{E} [d W_t \overline{d W_t}] = 2 dt$,  
 \begin{align*}
 & \mathbb{E} [G(k,m) G(\ell,m)] 
 \\ & = \mathbb{E} 
 \left[ \left(\Re \int_{e^{e^{m \log H/K}}}^{e^{e^{(m+1) \log H/K}}} \frac{t^{-ikh/H}}{\sqrt{2 t \log t}} 
 d W_t \right) \left(\Re \int_{e^{e^{m \log H/K}}}^{e^{e^{(m+1) \log H/K}}} \frac{t^{-i \ell h/H}}{\sqrt{2 t \log t}} 
 d W_t \right)
  \right] 
  \\  & \frac{1}{4} 
  \left[ \left( \int_{e^{e^{m \log H/K}}}^{e^{e^{(m+1) \log H/K}}} \left(\frac{t^{-ikh/H}}{\sqrt{2 t \log t}} 
 d W_t  + \frac{t^{ikh/H}}{\sqrt{2 t \log t}} 
  \overline{ dW_t} \right) \right)
  \dots \right. \\ & \left. \dots  \times   \left( \int_{e^{e^{m \log H/K}}}^{e^{e^{(m+1) \log H/K}}} \left(\frac{t^{-i\ell h/H}}{\sqrt{2 t \log t}} 
 d W_t  + \frac{t^{i\ell h/H}}{\sqrt{2 t \log t}} 
  \overline{ dW_t} \right) \right)
  \right]   
  \\ & = \frac{1}{4} 
   \int_{e^{e^{m \log H/K}}}^{e^{e^{(m+1) \log H/K}}} \frac{2 t^{-ikh/H} t^{i \ell h/H} + 2 t^{ikh/H} t^{-i \ell h/H}}{2t \log 
   t}   dt
   \\ & = \frac{1}{4} 
    \int_{e^{e^{m \log H/K}}}^{e^{e^{(m+1) \log H/K}}} \frac{ 4 \cos ((k-\ell) h (\log t) /H)}{2 t \log t} dt
    \\ & = \frac{1}{2} I (e^{e^{m \log H/K}}, 
 e^{e^{(m+1) \log H/K}},(k-\ell) h/H ).
 \end{align*}
  \end{proof}
Combining Propositions \ref{TV} and \ref{VG}, we obtain that the Fourier transforms of the variables $S_0(k,m)$ and the Gaussian variables 
$G(k,m)$ are close to each other.
By using Fourier inversion, we deduce the following, which is the main result of the section:  
\begin{proposition} \label{0.14}
Let $\varphi$ be a smooth function with compact support, $k, \ell \in \{0, \dots, H-1\}$, $x \in \mathbb{R}$, and $A > 0$. Then, with the notation of the previous propositions: 
$$\mathbb{E}
\left[ \prod_{m=1}^{K-1} \varphi 
(S_0(k,m) - x)  \right]
 =\mathbb{E}
\left[ \prod_{m=1}^{K-1} \varphi 
(G(k,m) - x)  \right]
 + \mathcal{O}_{\varphi, K, \delta, h, A} 
( (\log T)^{-A})$$
 and 
\begin{align*}
\mathbb{E}
\left[ \prod_{m=1}^{K-1} \varphi 
(S_0(k,m) - x) \varphi 
(S_0(\ell,m) - x) \right]
& =\mathbb{E}
\left[ \prod_{m=1}^{K-1} \varphi 
(G(k,m) - x) \varphi 
(G(\ell,m) - x) \right]
\\ & + \mathcal{O}_{\varphi, K, \delta, h, A} 
( (\log T)^{-A})
\end{align*}
Note that these estimates hold uniformly in $x$. 
\end{proposition}
\begin{proof}
Let us prove the second estimate: the proof of the first 
one is exactly similar. 
Let $\eta$ be the inverse Fourier transform of $\varphi$: from the assumptions of $\varphi$, $\eta$ is bounded and dominated by any negative power at infinity. We have 
$$\varphi 
(S_0(k,m) - x) = \int_{-\infty}^{\infty} 
\eta(\lambda) e^{-i\lambda (S_0(k,m) - x)} d \lambda$$
and then
\begin{align*}
&\prod_{m=1}^{K-1} \varphi 
(S_0(k,m) - x) \varphi 
(S_0(\ell,m) - x)  
 \\ & = \int_{\mathbb{R}^{2K-2} }
 \prod_{m=1}^{K-1} \eta(\lambda_m) \eta(\mu_m) e^{ ix\sum_{m=1}^{K-1} 
( \lambda_m + \mu_m) } 
e^{- i \sum_{m=1}^{K-1} (\lambda_m S_0(k,m)
+ \mu_m S_0(\ell,m))} \prod_{m=1}^{K-1} d \lambda_m d \mu_m, 
\end{align*}
and a similar formula for the variables $G(k,m)$, $G(\ell,m)$. 
Therefore, 
\begin{align*}
&\left|\mathbb{E}
\left[ \prod_{m=1}^{K-1} \varphi 
(S_0(k,m) - x) \varphi 
(S_0(\ell,m) - x) \right]
 - \mathbb{E}
\left[ \prod_{m=1}^{K-1} \varphi 
(G(k,m) - x) \varphi 
(G(\ell,m) - x) \right] \right|
\\ & 
\leq \int_{\mathbb{R}^{2K-2} }
 \prod_{m=1}^{K-1} |\eta(\lambda_m)| |\eta(\mu_m)|  \dots 
 \\ &  \dots \times 
\left| \mathbb{E} [e^{- i \sum_{m=1}^{K-1} (\lambda_m S_0(k,m)
+ \mu_m S_0(\ell,m))}] - 
\mathbb{E} [e^{- i \sum_{m=1}^{K-1} (\lambda_m G(k,m)
+ \mu_m G(\ell,m))}] \right| \prod_{m=1}^{K-1} d \lambda_m d \mu_m.
\end{align*}
If $T$ is large enough depending only on $K$ and $\delta$, we deduce, from the two previous propositions: 
$$\left| \mathbb{E} [e^{- i \sum_{m=1}^{K-1} (\lambda_m S_0(k,m)
+ \mu_m S_0(\ell,m))}] - 
\mathbb{E} [e^{- i \sum_{m=1}^{K-1} (\lambda_m G(k,m)
+ \mu_m G(\ell,m))}] \right| 
\ll_h \exp ( -(\log T)^{\delta/5K} ),$$
if $|\lambda_m|, |\mu_m| \leq (\log T)^{\delta/50K}$ for $m \in \{1, \dots, K-1\}$. 
Since the difference of expectations is 
unconditionally bounded by $2$, we get, for 
$T$ large enough depending on $\delta$ and $K$: 
\begin{align*}
&\left|\mathbb{E}
\left[ \prod_{m=1}^{K-1} \varphi 
(S_0(k,m) - x) \varphi 
(S_0(\ell,m) - x) \right]
 - \mathbb{E}
\left[ \prod_{m=1}^{K-1} \varphi 
(G(k,m) - x) \varphi 
(G(\ell,m) - x) \right] \right|
\\ & \ll_h (||\eta||_{\infty})^{2K-2}
\left( [2(\log T)^{\delta/50K}]^{2K-2} 
\exp ( -(\log T)^{\delta/5K} ) \right)
\\ & + \int_{\mathbb{R}^{2K-2}}
\prod_{m=1}^{K-1} |\eta(\lambda_m)|
|\eta(\mu_m)| \mathds{1}_{\exists m
\in \{1, \dots, K-1\}, \max(|\lambda_m|, 
|\mu_m|) \geq (\log T)^{\delta/50K}} 
 \prod_{m=1}^{K-1} d \lambda_m d \mu_m.
\end{align*}
It is immediate that the first term is 
$\mathcal{O}_{\eta, \delta, K, A}
((\log T)^{-A})$. 
By bounding the indicator of a union by 
the sum of the indicators, we see that the second term is at most: 
$$ (2K-2) (||\eta||_1)^{2K-3}
\int_{\mathbb{R} \backslash (- (\log T)^{\delta/50 K } , (\log T)^{\delta/50 K } ) } |\eta(\lambda)| d \lambda.
$$ 
Since the tail of $\eta$ decays faster than any power, we see that the second term is also $\mathcal{O}_{\eta, \delta, K, A}
((\log T)^{-A})$. 
This proves the proposition for $T$ large enough depending on $\delta$ and $K$: we can then remove this assumption by increasing the implicit constant in 
$\mathcal{O}_{\varphi, K, \delta, h, A}$. 
\end{proof}
\section{The lower bound in the main theorem}
 \label{section:lowerbound}

We start this section by proving an estimate of the expectations involved in Proposition \ref{0.14}. This will gives some 
information on our problem, by using the following remark: if $x > 0$, if $\varphi$ vanishes on $\mathbb{R}_-$ and if 
$$\prod_{m=1}^{K-1} \varphi(S_0(k,m) - x) \neq 0,$$
then $S_0(k,m) > x$ for all $m \in \{1, \dots, K-1\}$, which implies
$$\sum_{m=1}^{K-1} S(k,m) \geq \sum_{m=1}^{K-1} S_0(k,m) \geq (K-1)x.$$
Hence, if
 $$J :=  \sum_{k = 0}^{H-1} \prod_{m=1}^{K-1} \varphi( S_0(k,m) - x),$$
and if we are able to prove that $\mathbb{E}[J^2] =(\mathbb{E}[J])^2 ( 1 + o(1))$, then by using Paley-Zygmund inequality, we deduce that $J > 0$ with probability tending to one when $T$ goes to infinity,  and then
$$\sum_{m=1}^{K-1} S(k,m) \geq (K-1)x$$ 
for at least one value of $k$ between $0$ and $H-1$. After controlling the probability to get large negative values of 
$S(k,0)$ for some $k$, we deduce that $\sum_{m=0}^{K-1} S(k,m) $ is large for some value of $k$ with high probability. This  gives the lower bound we claim if the values of the parameters are suitably chosen. 

The two first moments of $J$ can be written as a sum of expectations involved in Proposition \ref{0.14}. 
In order to estimate these expectations, and compare the first and the second moment of $J$,  we use the fact that  $G(k,m)$ and $G(\ell,m)$ are approximately decorrelated if 
$k$ and $\ell$ have a distance larger than $H^{1-m/K}$.

All the main steps described here are classical, and can for example  be found in the papers by Kistler \cite{bib:Kistler} and by Arguin, Belius and Bourgade \cite{bib:ABB}.

We first consider the right-hand sides involved in Proposition \ref{0.14}, which only depend on Gaussian variables. 
We  can bound the covariance of 
$G(k,m)$ and $G(\ell,m)$ as follows: 
\begin{lemma}
We have, for $k, \ell \in \{0,1,\dots, H-1\}$, $m \in \{1,2, \dots, K-1\}$,
$$\mathbb{E} [G(k,m)^2] =  \frac{\log H}{2K}$$ and 
$$|\mathbb{E} [G(k,m) G(\ell,m)]|
\ll_h \frac{H^{1 - (m/K)}}{|k-\ell|}.$$
\end{lemma}
We recall that $H$ is the integer part of $(\log T)^{1-\delta}$ for some parameter $\delta \in (0,1/2)$ to be fixed later. 
 \begin{proof}
 We have seen that 
 $$\mathbb{E} [G(k,m) G(\ell,m)]
 = \frac{1}{2} I(x,y, \theta),$$
 where
 $$x = e^{e^{m \log H/K}}, 
 y = e^{e^{(m+1)\log H/K}}, 
 \theta = (k-\ell) h/H,$$
 and 
 $$I(x,y,\theta) = 
 \int_{x}^y \frac{\cos (\theta \log t)}{t 
 \log t} dt = \int_{\log x}^{\log y} 
 \frac{\cos(  \theta u)}{u} du.$$
 For $k = \ell$ and then $\theta = 0$, we deduce  
  $$\mathbb{E} [G(k,m)^2] = \frac{1}{2} 
 ( \log \log y 
  - \log \log x )= \frac{\log H}{2K}.$$
  For $k \neq \ell$ and then $\theta \neq 0$, we have
  $$I(x,y,\theta) = \int_{|\theta| \log x}^{|\theta| \log y} \frac{\cos v}{v} dv
  \ll \frac{1}{|\theta| \log x}$$
  by integration by parts, which gives 
  $$|\mathbb{E} [G(k,m) G(\ell,m)]|
  \ll \frac{H}{h|k-\ell|} e^{- m \log H / K}$$
  and the conclusion of the lemma. 
 \end{proof}
 This lemma is used to get the following estimates of the quantities involved in Proposition \ref{0.14}: 
 \begin{proposition} \label{0.16}
 Let us take the notation above, and let us 
 assume that $\varphi$ is not identically zero, smooth, with compact support, nonnegative and equal to zero on $\mathbb{R}_-$. 
 Then, for $k, \ell \in \{0,1, \dots, H-1\}$, $m \in \{1, \dots ,K-1\}$, 
 $\nu \in (0,1)$, $x =  K^{-1} (\log H)\sqrt{1-\nu}$,
 we have
 $$\frac{H^{-(1-\nu)/K}}{\sqrt{\log H}} \ll_{K, \varphi} \mathbb{E} [  \varphi(G(k,m) - x)] 
 \ll_{K, \varphi} \frac{H^{-(1-\nu)/K}}{\sqrt{\log H}} ,$$
 for $|k - \ell| > H^{1-(m/K)}$,
$$\mathbb{E} [ \varphi(G(k,m) - x)
\varphi(G(\ell,m) - x)] 
\ll_{h,K, \varphi} H^{-2(1-\nu)/K}$$
and for $|k - \ell| > H^{1 - (1/2K)}$, 
\begin{align*}
&\mathbb{E} [ \varphi(G(k,m) - x)
\varphi(G(\ell,m) - x)]
\\ & = \mathbb{E} [ \varphi(G(k,m) - x)]
\mathbb{E} [ \varphi(G(\ell,m) - x)]
\left(1 + \mathcal{O}_{h,K, \varphi} \left( H^{-1/2K} \right) \right).
\end{align*}
 \end{proposition}
 \begin{proof}
 We have 
 $$\mathbb{E} [\varphi(G(k,m) - x)] 
 = \frac{1}{\sqrt{2 \pi (\log H)/2K}}
 \int_{0}^{A} e^{- K (x+ t)^2/ (\log H)}
\varphi(t) dt,$$
 if the support of $\varphi$ is included in 
 $[0,A]$. Hence, 
  $$\mathbb{E} [\varphi(G(k,m) - x)] 
 \leq  \frac{A ||\varphi||_{\infty}}{\sqrt{2 \pi (\log H)/2K}}
 e^{- K x^2/ (\log H)}
 \ll_{K,\varphi}  \frac{H^{-(1-\nu)/K}}{\sqrt{\log H}}.
 $$
 On the other hand, if $\varphi$ is 
 larger than $\varepsilon >0$ on an interval 
 $[t_0,t_1] \subset [0,A]$, 
 $$\mathbb{E} [\varphi(G(k,m) - x)] 
 \geq \frac{\varepsilon (t_1 - t_0)}{\sqrt{2 \pi (\log H)/2K}}
 e^{- K (x+t_1)^2/ (\log H)}
 \gg_{K,\varphi}  \frac{H^{-(1-\nu)/K}}
 {\sqrt{\log H}},
 $$
 since 
 $$\frac{K [(x+t_1)^2 - x^2]}{\log H}
 = \frac{K t_1 (2x + t_1)}{\log H}
 \ll_{K,\varphi}  1.
 $$
 For the second estimate, we observe that 
 $$ \mathbb{E} [ \varphi(G(k,m) - x)
\varphi(G(\ell,m) - x)] 
 \leq ||\varphi||_{\infty}^2
\mathbb{P} [G(k,m), G(\ell,m) \in [x,x+A]].$$
Now, we can write the equality in distribution 
$$G(k,m) = \rho \, G(\ell,m) + 
\widetilde{G} \sqrt{1 - \rho^2} ,$$
where $\widetilde{G}$ is independent of $G(\ell,m)$, with the same variance, and
$$\rho := \frac{\mathbb{E}[G(k,m) G(\ell,m)]}{\mathbb{E}[G(\ell,m)]^2}.$$
If $|k-\ell| > H^{1-(m/K)}$, then 
$$|\rho| \ll_h \frac{H^{1-(m/K)}}{|k-\ell| (\log H)/2K} \ll_{h,K} \frac{1}{\log H}$$
We deduce that if $G(k,m)$ and $G(\ell,m)$ are in $[x,x+A]$, and if $H$
is large enough depending on $h$ and $K$,
\begin{align*}
\widetilde{G} & = \frac{1}{\sqrt{1 - \rho^2}} \left( G(k,m) - \rho \, G(\ell,m) \right)
\\ &  = \left(1 - \mathcal{O}_{h,K} \left(\frac{1}{\log^2 H} \right) \right)^{-1/2} 
 \left( x + \mathcal{O}(A)  
 - \mathcal{O}_{h,K} \left( \frac{x+A}{\log H} \right) \right)
 \\ & = x + \mathcal{O}_{h,K, \varphi} (1).
\end{align*}
Hence 
\begin{align*}
& \mathbb{E} [ \varphi(G(k,m) - x)
\varphi(G(\ell,m) - x)] 
\ll_{\varphi} 
\mathbb{P} [G(\ell,m), \widetilde{G}
 \geq x - \mathcal{O}_{h,K, \varphi}(1) ]
 \\  & \leq 
 \exp \left(  - \frac{2K(x - \mathcal{O}_{h,K,\varphi}(1))_+^2}{ \log H} \right)
 \ll_{h,K,\varphi}  \exp \left(  - \frac{2 Kx^2}{ \log H} \right) 
 = H^{-2(1 - \nu)/K},
 \end{align*}
 for $H$ large enough depending on $h$ and $K$. This condition can then be removed by changing the implicit constant of the estimate. 
 
 For the last estimate, under the assumption 
 $|k-\ell| > H^{1- (1/2K)}$, 
 we have, as soon as $C$ is invertible: 
\begin{align*}
& \mathbb{E} [ \varphi(G(k,m) - x)
\varphi(G(\ell,m) - x)]
- \mathbb{E} [ \varphi(G(k,m) - x)]
\mathbb{E} [ \varphi(G(\ell,m) - x)]
\\ & = \frac{1}{2 \pi} \int_{0}^{A} \int_0^A 
\left( \frac{e^{-  \frac{1}{2} (t+x, u+x) C^{-1} (t+x ,u+x)^t}}{ \sqrt{\operatorname{det}(C)}} 
- \frac{e^{-  \frac{1}{2} (t+x, u+x) C_0^{-1} (t+x ,u+x)^t}}{ \sqrt{\operatorname{det}(C_0)}}
\right)\varphi(t) \varphi(u) dt \,  du,
\end{align*}
where $C$ is the covariance matrix of 
$(G(k,m), G(\ell,m))$, and $C_0$ its
diagonal part. 
 Now,  
 $$|\mathbb{E} [G(k,m) G(\ell,m)] |
 \ll_h \frac{H^{1-(m/K)}}{|k-\ell|} 
 \leq \frac{H^{1-(1/K)}}{H^{1 - (1/2K)}} 
 \leq H^{-1/2K},$$
 $$C_0 = \left( \begin{array}{cc}
\frac{\log H}{2K} & 0 \\
0 &  \frac{\log H}{2K}\end{array} \right),
$$ 
$$C =  \left( \begin{array}{cc}
\frac{\log H}{2K} & \mathcal{O}_h(H^{-1/2K}) \\
 \mathcal{O}_h(H^{-1/2K})  &  \frac{\log H}{2K}\end{array} \right),$$
$$\operatorname{\det} (C_0) = \frac{\log^2 H}{4K^2},$$
$$\operatorname{\det} (C) = \frac{\log^2 H}{4K^2} + \mathcal{O}_h( H^{-1/K}).$$
For $H$ large enough depending only on $h$ and $K$, $C$ is invertible and 
\begin{align*}
& C^{-1} 
= \left( \frac{\log^2 H}{4K^2} + \mathcal{O}_h( H^{-1/K}) \right)^{-1} 
\left( \begin{array}{cc}
\frac{\log H}{2K} & \mathcal{O}_h(H^{-1/2K}) \\
 \mathcal{O}_h(H^{-1/2K})  &  \frac{\log H}{2K}\end{array} \right)
 \\ & = \left( 1 + \mathcal{O}_{h,K} (H^{-1/K}
 \log^{-2}H) \right)
 \left( \begin{array}{cc}
\frac{2K}{\log H} & \mathcal{O}_{h,K}(H^{-1/2K} \log^{-2} H ) \\
 \mathcal{O}_{h,K}(H^{-1/2K} \log^{-2} H )  &  \frac{2K}{\log H}\end{array} \right) 
 \\ & = C_0^{-1} + \mathcal{O}_{h,K} (H^{-1/2K}\log^{-2} H).
 \end{align*}
 We then have 
 $$ |(t+x, u+x) ( C^{-1} - C_0^{-1}) (t+x ,u+x)^t| \ll_{h,K, \varphi} 
 (\log H) (H^{-1/2K}\log^{-2} H) (\log H)
 = H^{-1/2K}$$
 whereas 
 $$\sqrt{\frac{\operatorname{det}(C)}
 {\operatorname{det}(C_0)}}
 = 1+ \mathcal{O}_{h,K} (H^{-1/K} \log^{-2} H) ,$$
 which implies 
 $$\frac{e^{-  \frac{1}{2} (t+x, u+x) C^{-1} (t+x ,u+x)^t}}{ \sqrt{\operatorname{det}(C)}} 
=  \frac{e^{-  \frac{1}{2} (t+x, u+x) C_0^{-1} (t+x ,u+x)^t}}{ \sqrt{\operatorname{det}(C_0)}} \left( 
 1+ \mathcal{O}_{h,K,\varphi} (H^{-1/2K})
  \right),$$
 \begin{align*}
& \mathbb{E} [ \varphi(G(k,m) - x)
\varphi(G(\ell,m) - x)]
- \mathbb{E} [ \varphi(G(k,m) - x)]
\mathbb{E} [ \varphi(G(\ell,m) - x)]
\\ & =  \frac{1}{2 \pi} \int_{0}^{A} \int_0^A 
 \frac{e^{-  \frac{1}{2} (t+x, u+x) C_0^{-1} (t+x ,u+x)^t}}{ \sqrt{\operatorname{det}(C_0)}} 
\mathcal{O}_{h,K,\varphi} (H^{-1/2K})
\varphi(t) \varphi(u) dt \,  du
\\ & \ll_{h,K,\varphi} \frac{H^{-1/2K}}{2 \pi} 
\int_{0}^{A} \int_0^A 
 \frac{e^{-  \frac{1}{2} (t+x, u+x) C_0^{-1} (t+x ,u+x)^t}}{ \sqrt{\operatorname{det}(C_0)}} 
\varphi(t) \varphi(u) dt \,  du
\end{align*}
by positivity of the last integrand, which proves the result of the proposition, for $H$ large enough depending on $h$ and $K$. 
 Again, this assumption can be removed by changing the implicit constant in $\mathcal{O}_{h,K, \varphi}$. 
 \end{proof}
 
 By using Propositions \ref{0.14} and \ref{0.16}, we can apply the second moment method to the random variables $S_0(k,m)$: 
 \begin{proposition}
 With the notation and under the assumptions of the previous proposition and lemma, we have, for $T$ large enough 
 depending on $\varphi, K, \delta, h$, 
 and for $\nu \in (0,1/2)$, 
 $$\mathbb{E} [J^2] = \left(\mathbb{E}[J]
 \right)^2 (1+  \mathcal{O}_{\varphi, K, 
 \delta, h} ((\log T)^{-\nu(1-\delta)/K} (\log \log T)^{K-1})),$$
 where 
 $$J :=  \sum_{k = 0}^{H-1} \prod_{m=1}^{K-1} \varphi( S_0(k,m) - x).$$
 \end{proposition}

\begin{proof}
Using Proposition \ref{0.14} and the independence of the Gaussian variables $G(k,m)$, we get 
\begin{align*}\mathbb{E}[J] 
& = \sum_{k=0}^{H-1}  \left(
\prod_{m=1}^{K-1} \mathbb{E} \left[   \varphi( G(k,m) - x)
\right]  + \mathcal{O}_{\varphi,K, \delta, h, A} ((\log T)^{-A}) \right)
\\ & = H \left(\mathbb{E} \left[   \varphi( G(0,1) - x)
\right] \right)^{K-1} 
+ \mathcal{O}_{\varphi, K, \delta, h,B} 
(H^{-B}) 
\end{align*}
for all $B > 0$, since $ H \leq \log T$. 
Hence, by taking $B = 1$ (say), and by using  Proposition \ref{0.16}, 
$$\mathbb{E}[J] \gg_{\varphi, K, \delta, h}  \frac{H^{1 - (1-\nu)(K-1)/K}}{(\log H)^{(K-1)/2}}$$
for $H$, and then $T$,  large enough depending on $\varphi, K, \delta, h$. 
On the other hand, 
\begin{align*}
\mathbb{E}[J^2] 
& = \sum_{0 \leq k, \ell \leq H-1}  \left( 
\prod_{m=1}^{K-1} \mathbb{E} [\varphi(G(k,m) - x) \varphi(G(\ell,m) - x) ] + \mathcal{O}_{\varphi,K, \delta, h, A} ((\log T)^{-A}) \right)
\\ & =  \sum_{0 \leq k, \ell \leq H-1}  
\prod_{m=1}^{K-1} \mathbb{E} [\varphi(G(k,m) - x) \varphi(G(\ell,m) - x) ]  + \mathcal{O}_{\varphi, K, \delta, h,B} (H^{-B}).
\end{align*}
For $H^{1-(r/K)}  < |k-\ell| \leq H^{1-((r-1)/K)}$, $r$ integer between $1$ and $K$, or $|k-\ell| \leq 1$ for $r = K$, we get 
\begin{align*}
\prod_{m=1}^{r-1} 
\mathbb{E} [\varphi(G(k,m) - x) \varphi(G(\ell,m) - x) ] 
& \leq ||\varphi||_{\infty}^{r-1}
\prod_{m=1}^{r-1} 
\mathbb{E} [\varphi(G(k,m) - x)]
\\ & \ll_{K, \varphi} \left(\frac{H^{-(1-\nu)/K}}{\sqrt{\log H}} \right)^{r-1} 
\ll_{K, \varphi} H^{-(r-1)(1-\nu)/K}, 
\end{align*}
and from the fact that 
$|k-\ell| > H^{1-(m/K)}$ for $m \geq r$
(which implies that we do not have $r = K$), 
$$\prod_{m=r}^{K-1} 
\mathbb{E} [\varphi(G(k,m) - x) \varphi(G(\ell,m) - x) ] 
\ll_{h,K, \varphi}  H^{-2(K-r) (1-\nu)/K},
$$
and then 
$$\prod_{m=1}^{K-1} 
\mathbb{E} [\varphi(G(k,m) - x) \varphi(G(\ell,m) - x) ] 
\ll_{h,K, \varphi}  H^{-[r-1 + 2(K-r)] (1-\nu)/K}.
$$
For $2 \leq r \leq K$, the number of couples $(k, \ell)$ with 
$|k-\ell| \leq H^{1 - ((r-1)/K)}$ is 
dominated by $H^{2  - ((r-1)/K)}$, which gives 
\begin{align*}
\sum_{0 \leq k, \ell \leq H-1, 
H^{1 - (r/K)} < |k-\ell| \leq 
H^{1 - ((r-1)/K)}} 
& \prod_{m=1}^{K-1} 
\mathbb{E} [\varphi(G(k,m) - x) \varphi(G(\ell,m) - x) ] 
\\ & \ll_{h,K, \varphi} 
H^{2 - [r-1 + (1 -\nu) (r-1 + 2(K-r))]/K},
\end{align*}
where 
\begin{align*}
& 2 - \frac{r-1 + (1 -\nu) (r-1 + 2(K-r))}{K} 
\\ & = \frac{ 2K + 1 +  (1 - \nu)  - 2K (1 - \nu)}{K} - \frac{r}{K} \left( 
 1 + (1-\nu) - 2 (1 - \nu) \right) 
 \\ & = \frac{2 K \nu + 2 - \nu}{K} 
 - \frac{r \nu}{K}  \leq 2 \nu + \frac{2 -3 \nu}{K},
\end{align*}
since $r \geq 2$. 
Hence, 
\begin{align*}\sum_{0 \leq k, \ell \leq H-1,  |k-\ell| \leq 
H^{1 - (1/K)}} 
& \prod_{m=1}^{K-1} 
\mathbb{E} [\varphi(G(k,m) - x) \varphi(G(\ell,m) - x) ] 
\\ & \ll_{h,K,\varphi} H^{ 2 \nu
+ (2 - 3\nu)/K}.
\end{align*}
The number of couples $(k,\ell)$ with 
$|k-\ell| \leq H^{1 - (1/2K)}$ is 
dominated by $H^{2 - (1/2K)}$, and then, by 
using the case $r = 1$, 
\begin{align*}
\sum_{0 \leq k, \ell \leq H-1, 
H^{1 - (1/K)} < |k-\ell| \leq 
H^{1 - (1/2K)}} 
& \prod_{m=1}^{K-1} 
\mathbb{E} [\varphi(G(k,m) - x) \varphi(G(\ell,m) - x) ] 
\\ & \ll_{h,K, \varphi} 
H^{2 - (1/2K) - (2(1 -\nu) (K-1)/K)},
\end{align*}
with 
\begin{align*}
& 2 - \frac{1}{2K} - \frac{2 (1-\nu) (K-1)}{K} 
\\ & = \frac{2K - (1/2) - 2K(1-\nu) + 2 (1 - \nu)}{K} 
\\ & = \frac{2K \nu+ (3/2) - 2 \nu}{K} 
= 2 \nu  + \frac{(3/2) - 2\nu}{K}
\leq 2 \nu +   \frac{2 - 3\nu}{K},
\end{align*}
since $\nu < 1/2$, 
and then 
\begin{align*}\sum_{0 \leq k, \ell \leq H-1,  |k-\ell| \leq 
H^{1 - (1/2K)}} 
& \prod_{m=1}^{K-1} 
\mathbb{E} [\varphi(G(k,m) - x) \varphi(G(\ell,m) - x) ] 
\\ & \ll_{h,K,\varphi} H^{ 2 \nu
+ (2 - 3\nu)/K}.
\end{align*}
By using the last estimate of Proposition \ref{0.16}, available for $|k-\ell| > H^{1 - (1/2K)}$,  we deduce 
\begin{align*}
\mathbb{E} [J^2] 
& \leq  \sum_{0 \leq k, \ell \leq H-1}  \left( 
\prod_{m=1}^{K-1} \mathbb{E} [\varphi(G(k,m) - x) ] \mathbb{E}[ \varphi(G(\ell,m) - x) ]
\left( 1 + \mathcal{O}_{h,K, \varphi} 
(H^{-1/2K}) \right) \right) 
\\ & + 
\mathcal{O}_{h,K,\varphi} 
(H^{ 2 \nu
+ (2 - 3\nu)/K} )+ \mathcal{O}_{\varphi, K, \delta, h, B} (H^{-B}) 
\\ & = \left( 1 + \mathcal{O}_{h,K, \varphi} 
(H^{-1/2K}) \right) \sum_{0 \leq k, \ell \leq H-1}  \left( 
\prod_{m=1}^{K-1} \mathbb{E} [\varphi(G(k,m) - x) ] \mathbb{E}[ \varphi(G(\ell,m) - x) ] \right)
\\ & + \mathcal{O}_{\varphi, K, \delta, h}
(H^{2 \nu + (2 - 3 \nu)/K})
\\ & =  \left( 1 + \mathcal{O}_{h,K, \varphi} 
(H^{-1/2K}) \right) \left( H( \mathbb{E} [\varphi
(G(0,1) - x) ])^{K-1}  \right)^2 +  \mathcal{O}_{\varphi, K, \delta, h}
(H^{2 \nu + (2 - 3 \nu)/K})
\\ & 
= \left( 1 + \mathcal{O}_{h,K, \varphi} 
(H^{-1/2K}) \right)
\left( \mathbb{E}[J]  + \mathcal{O}_{\varphi,K, \delta, h, B} (H^{-B}) \right)^2 + \mathcal{O}_{\varphi, K, \delta, h}
(H^{2 \nu + (2 - 3 \nu)/K}).
\end{align*}
Expanding the square and using the trivial 
estimate 
$$\mathbb{E} [J] \leq H ||\varphi||_{\infty}^{K-1} \ll_{K, \varphi} H,$$
we get 
$$\left( \mathbb{E}[J]  + \mathcal{O}_{\varphi,K, \delta, h, B} (H^{-B}) \right)^2  =  (\mathbb{E}[J])^2
+  \mathcal{O}_{\varphi,K, \delta, h, B'} (H^{-B'}),$$
and then 
$$\mathbb{E} [J^2] 
 \leq \left( 1 + \mathcal{O}_{h,K, \varphi} 
(H^{-1/2K}) \right)(\mathbb{E}[J])^2 + \mathcal{O}_{\varphi, K, \delta, h}
(H^{2 \nu + (2 - 3 \nu)/K}).$$
Now, for $T$ large enough depending on $\varphi, K, \delta, h$, 
$$(\mathbb{E}[J])^2
\gg_{\varphi, K, \delta, h}
(\log H)^{-(K-1)} H^{2 - (2-2\nu)(K-1)/K},$$
where 
\begin{align*}
2 - \frac{(2 - 2\nu) (K-1)}{K} 
& = \frac{2K - 2K + 2 + 2 K \nu - 2 \nu}{K} 
\\ & = 2 \nu + \frac{2 - 2 \nu}{K},
\end{align*}
and then 
$$\mathbb{E}[J^2] 
\leq (\mathbb{E}[J])^{2} ( 1+ \mathcal{O}_{\varphi, K, \delta, h} (H^{-\nu/K} (\log H)^{K-1})),$$
(recall that $1/2K \geq \nu/K$),
the converse inequality being obvious. 
\end{proof} 
 Using this bound on the second moment and the Paley-Zygmund inequality, we deduce the following: 
 \begin{proposition} \label{notbeginning}
 For all $\omega < (1-\delta) (K-1)/K$, we have, with probability tending to $1$ when 
 $T$ goes to infinity: 
 $$\sup_{k \in \{0, \dots, H-1\}} 
 \Re \left( \kappa \sum_{p \in \mathcal{P} \cap (e^{e^{\log H/K}}, e^H]} p^{-\frac{1}{2} - i (TU - h/2 + kh/H)}  \right)
 \geq \omega \log \log T.$$
 \end{proposition}
 \begin{proof}
 With the notation of the previous proposition, the Paley-Zygmund inequality can be written as follows, using the Cauchy-Schwarz inequality (recall that $J \geq 0$):
$$ \mathbb{P} [J > 0] \mathbb{E} [J^2]
 = (\mathbb{P} [J > 0])^2 
 \mathbb{E} [J^2 | J > 0]
 \geq (\mathbb{P} [J > 0])^2 
 (\mathbb{E} [J | J > 0])^2
 =( \mathbb{E} [J])^2,$$
 $$\mathbb{P}[J > 0] \geq \frac{(\mathbb{E}[J])^2}{ \mathbb{E}[J^2]},$$
 which, by the previous proposition, tends 
 to $1$ when $T$ goes to infinity, 
 for fixed $\varphi, K, \delta, h, \nu$. 
 Hence, with probability tending to $1$, 
 there exists $k \in \{0,1, \dots, H-1\}$ such that for all $m \in \{1, \dots, K-1\}$, $\varphi(S_0(k,m) - x ) > 0$, which implies $S_0(k,m) \geq x$, and then 
 $S(k,m) \geq x$, since $S_0(k,m)$ is a truncation of $S(k,m)$. Summing this inequality for $m \in \{1, \dots, K-1\}$, we deduce 
 \begin{align*}\Re \left( \kappa \sum_{p \in \mathcal{P} \cap (e^{e^{\log H/K}}, e^H]} p^{-\frac{1}{2} - i (TU - h/2 + kh/H)} \right)& \geq (K-1)x
 = (K-1)(\log H) \sqrt{1 -\nu}/K
\\ & \sim_{T \rightarrow \infty} 
(K-1) (1-\delta)( \log \log T) \sqrt{1 -\nu}/K.
 \end{align*}
 If we fix $\nu \in (0,1/2)$ small enough,
 the lower bound we obtain is larger than 
 $\omega \log \log T$ for $T$ large enough. 
 \end{proof}
In the previous proposition, the sum on primes which is involved starts at $p \geq e^{e^{\log H/K}}$. In order to 
solve our main problem, 
it remains to control the 
 sum over the primes smaller than 
 $e^{e^{\log H/K}}$. This sum is not as close to a Gaussian variable as the sums involving larger primes, and then the control we get is not optimal.  This will be sufficient for our purpose, but it may be problematic in we want to study finer asymptotics of the extreme values of $\log \zeta$.

This is given by the following proposition: 
 \begin{proposition} \label{beginning}
 With probability tending to $1$ when 
 $T$ goes to infinity: 
 $$\inf_{k \in \{0, \dots, H-1\}} 
 \Re \left( \kappa \sum_{p \in \mathcal{P} \cap [3, e^{e^{\log H/K}}]} p^{-\frac{1}{2} - i (TU - h/2 + kh/H)} \right)
 \geq - \frac{2}{\sqrt{K}} \log \log T.$$
 \end{proposition}
 \begin{proof}
 By Proposition \ref{TV}, we have, 
 for $|\lambda| \leq (\log T)^{\delta/100}$, $k \in \{0,1, \dots, H-1\}$, 
 $$\mathbb{E} [\exp (\lambda S_0(k,0))] 
 = \mathbb{E} [\exp (\lambda V(k,0))]
 + \mathcal{O} \left( \exp (-(\log T)^{\delta/4}) \right).$$
 On the other hand, 
\begin{align*}
 \mathbb{E} [\exp (\lambda V(k,0))]
 & = \prod_{p \in \mathcal{P} \cap [3, e^{e^{\log H/K}}]} \mathbb{E} [e^{\lambda 
 \Re X_p/\sqrt{p}}]
 \\ & =  \prod_{p \in \mathcal{P} \cap [3, e^{e^{\log H/K}}]} \mathbb{E} [\cosh(\lambda 
 \Re X_p/\sqrt{p})]
 \\ & \leq \prod_{p \in \mathcal{P} \cap [3, e^{e^{\log H/K}}]} \cosh(\lambda/\sqrt{p})
 \\  & \leq \exp \left( \frac{\lambda^2}{2} \sum_{p \in \mathcal{P} \cap [3, e^{e^{\log H/K}}]} \frac{1}{p} \right) 
 \\ & = \exp \left( \frac{\lambda^2}{2} 
 \left( \frac{\log H}{K} + \mathcal{O}(1) \right) \right).
\end{align*}
We deduce, for $\lambda = - 2 \sqrt{K}$, and 
$T$ large enough depending on $K$ and $\delta$ (in order to have $|\lambda| 
\leq (\log T)^{\delta/100}$), 
$$\mathbb{E} [\exp (- 2 \sqrt{K} S_0(k,0))] 
\ll_{K} H^{2},$$
$$\mathbb{P} [S_0(k,0) \leq - 2 K^{-1/2} \log H] \leq H^{-4} \mathbb{E} [\exp (- 2 \sqrt{K} S_0(k,0))] \ll_K H^{-2}.$$
By a union bound, we have with probability tending to $1$, $S_0(k,0) \geq - 2 K^{-1/2} 
\log H$, and then $S(k,0) \geq - 2 K^{-1/2} \log H$ for $T$ large enough depending on $K$ and $\delta$, since the truncation 
$-(\log T)^{\delta/3}$ is strictly below 
$-2 K^{-1/2} \log H$. This gives the desired result. 
 \end{proof}
 We deduce the lower bound part of Theorem \ref{main}:
 \begin{proposition}
Let us assume the Riemann hypothesis. For $h > 0$, $U$ uniform on $[0,1]$, $\epsilon \in (0,1)$, $\kappa \in \{1,-i,i\}$,  we have, with probability tending to $1$ when $T$ goes to infinity: 
 $$\underset{\tau \in [UT-h,UT + h] }{\sup}
\Re \left( \kappa  \log  \zeta \left( \frac{1}{2} + i 
 \tau \right) \right)
 \geq ( 1- \epsilon) \log \log T.$$
 \end{proposition}
 \begin{proof}
 By combining Propositions \ref{comparison}, \ref{notbeginning} and \ref{beginning},  
 we have, with probability tending to $1$, 
$$ \underset{\tau \in [UT-h,UT + h] }{\sup}
\left(\Re \left( \kappa  \log  \zeta \left( \frac{1}{2} + i 
 \tau \right) \right) \right)_+ \geq 
 \left( \omega - \frac{2}{\sqrt{K}} \right)
 \log \log T - \mathcal{O}_h (\sqrt{\log \log T} ),$$
 for all $\omega < (1-\delta) (K-1)/K$.
 If we choose $\delta$ sufficiently small and $K$ sufficiently large, depending on $\epsilon$, we deduce, with probability tending to $1$, 
$$ \underset{\tau \in [UT-h,UT + h] }{\sup}
 \left(\Re \left( \kappa  \log  \zeta \left( \frac{1}{2} + i 
 \tau \right) \right) \right)_+   \geq 
 (1 - (\epsilon/2)) \log \log T 
 - \mathcal{O}_h (\sqrt{\log \log T}) 
 \geq (1 - \epsilon) \log \log T,$$
 for $T$ large enough depending on $h$ and 
 $\epsilon$. Since the lower bound is positive, we can then remove the positive part in $\Re( \kappa \log \zeta)$. 
 
 \end{proof}
\bibliographystyle{halpha}
\bibliography{maxzetabib}

\end{document}